\documentclass[12pt]{amsart}
\usepackage{mathrsfs, amssymb,amsthm, amsfonts, amsmath,extarrows}
\usepackage[all]{xy}
\usepackage{upgreek}
\usepackage{enumerate}
\newtheorem{theorem}[equation]{Theorem}
\newtheorem*{theorem*}{Theorem}
\newtheorem{lemma}[equation]{Lemma}
\newtheorem{corollary}[equation]{Corollary}

\newtheorem{proposition}[equation]{Proposition}

\theoremstyle{definition}
\newtheorem{example}[equation]{Example}
\newtheorem{definition}[equation]{Definition}
\newtheorem*{definition*}{Definition}

\theoremstyle{remark}
\newtheorem{remark}[equation]{Remark}

\makeatletter\@addtoreset{equation}{section}

\makeatother

\usepackage{hyperref}

\def \KK {\mathbb{K}}
\def \LL {\mathbb{L}}
\def \Z {\mathbb{Z}}
\def \Q {\mathbb{Q}}

\def \A {\mathbb{A}}
\newcommand{\PP}{\mathbb{P}}
\newcommand{\G}{\mathbb{G}}
\newcommand{\Gm}{\mathbb{G}_{\mathrm{m}}}
\newcommand{\Ga}{\mathbb{G}_{\mathrm{a}}}
\def \Mat {\mathrm{Mat}}
\def \pr {\mathrm{pr}}

\newcommand{\cT}{\mathcal{T}}
\newcommand{\cS}{\mathcal{S}}
\newcommand{\cX}{\mathcal{X}}
\newcommand{\cO}{{\mathcal O}}

\newcommand{\Br}{\operatorname{Br}}
\newcommand{\Hom}{\operatorname{Hom}}
\newcommand{\Spec}{\operatorname{spec}}
\newcommand{\Aut}{\operatorname{Aut}}
\newcommand{\Iso}{\operatorname{Iso}}
\newcommand{\GL}{\operatorname{GL}}
\newcommand{\SL}{\operatorname{SL}}

\newcommand{\Bir}{\operatorname{Bir}}
\newcommand{\PGL}{\operatorname{PGL}}

\newcommand{\Res}{\operatorname{Res}}

\newcommand{\coker}{\operatorname{coker}}
\newcommand{\Ann}{\operatorname{Ann}}
\newcommand{\Id}{\operatorname{Id}}
\newcommand{\Proj}{\operatorname{Proj}}
\newcommand{\rk}{\operatorname{rk}}
\newcommand{\ad}{\operatorname{ad}}
\newcommand{\Pic}{\operatorname{Pic}}
\newcommand{\Char}{\operatorname{char}}
\newcommand{\Gal}{\operatorname{Gal}}
\newcommand{\rar}[1]{\stackrel{#1}{\longrightarrow}}
\newcommand{\Norm}{\operatorname{Norm}}
\newcommand{\iso}{\buildrel{\sim}\over{\longrightarrow}}
\def \et {{\acute e}t}

\def \ge {\geqslant}
\def \le {\leqslant}

\title{Automorphisms of pointless surfaces}

\author{Constantin Shramov and Vadim Vologodsky}

\address{\emph{Constantin Shramov}
\newline
\textnormal{Steklov Mathematical Institute of RAS,
8 Gubkina street, Moscow 119991, Russia.
}
\newline
\textnormal{National Research University Higher School of Economics, Laboratory of Algebraic Geometry, NRU HSE, 6 Usacheva str., Moscow, 117312, Russia.
}
\newline
\textnormal{\texttt{costya.shramov@gmail.com}}}

\address{\emph{Vadim Vologodsky}
\newline
\textnormal{National Research University Higher School of Economics,  Laboratory of Mirror Symmetry, NRU HSE, 6 Usacheva str., Moscow, 117312, Russia.
}
\newline
\textnormal{\texttt{vologod@gmail.com}}}

\begin{document}

\begin{abstract}
For a geometrically
rational surface $X$ over an arbitrary field of characteristic
different from $2$ and $3$ that contains
all roots of~$1$,
we show that either $X$ is birational to a product of a projective line and a conic,
or the group of birational automorphisms of $X$ has bounded finite subgroups.
As a key step in the proof,
we show boundedness of finite subgroups in any anisotropic reductive
algebraic group over a perfect field that  contains
all roots of~$1$.
Also, we provide applications to Jordan property for
groups of birational automorphisms.
\end{abstract}

\maketitle
\tableofcontents

\section{Introduction}

The group of birational automorphisms $\Bir(X)$ of an algebraic variety $X$ can be rather difficult to understand.
However, in many cases the structure of its finite subgroups is more accessible.
An amazing example of this phenomenon is provided by the following theorem.
We shall say that a field $\KK$ \emph{contains all roots of~$1$},
 if, for every positive integer~$n$, the polynomial $x^n-1$ splits completely in $\KK[x]$.

\begin{theorem}[{\cite[Corollary 4.11]{BandmanZarhin2015a}}]
\label{theorem:Zarhin-conic}
Let $\KK$ be a field of characteristic zero
that contains all roots of~$1$.
Let $C$ be a conic over $\KK$.
Assume that $C$ is not rational, i.e.,
that~\mbox{$C(\KK)=\varnothing$}. Then
every non-trivial element of finite order in $\Aut(C)$
has order $2$,
and every finite subgroup of~\mbox{$\Aut(C)$}
has order at most~$4$.
\end{theorem}

In this paper we study finite subgroups of the group  of birational automorphisms  of geometrically rational surfaces and certain higher-dimensional varieties.
A group $\Gamma$ is said to have \emph{bounded finite subgroups},
if there exists a constant~\mbox{$B=B(\Gamma)$} such that,
for any finite subgroup
$G\subset\Gamma$, one has $|G|\le B$.
If this is not the case, we say that $\Gamma$ \emph{has unbounded finite subgroups}.

One of the important results concerning
boundedness of finite subgroups of birational automorphism groups is the
following theorem proved by Yu.\,Prokhorov and C.\,Shramov.

\begin{theorem}[{\cite[Theorem~1.6]{ProkhorovShramov-3folds}}]
\label{theorem:rational-surface-vs-BFS}
Let $\KK$ be a field of characteristic zero
that contains
all roots of~$1$, and let $X$ be a geometrically rational surface over $\KK$.
Assume that $X$ is not rational but has a $\KK$-point.
Then the group $\Bir(X)$ has bounded finite subgroups.
\end{theorem}

The first main result of the present paper is the following.

\begin{theorem}\label{theorem:main}
Let $\KK$  be a field  that contains
all roots of~$1$.
Assume that either $\Char \KK$ is different from~$2$ and~$3$, or $\KK$ is a perfect field of characterstic~$3$.
Let $X$ be a geometrically rational surface over $\KK$.
Then the group $\Bir(X)$ has bounded finite subgroups if and only if
$X$ is not birational to~\mbox{$\PP^1\times C$}, where $C$ is a conic.
\end{theorem}

According to a theorem of Lang and Nishimura, if a smooth projective variety  $X$ over a field $\KK$ has a $\KK$-point, then the same is true for any  smooth projective variety  $X'$ which is birational to $X$; see for instance~\mbox{\cite[Proposition~IV.6.2]{Kollar-1996-RC}}.
It follows that if a smooth projective variety $X$
is birational to $\PP^1\times Y$, for some smooth projective variety $Y$, then~$X$ has a $\KK$-point if and only if $Y$ has a  $\KK$-point.
Thus, Theorem~\ref{theorem:main} is a generalization of
Theorem~\ref{theorem:rational-surface-vs-BFS}.

In  Example~\ref{example:main-counterexample} and Lemma~\ref{p-center}(ii)  we show that the group $\Bir(X)$ may have unbounded finite subgroups if the assumptions on the characteristic
of $\KK$ in  Theorem~\ref{theorem:main}(ii) are dropped. However, in small characteristics  we can still prove a weaker result stated in  Theorem \ref{theorem:main-char} below.

Given an integer $m>0$,  we say that a group $\Gamma$
\emph{has  $m$-bounded finite subgroups}
if there exists a constant~\mbox{$B=B(\Gamma)$} such that,
for any finite subgroup
$G\subset\Gamma$, one has $|G|'\le B$,  where~$|G|'$ is
the largest factor of $|G|$ which is coprime to $m$.
We say that $\Gamma$ has \emph{$0$-bounded subgroups}
if it has bounded subgroups in the usual sense.

\begin{theorem}\label{theorem:main-char}
Let $\KK$  be a field of characteristic $2$ or $3$ that contains
all roots of~$1$, and let~$X$ be a geometrically rational surface over $\KK$.
Then the group $\Bir(X)$ has $\Char \KK$-bounded finite subgroups
if and only if~$X$ is not birational to~\mbox{$\PP^1\times C$}, where $C$ is a conic.
\end{theorem}

Given  a geometrically rational surface $X$ it could be hard to decide whether it is birational to~\mbox{$\PP^1\times C$}, for some conic  $C$, or not. However, in some cases, one can show
that~$X$ is not birational to~\mbox{$\PP^1\times C$} by analyzing  the kernel of the  pullback
homomorphism~\mbox{$\Br(\KK) \to \Br(X)$} of the Brauer groups. Using this approach we obtain the following
corollary of Theorems~\ref{theorem:main} and~\ref{theorem:main-char}.

\begin{corollary}\label{corollary:SB-Bir}
Let  $\KK$ be a field that contains
all roots of~$1$, and let  $X$ be one of the following
\begin{itemize}
\item  a  Severi--Brauer surface
 over $\KK$ without $\KK$-points;

\item a product  of two non-rational  non-isomorphic conics;

\item a  smooth quadric $X$ in $\PP^3$ with no $\KK$-points  and  $\Pic(X)\cong \Z$.
\end{itemize}
Then the group $\Bir(X)$ has $\Char \KK$-bounded finite subgroups. Moreover, if either
the characteristic of $\KK$ is different from~$2$  and~$3$ or $\KK$ is a perfect field of characterstic~$3$ then
the group $\Bir(X)$  has bounded finite subgroups.
\end{corollary}

Note that a smooth quadric surface with Picard group different from
$\Z$ is  isomorphic to the product  $C\times C$, where $C$ is a conic (see Lemma \ref{lemma:DP8ptnew}(ii)).
The latter is birationally equivalent to $\PP^1 \times C$.

The key step in the proof of Theorem~\ref{theorem:main} is a certain general property of linear algebraic groups, see Theorem \ref{theorem:LAG} below.
To state the result we must  recall a bit of terminology.

A linear algebraic group is said to be \emph{anisotropic}
if it does not contain a subgroup isomorphic to  $\Gm$.
A connected semi-simple algebraic group $\Gamma$ is said to be \emph{simply connected} if every central isogeny $\tilde \Gamma \to \Gamma$, where $\tilde \Gamma$ is
a connected semi-simple group, is  necessarily  an isomorphism. Recall that every connected semi-simple  group $\Gamma$ admits a \emph{universal cover} which is a pair consisting of a connected semi-simple
simply connected group $\tilde \Gamma$ and a central isogeny $\pi\colon \tilde \Gamma \to \Gamma$. The group scheme theoretic kernel of $\pi$ is called the  \emph{algebraic fundamental group}
of $\Gamma$ and denoted by $\pi_1(\Gamma)$. This is a finite group scheme
whose order~$|\pi_1(\Gamma)|$ equals the order of the topological fundamental group of the
connected semi-simple group over $\mathbb{C}$ constructed from the the root datum of $\Gamma_{\bar{\KK}}$.

Let $H$ be a quasi-simple algebraic group over an algebraically closed field
(that is, $H$ has no proper infinite normal closed subgroups).
One defines the set $\mathcal{T}(H)$ of \emph{torsion primes of $H$} to be the empty set  if $H$ has type
$\mathrm{A}_n$ or $\mathrm{C}_n$. If $H$ has type $\mathrm{B}_n$, $\mathrm{D}_n$, or $\mathrm{G}_2$,
we set $\mathcal{T}(H)=\{2\}$; if  $H$ has type $\mathrm{F}_4$, $\mathrm{E}_6$, or $\mathrm{E}_7$, we set
$\mathcal{T}(H)= \{2, 3\}$; if  $H$ has type  $\mathrm{E}_8$, we set $\mathcal{T}(H)= \{2, 3, 5\}$.
Given any connected semi-simple algebraic group $\Gamma$ over a field $\KK$
we say that a prime $p$  is a \emph{torsion prime of $\Gamma$} if  $p$ is a torsion prime for some quasi-simple direct factor of $\tilde \Gamma_{\bar{\KK}}$, where $\tilde \Gamma$
is the universal cover of $\Gamma$.

The second main result of our paper is the following.

\begin{theorem}\label{theorem:LAG}
Let $r$ and $n$ be positive integers.
Then there exists a constant $L=L(r,n)$ with the following property.
Let $\KK$ be a field that contains all roots of~$1$,
and let $\Gamma$ be an anisotropic linear algebraic group
over $\KK$ such that the number of connected components of $\Gamma$ is at most~$r$
and the rank of $\Gamma$ is at most~$n$.
Let $G$ be a finite subgroup of $\Gamma(\KK)$.
The following assertions hold.
\begin{itemize}
\item[(i)]
If $\Gamma$ is reductive and   $\KK$ is perfect, then $|G|\le L$.

\item[(ii)]
Suppose that $\Gamma$ is an arbitrary linear algebraic group. If $\Char\KK>0$, denote by~$|G|'$
the largest factor of $|G|$ which is coprime to
$\Char\KK$; otherwise put~\mbox{$|G|'=|G|$}. Then~\mbox{$|G|'\le L$}.
In particular, the group $\Gamma(\KK)$ has $\Char\KK$-bounded finite subgroups.

\item[(iii)] Assume that $\Gamma$  is  connected,  semi-simple, and  $\Char \KK=p>0$  is not a torsion prime for $\Gamma$.
Write~\mbox{$|\pi_1(\Gamma)| = l p^m$},  for some non-negative integers $m$ and $l$ such that $l$ is coprime to~$p$.
Then $G$
is a semi-direct product $G=H\rtimes S$ of its normal subgroup $H$ whose order is coprime to $p$ and is less than or equal to~$L$,
and an abelian $p$-group $S$ of exponent less than or equal to~$p^m$.
\end{itemize}
\end{theorem}

As an application of Theorem~\ref{theorem:LAG}, 
we obtain the following result.

\begin{theorem}\label{theorem:family}
Let $\KK$ be a perfect field that contains all roots of~$1$,
and let~$\mathcal{S}$ be  a variety over its algebraic closure~$\bar{\KK}$.
Let $\pi\colon \mathcal{X}\to\mathcal{S}$ be a  projective morphism
such that every fiber $\mathcal{X}_s$ of $\pi$ is an integral 
scheme with~\mbox{$h^1(\mathcal{O}_{\mathcal{X}_s})=0$}.
Then there exists a constant~\mbox{$B=B(\mathcal{X})$}
with the following property.
Suppose that $Y$ is a variety over $\KK$ such that $Y_{\bar{\KK}}$ is isomorphic to the fiber
$\mathcal{X}_s$, for some~\mbox{$s\in S(\bar{\KK})$,
and~$Y$ is not birational to~\mbox{$\PP^1\times Y'$},  for any variety~$Y'$
defined over $\KK$.}
Then every  finite subgroup of $\Aut(Y)$ has order at most~$B$.
\end{theorem}

Applying a result of C.\,Birkar on the boundedness of Fano varieties~\cite[Theorem~1.1]{Birkar}, we obtain the following.

\begin{corollary}\label{corollary:Fanos}
Let $\KK$ be a field of characteristic zero that  contains all roots of~$1$.  Then,  for every $\varepsilon>0$ and a  positive integer $n$, there exists a constant
$B(\varepsilon, n)$ with the following property. Suppose that $Y$ is a Fano variety of
dimension~$n$ over $\KK$ with $\varepsilon$-log canonical singularities which is not birational to  $\PP^1\times Y'$,  for any variety~$Y'$. Then the orders of finite subgroups of
$\Aut(Y)$ are bounded by $B(\varepsilon, n)$.
 \end{corollary}

A particular case of Corollary~\ref{corollary:Fanos}
for automorphism groups of flag varieties was established in~\cite{Atilla}.

One can apply  Theorem~\ref{theorem:LAG} to study the automorphism groups of Severi--Brauer varieties.
For a Severi--Brauer variety $X$ associated to a central simple algebra $A$ over a perfect  field $\KK$
that contains all roots of~$1$, Theorem~\ref{theorem:LAG} implies that $\Aut(X)$ has bounded finite
subgroups if and only if $A$ is a division algebra; see Remark~\ref{remark:SB-via-LAG} for details.
The following proposition, which we prove directly, amplifies this observation.

\begin{proposition}\label{proposition:SB}
Let $\KK$ be a field
that contains all roots of~$1$. Let~$X$ be a Severi--Brauer variety of dimension $n-1$ over~$\KK$,
and let $A$ be the corresponding central simple algebra. Assume that the characteristic
$\Char\KK$ of $\KK$ does not divide $n$.
The following assertions hold.
\begin{itemize}
\item[(i)] The group $\Aut(X)$ has bounded finite
subgroups if and only if $A$ is a division algebra; in particular, if $n$ is a prime number,
then $\Aut(X)$ has bounded finite
subgroups if and only if $X(\KK)=\varnothing$, i.e., $X$ is not isomorphic to~$\PP^{n-1}$.

\item[(ii)] Suppose that $A$ is a division algebra.  Let~$g\in\Aut(X)$ be an element of finite order,
and $G\subset\Aut(X)$ be a finite subgroup. Then $g^{n}=1$ and $|G|\le {n}^{2(n-1)}$.

\item[(iii)] Suppose that  $\Char\KK =0$, $n=3$, and $X(\KK)=\varnothing$. Let $G\subset\Aut(X)$ be a finite subgroup.
Then~\mbox{$|G|\le 27$}.
\end{itemize}
\end{proposition}

In particular, if $\KK$ is a perfect field that contains all roots of~$1$, then
Proposition~\ref{proposition:SB} applies to all Severi--Brauer varieties over~$\KK$;
indeed, in this case $\Char\KK$ cannot divide the dimension of a division algebra over~$\KK$,
see Remark~\ref{remark:char-vs-dim-perfect} below. In the case of an arbitrary field whose characteristic divides
the dimension of the division algebra, the structure of finite subgroups of the
corresponding automorphism group is still rather simple,
cf. Theorem~\ref{theorem:LAG}(iii).

\begin{proposition}\label{proposition:SBp}
Let $\KK$ be a field of characteristic $p>0$
that contains all roots of~$1$. Let $A$ be a division algebra over $\KK$ of dimension~$n^2$,
and let $X$ be a corresponding Severi--Brauer variety.
Write~\mbox{$n= n' p^m$} for some non-negative integers $m$ and $n'$ such that $n'$ is coprime to~$p$.
Then every finite subgroup~\mbox{$G\subset\Aut(X)$}
is a semi-direct product $G=H\rtimes S$ of its normal subgroup $H$ whose order is coprime to $p$ and is less than or equal to~${n'}^{2(n-1)}$,
and an abelian $p$-group $S$ of exponent less than or equal to~$p^m$.
\end{proposition}

As another application of  Theorem~\ref{theorem:LAG}, one can prove that the automorphism group of a smooth quadric $Q$ over a field $\KK$ of characteristic different from $ 2$ that contains all roots of~$1$ has bounded finite subgroups if and only if~\mbox{$Q(\KK)=\varnothing$}. In the case when $\KK$ has zero characteristic, this result was proved earlier in \cite[\S4]{BandmanZarhin2015a}. We
find explicit bounds for orders of finite automorphism groups
of quadrics over appropriate fields, thus generalizing the results of
\cite[\S4]{BandmanZarhin2015a} and making them more precise.

\begin{proposition}
\label{proposition:quadricnew}
Let $\KK$ be a field  that contains all roots of $1$. Assume that $\Char \KK \ne 2$ or~$\KK$ is perfect.
Let~\mbox{$n\ge 3$} be an  integer, and
let $Q\subset\PP^{n-1}$ be a smooth quadric hypersurface over~$\KK$.
The following assertions hold.
\begin{itemize}
\item[(i)] The group $\Aut(Q)$ has bounded finite
subgroups if and only if~\mbox{$Q(\KK)=\varnothing$}.

\item[(ii)] If $n$ is odd  and~\mbox{$Q(\KK)=\varnothing$}, then every non-trivial element
of finite order in the group $\Aut(Q)$ has order $2$, and
every finite subgroup of $\Aut(Q)$ has order
at most~$2^{n-1}$.

\item[(iii)] If $n$ is even and~\mbox{$Q(\KK)=\varnothing$}, then every
non-trivial element of finite order in the group $\Aut(Q)$
has order $2$ or $4$, and
every finite subgroup of $\Aut(Q)$ has order
at most~$8^{n-1}$.

\item[(iv)]
If $n=4$ and~\mbox{$Q(\KK)=\varnothing$}, then
every finite subgroup of $\Aut(Q)$ has order
at most~$32$.
\end{itemize}
\end{proposition}

In particular, we see from Proposition~\ref{proposition:SB}(ii) or Proposition~\ref{proposition:quadricnew}(ii)
that Theorem~\ref{theorem:Zarhin-conic} actually holds
over any perfect field, and also over any field of characteristic different from~$2$ (we will use this in the proof of our Theorem~\ref{theorem:main}). Note that over a non-perfect field
of characteristic~$2$ this is not the case: for every conic over  any infinite field
of characteristic~$2$, its automorphism group has unbounded finite subgroups;
see Lemma~\ref{p-center}(ii) below for details.
Note also that one cannot drop the assumption on the existence of  roots of~$1$ in
Theorem~\ref{theorem:Zarhin-conic}, as well as in Propositions~\ref{proposition:SB} and~\ref{proposition:quadricnew}.
Indeed, the conic over the field of real numbers defined by the equation~\mbox{$x^2+y^2+z^2=0$}
has automorphisms of arbitrary finite order.

Next, we provide explicit bounds for orders of finite subgroups in the
automorphism groups of smooth del Pezzo surfaces.

\begin{proposition}\label{proposition:dP}
Let $\KK$ be a field that
contains all roots of~$1$. Let~$X$ be a smooth
del Pezzo surface over $\KK$ with $K_X^2\ge 6$ such that
$X$ is not birational to $\PP^1\times C$, where $C$ is a conic.
Let $G$ be a finite subgroup of $\Aut(X)$.
The following assertions hold.
\begin{itemize}
\item[(i)] If either $\Char \KK$ is different from $2$ and $3$,
or~$\KK$ is perfect, then $|G|\le 432$.

\item[(ii)] If $\Char \KK$  is equal to $2$ or $3$, then
$|G|' \le 432$, where  $|G|'$ is the largest factor of $|G|$ coprime to $\Char \KK$.
\end{itemize}
\end{proposition}

We point out that if $\Char \KK$  is equal to $2$ or $3$, the bound provided by
Proposition~\ref{proposition:dP} can be improved using a more accurate case-by-case study, see
Remark~\ref{remark:dP-char-2-3}. Also,
we remind the reader that for a smooth del Pezzo surface $X$
of degree $K_X^2\le 5$ over an arbitrary field the group $\Aut(X)$
is finite, and there is are explicit upper bounds for  its order;
see Remark~\ref{remark:dP-low-degree} for details.

Similarly to~\cite{BandmanZarhin2015a}, where Theorem~\ref{theorem:Zarhin-conic}
was applied to study
groups of birational selfmaps of conic bundles over
non-uniruled varieties, we will
apply Theorem~\ref{theorem:main} to higher dimensional varieties
whose maximal rationally connected
fibration has relative dimension~$2$,
see Proposition~\ref{proposition:Jordan} below.

\smallskip
The plan of our paper is as follows.
In~\S\ref{section:tori} we prove Theorem \ref{theorem:LAG} in the case when $\Gamma$ is a torus. The proof is based on the  Minkowski theorem on finite subgroups of $\GL_n(\Z)$ and elementary Galois theory.

In~\S\ref{section:lin-alg-groups} we study finite subgroups of linear algebraic groups
and prove Theorem~\ref{theorem:LAG}. The idea  of the proof is the following. According to
a result of Borel and Tits, for every connected anisotropic reductive group $\Gamma$ over a perfect field~$\KK$, every element~\mbox{$g\in \Gamma(\KK)$} is contained in $T(\KK)$, for some torus $T\subset \Gamma$. Using the results of~\S\ref{section:tori} we bound the order of $g$. On the other hand, choosing a faithful representation of $\Gamma$
we get an embedding~\mbox{$\Gamma(\KK) \subset \GL_N(\KK)$} for some positive integer $N$. This, together with a Burnside type result due to~\cite{HerzogPraeger} (see Theorem~\ref{theorem:Burnside} below), proves that $\Gamma(\KK)$ has bounded finite subgroups.

In~\S\ref{section:SB} we describe automorphism groups of Severi--Brauer varieties and prove
Propositions~\ref{proposition:SB} and~\ref{proposition:SBp}.
In~\S\ref{section:quadrics} we prove
assertions~\mbox{(i)--(iii)} of
Proposition~\ref{proposition:quadricnew}.
In~\S\ref{section:zero-irregularity} we discuss automorphism groups
of varieties whose first cohomology group
of the structure sheaf is trivial, 
and prove Theorem~\ref{theorem:family}.
In~\S\ref{section:DP} we study groups acting on del Pezzo surfaces, prove
Proposition~\ref{proposition:dP},  and complete the proof
of Proposition~\ref{proposition:quadricnew}.
In~\S\ref{section:conic-bundles} we study groups acting on conic bundles. The results of this section are summarized in Proposition  \ref{proposition:conicsummary} that plays a key role
in the proof of  Theorem~\ref{theorem:main}.

In~\S\ref{section:proof} we prove Theorems~\ref{theorem:main} and~\ref{theorem:main-char}
and Corollary~\ref{corollary:SB-Bir};
we also provide a counterexample to Theorem~\ref{theorem:main} over a perfect field of characteristic~$2$.
The
strategy of our  proof of Theorem~\ref{theorem:main} is the following. For every finite subgroup~\mbox{$G\subset \Bir(X)$}, there exists  a smooth projective surface $Y$ birational to~$X$
such that the induced action of $G$ on $Y$ is regular, that is, one has $G\subset \Aut(Y)$. We may assume that $Y$ is  $G$-minimal.  According
to a well known classification result (see~\cite[Theorem~1G]{Iskovskikh80}),
the surface $Y$ is either a del Pezzo surface or a $G$-conic bundle.
In the former case the order of $G$ is bounded by
universal constants from Proposition~\ref{proposition:dP} and  Remark~\ref{remark:dP-low-degree}. We remark that if $\KK$ is perfect then instead of detailed analysis of the automorphism groups of  del Pezzo surfaces
from Proposition~\ref{proposition:dP} and  Remark~\ref{remark:dP-low-degree}
we could refer to Theorem \ref{theorem:family}: it is well known that smooth del Pezzo surfaces over $\bar{\KK}$ form finitely many families. This
approach does not work for conic bundles. One reason for this is that
we do not have  an {\it a priory}  bound for the Picard number of~$Y_{\bar{\KK}}$, where $Y$ is a
conic bundle appearing as a $G$-minimal model of~$X$. The trick here that we borrow from~\cite{ProkhorovShramov-3folds}
is based on Proposition~\ref{proposition:conicsummary} and a deep result of  Iskovskikh, see Theorem~\ref{theorem:conic-bundles-negative-degree}.

In~\S\ref{section:Jordan} we derive some consequences
of Theorem~\ref{theorem:main} for birational automorphism groups
of higher dimensional varieties, and prove
Jordan property for some of these groups.
Finally, in Appendix~\ref{section:appendix} we collect several
well known facts about quotients of algebraic varieties by
the groups~$\Gm$ and~$\Ga$.

\smallskip
Throughout the paper by $\bar{\KK}$ we denote an algebraic closure
of a field~$\KK$, and by $\KK^{sep}$ we denote a separable closure of~$\KK$
(recall that $\KK^{sep}=\bar{\KK}$ provided that $\KK$ is perfect;
we prefer to use the notation $\bar{\KK}$
in this case).
Given a variety $X$ or a morphism $\phi$ defined over~$\KK$,
for an arbitrary field extension~\mbox{$K\supset\KK$} we denote
by~$X_{K}$ and~$\phi_{K}$
the corresponding scalar extensions to~$K$, and by $X(K)$ we denote the set of $K$-points of~$X$.
Abusing notation a bit, we write $\PP^n$ for a projective space over a field $\KK$,
and similarly write $\Gm$ and $\Ga$ for the multiplicative and additive groups, respectively.

\smallskip
We are grateful to I.\,Cheltsov,
J.-L.\,Colliot-Th\'el\`ene, A.\,Duncan,
S.\,Gorchinsky, A.\,Kuznetsov, Ch.\,Liedtke, V.\,Popov, Yu.\,Prokhorov,
D.\,Timashev, A.\,Trepalin, and B.\,Zavyalov for useful discussions.
Constantin Shramov was partially supported by the Russian Academic Excellence Project ``5-100'',
by Young Russian Mathematics award, and by the Foundation for the
Advancement of Theoretical Physics and Mathematics ``BASIS''.
Vadim Vologodsky was partially supported
by the  Laboratory of Mirror Symmetry NRU~HSE, RF government grant,
ag. N\textsuperscript{\underline{o}}~$14.641.31.0001$.

\section{Tori}
\label{section:tori}

In this section we  study elements of finite order
in algebraic tori.

The following result is a famous theorem of Minkowski, see~\cite{Minkowski-1887}
or~\cite[Theorem~1]{Serre-07}.

\begin{theorem}
\label{theorem:Minkowski}
For any positive integer $n$, the group $\GL_n(\mathbb{Z})$ has bounded finite subgroups.
\end{theorem}

Theorem~\ref{theorem:Minkowski} tells us that there
is a constant $\Upsilon(n)<\infty$ equal to the
maximal order of a finite subgroup
in $\GL_n(\mathbb{Z})$.

\begin{remark}\label{remark:Upsilon-1-2-3}
One can deduce Theorem~\ref{theorem:Minkowski} from the following assertion:
for every integer~\mbox{$m>2$}, the kernel of the reduction
homomorphism~\mbox{$\GL_n(\mathbb{Z}) \to  \GL_n(\mathbb{Z}/m \mathbb{Z})$}
is torsion free.
In particular, one has~\mbox{$\Upsilon(n)\le |\GL_n(\mathbb{Z}/3\mathbb{Z})|$}.
However, in general this bound is not sharp. However, for small $n$ one can find the precise value of the
constant~\mbox{$\Upsilon(n)$}.
For instance, we have $\Upsilon(1)=2$, $\Upsilon(2)=12$,
and $\Upsilon(3)=48$, see e.g.~\mbox{\cite[\S1.1]{Serre-07}}
or~\mbox{\cite[\S1]{Tahara}}.
\end{remark}

The following is the main technical result of this section.

\begin{lemma}\label{lemma:tori-bounded-subgroups-general-case}
Let $n$ and $d>\Upsilon(n)$ be positive integers.
Let $\KK$ be a field such that the characteristic of
$\KK$ does not divide $d$, and $\KK$
contains a primitive $d$-th root of $1$.
Let~$T$ be an $n$-dimensional algebraic
torus over $\KK$ such that $T(\KK)$ contains a point of order~$d$.
Then~$T$
contains a subtorus isomorphic to $\Gm$.
\end{lemma}
\begin{proof}
Let  $\check T = \mathrm{Hom} (\Gm, T_{\KK^{sep}})$ be the lattice of cocharacters of $T$.
Recall (see~\mbox{\cite[\S8.12]{Borel}}) that the functor $T \mapsto \check T$ induces an equivalence between the category of algebraic tori over $\KK$ and the category of free abelian groups of finite rank equipped with
an action of the Galois group $\Gal(\KK^{sep}/\KK)$ such that the image of
the homomorphism~\mbox{$\Gal(\KK^{sep}/\KK) \to \Aut(\check T)$} is finite. Denote this image by~$\Gamma$.

The group of $d$-torsion elements of $T(\KK^{sep})$ is isomorphic,
as a Galois module, to~\mbox{$\check T \otimes \mu_d$}, where $\mu_d$  is the group of $d$-th roots of unity in
$\KK^{sep}$.  Since $\KK$ contains a primitive $d$-th root of $1$,
the Galois module $\mu_d$ is the trivial module $\mathbb{Z}/d\mathbb{Z}$, so that the Galois
module~\mbox{$\check T \otimes \mu_d$}  is isomorphic to $\check T/d\check T$. Hence, a point $x\in T(\KK)$
of order $d$ can be viewed as a
$\Gal(\KK^{sep} /\KK)$-invariant element $\bar v \in \check T/d\check T $ of order $d$ (so that $m\bar{v}\neq 0$
for~\mbox{$m<d$}).
Let~\mbox{$v\in  \check T$} be any preimage of $\bar v$ under the projection $ \check T \to  \check T/d\check T$, and let
$$
w= \sum_{\gamma \in \Gamma} \gamma(v).
$$
Since  $\bar v$ is  $\Gal(\KK^{sep} /\KK)$-invariant, the image of
$w$ in $\check T/d\check T$ is equal to $|\Gamma|\bar v$. By assumption, the latter is not equal to $0$. Hence $w\ne 0$.
On the other hand, it is clear that $w$ is a $\Gal(\KK^{sep} /\KK)$-invariant element of $\check T$.
By the above mentioned equivalence of categories, $w$ gives rise to
a non-zero homomorphism  $\Gm \to T$ whose image is the required subtorus.
\end{proof}

\begin{remark}
J.-L.\,Colliot-Th\'el\`ene pointed out to us that the proof of Lemma~\ref{lemma:tori-bounded-subgroups-general-case}
can be reformulated in the following way.
The short exact sequence of $\Gamma$-modules
$$ 0 \longrightarrow \check T \stackrel{d}{\longrightarrow}\check T \longrightarrow \check T/d\check T \longrightarrow 0$$
gives rise to the long exact sequence of cohomology groups
\begin{equation}\label{longex}
\ldots\to  H^0(\Gamma, \check T) \to
H^0(\Gamma, \check T/d\check T) \to  H^1(\Gamma, \check T) \to \ldots
\end{equation}
If $H^0(\Gamma, \check T)=0$ then the second map in~\eqref{longex} is injective.
On the other hand, the group~\mbox{$H^1(\Gamma, \check T)$} is annihilated
by $|\Gamma|$ (see e.g.~\cite[Proposition~IV.6.3]{CF67}).
It follows that the group  $H^0(\Gamma, \check T/d\check T)$  of $d$-torsion points of $T(\KK)$  is also  annihilated by~$|\Gamma|$.
\end{remark}

Lemma~\ref{lemma:tori-bounded-subgroups-general-case} immediately implies the following result.

\begin{corollary}\label{corollary:tori}
Let $\KK$ be a field
that contains all roots of~$1$, and let $T$ be an anisotropic $n$-dimensional
torus over $\KK$.
Then every element of finite order in~$T(\KK)$ has
order at most~$\Upsilon(n)$, and every finite subgroup in $T(\KK)$ has
order at most~$\Upsilon(n)^n$. Moreover, every element of finite order and every finite subgroup of
$T(\KK)$ has order coprime to the characteristic of~$\KK$.
\end{corollary}

\begin{proof}
Note that if $\Char\KK=p$ is positive, then $T$ does not contain elements of order $p$,
because there are no such elements
even in $T_{\KK^{sep}}\cong\Gm^n$.
Thus, if there is an element of finite order $d$ in $T(\KK)$, then $d$ is coprime to~$p$,
so that $d\le\Upsilon(n)$ by Lemma~\ref{lemma:tori-bounded-subgroups-general-case}.
It remains to notice that every finite subgroup of $T(\KK)$ is an abelian
group generated by at most $n$ elements, and thus has order at most~$\Upsilon(n)^n$.
\end{proof}

Remark~\ref{remark:Upsilon-1-2-3}
and Corollary~\ref{corollary:tori} immediately imply the following (well-known)
assertion.

\begin{lemma}
\label{lemma:P1-without-2-pts}
Let $\KK$ be a field
that contains all roots of~$1$, and let $T$ be a one-dimensional torus over $\KK$ that is
different from $\Gm$.
Then every non-trivial finite subgroup of $T(\KK)$ has order~$2$.
\end{lemma}

Lemma~\ref{lemma:P1-without-2-pts} implies the following assertion that is
in some sense analogous to Theorem~\ref{theorem:Zarhin-conic}.

\begin{corollary}[{cf.~\cite[Lemma~3.3]{BandmanZarhin2017}}]
\label{corollary:P1-without-2-pts}
Let $\KK$ be a field
that contains all roots of~$1$.
Let~\mbox{$O_1, O_2\in\PP^1(\KK^{sep})$}
be two distinct points
that form a $\Gal(\KK^{sep}/\KK)$-orbit (so that~\mbox{$\{O_1,O_2\}$} is a closed point of
$\PP^1$ whose residue field~$\KK_{O_1,O_2}$ is a separable quadratic extension of~$\KK$).
Consider the affine curve
$$
U=\PP^1 - \{O_1, O_2\}
$$
defined over $\KK$.
Let $G$ be a finite group acting
by automorphisms of $U$. Then every non-trivial element of $G$ has order $2$, and $|G|\le 4$.
\end{corollary}

\begin{proof}
Let $T$ be the one-dimensional
torus that corresponds to the quadratic field extension $\KK\subset\KK_{O_1,O_2}$. Then $T(\KK)$ is isomorphic to  the group
of all automorphisms $\phi$ of $\PP^1$ such that $\phi_{\KK^{sep}}$ preserves
both points~$O_1$ and~$O_2$.
Therefore, we have~\mbox{$\Aut(U)\cong T(\KK) \rtimes \Z/2\Z$},
where $\Z/2\Z$ acts on  $T(\KK)$ carrying every element of $ T(\KK)$ to its inverse.
Let~\mbox{$g\in   T(\KK) \rtimes \Z/2\Z$} be an element of finite order. If $g \in  T(\KK)$, then~\mbox{$g^2=1$} by  Lemma~\ref{lemma:P1-without-2-pts}. Otherwise,
write $g=t \sigma$, where $t\in  T(\KK) $ and $\sigma$ is the non-trivial element of  $ \Z/2\Z $. We have that
$$
g^2 = t\sigma t \sigma=tt^{-1}=1.
$$
Finally, by Lemma~\ref{lemma:P1-without-2-pts}, we have~\mbox{$|G\cap T(\KK)|\le 2$}, so~\mbox{$|G|\le 4$}.
\end{proof}

Now we will provide more precise bounds for the possible finite orders of elements and subgroups
of two-dimensional tori that will be used in~\S\ref{section:DP}.
Note that there is an explicit classification
of two-dimensional tori (see~\cite{Voskr65}), but we will not use it.
Our approach here is similar to the proof of
Lemma~\ref{lemma:tori-bounded-subgroups-general-case} but instead of the bounds
on the orders of finite subgroups of $\GL_n(\mathbb{Z})$ we
use the bounds on the orders of \emph{elements} of $\GL_2(\mathbb{Z})$
together with the following simple observation that is specific to rank~$2$.

\begin{lemma}\label{lemma:GL2Z-ord-2}
Put $\Lambda=\Z^2$, and let $\Gamma\subset\GL(\Lambda)$ be a finite subgroup.
Suppose that every element of $\Gamma$
has a non-zero invariant vector in $\Lambda$. Then $\Gamma$ has a non-zero invariant vector in $\Lambda$.
\end{lemma}
\begin{proof}
Let $\gamma$ be an element of $\Gamma$. Since $\gamma\in\GL(\Lambda)\cong\GL_2(\Z)$,
its determinant is $\pm 1$. Since one of the eigen-values of $\gamma$ equals $1$,
the other eigen-value equals $\pm 1$.
Since $\gamma$ is of finite order, this implies that $\gamma^2=1$.
Hence the group $\Gamma$ is abelian,
and thus there are two non-proportional vectors
$v_1$ and $v_2$ in $\Lambda$
that generate one-dimensional $\Gamma$-invariant sublattices.
If there exist two non-trivial elements $\gamma_1$
and $\gamma_2$ in $\Gamma$ such that $\gamma_iv_i=-v_i$, then
it is easy to see that there is an element $\gamma'$ in $\Gamma$
such that $\gamma'v_i=-v_i$ for both $i=1,2$. This
means that $\gamma'$ acts on $\Lambda$ by multiplication by $-1$, which
is impossible because
$\gamma'$ must have a non-zero invariant vector by assumption.
Therefore, one of the vectors $v_1$ and $v_2$ is invariant with respect
to the whole group~$\Gamma$.
\end{proof}

\begin{lemma}\label{lemma:GL2Z}
Put $\Lambda=\Z^2$, and let $\Gamma\subset\GL(\Lambda)$ be a finite subgroup.
Let $d>6$ be an integer.
Suppose that there is a $\Gamma$-invariant vector
$v$ in $\Lambda\pmod d$ such that $mv\neq 0$ for~\mbox{$m<d$}.
Then there is a non-zero $\Gamma$-invariant vector
in $\Lambda$.
\end{lemma}
\begin{proof}
Let $\gamma$ be an element of the group $\Gamma$,
and $r$ be the order of $\gamma$.
It is easy to check that $r\in\{1,2,3,4,6\}$, see for instance
\cite[\S1]{Tahara}.

Suppose that $\gamma$ has no non-zero invariant vectors in $\Lambda$.
Then the matrix $\gamma-1$ is invertible in $\GL_2(\Q)$.
Since we know that $\gamma^r-1=0$, this implies that the matrix
$$
\delta=\gamma^{r-1}+\ldots+\gamma+1\in\Mat_2(\Z)\subset\Mat_2(\Q)
$$
is the zero matrix, where $\Mat_2$ denotes the ring of $2\times 2$-matrices.
Thus
$$
\delta_d=\delta\pmod d
$$
is a zero matrix as well.
Write
$$
0=\delta_d v= \left(\gamma^{r-1}+\ldots+\gamma+1\right) v \pmod d=rv,
$$
which is a contradiction since $r\le 6<d$.

Therefore, we see that every element of $\Gamma$ has an invariant vector in $\Lambda$.
Now it remains to apply Lemma~\ref{lemma:GL2Z-ord-2}.
\end{proof}

\begin{lemma}\label{lemma:tori-bounded-subgroups}
Let $\KK$ be a field that contains all roots of~$1$,
and let~$T$ be an anisotropic two-dimensional torus over $\KK$.
Then every element of finite order in~$T(\KK)$
has order at most~$6$, and every finite subgroup
of~$T(\KK)$ has order at most~$36$.
\end{lemma}
\begin{proof}
As in the proof of Lemma~\ref{lemma:tori-bounded-subgroups-general-case},
we put
$\check{T} = \mathrm{Hom}(\Gm, T_{\KK^{sep}})$ 
and let $\Gamma$ be the image of the Galois group
$\Gal(\bar{\KK}/\KK)$
in $\Aut(\check{T})$.
Then $\Gamma$ is a finite
subgroup of $\Aut(\check{T})\cong\GL_2(\Z)$.

Suppose that $T(\KK)$ has an element of finite order $d>1$.
We note that $d$ is coprime to the characteristic
of the field~$\KK$. As in the proof of Lemma~\ref{lemma:tori-bounded-subgroups-general-case},
we see that $\check{T}/d\check{T}$
has a $\Gamma$-invariant
vector $v$ such that $mv\neq 0$ for~\mbox{$m<d$}.
By Lemma~\ref{lemma:GL2Z}, if~\mbox{$d>6$}, then there
is a non-zero $\Gamma$-invariant vector in~$\check{T}$.
Hence there is an embedding~\mbox{$\Gm\hookrightarrow T$}. Since the latter
is not the case by assumption, we conclude that $d\le 6$.
As in the proof of Corollary~\ref{corollary:tori}, we observe that
every finite subgroup of $T$ is an abelian
group generated by at most two elements, and thus
has order at most~$36$.
\end{proof}

\begin{remark}
The proofs of Lemmas~\ref{lemma:GL2Z}
and~\ref{lemma:tori-bounded-subgroups}
also imply that an anisotropic two-dimensional torus over $\KK$ does not
have elements of order~$5$.
\end{remark}

\section{Linear algebraic groups}
\label{section:lin-alg-groups}

In this section we study finite subgroups
of  linear algebraic groups and prove
Theorem~\ref{theorem:LAG}.

Recall that a linear algebraic group $\Gamma$ over a field $\KK$ is a smooth closed subgroup scheme
of $\GL_N$ over $\KK$. In particular, the group $\Gamma(\KK)$ of its $\KK$-points
has a faithful finite-dimensional representation in a $\KK$-vector space.
We refer the reader to~\cite{Borel} and~\cite{Springer}
for the basics of the theory of linear algebraic groups.

Similarly to the case of an algebraically closed field,
many properties of linear algebraic groups are determined
by their maximal tori. Note that in general a linear algebraic group
$\Gamma$ over a non-algebraically closed field $\KK$
may contain non-isomorphic maximal tori,
but their dimension still equals the dimension
of maximal tori in $\Gamma_{\bar{\KK}}$, see~\mbox{\cite[Theorem~13.3.6(i)]{Springer}}
and~\cite[Remark~13.3.7]{Springer}; this dimension is called the \emph{rank} of~$\Gamma$.

Recall that an element $\gamma \in \Gamma(\KK)$  is called semi-simple if  its image  in $\GL_N(\KK)$ is diagonalizable over an algebraic closure $\bar{\KK}$  of $\KK$. The notion of a
semi-simple element is intrinsic, that is,
it does not depend on the choice of $N$ and an
embedding~\mbox{$\Gamma \hookrightarrow \GL_N(\KK)$},  see~\mbox{\cite[\S2.4]{Springer}}.
The main tool that will allow
us to apply the results of~\S\ref{section:tori}
is the following theorem.

\begin{theorem}[{see~\cite[Corollary~13.3.8(i)]{Springer}}]
\label{theorem:Springerold}
Let $\Gamma$ be a
connected linear algebraic group over a field $\KK$, and
let $\gamma\in\Gamma(\KK)$ be a  semi-simple  element. Then
there exists a torus $T\subset\Gamma$ such that
$\gamma$ is contained in~$T(\KK)$.
\end{theorem}
\begin{corollary}\label{cor:Springerold}
Let ~$\Gamma$ be a
connected linear algebraic group over a field $\KK$, and
let~\mbox{$\gamma\in\Gamma(\KK)$} be a finite order  element whose order is coprime to the characteristic of~$\KK$.  Then
there exists a torus $T\subset\Gamma$ such that
$\gamma$ is contained in~$T(\KK)$.
\end{corollary}

For anisotropic reductive  groups over perfect  fields  and for simply connected  semi-simple  anisotropic  groups over arbitrary fields
whose characteristic is large enough, one has a stronger result.

\begin{theorem}[{see \cite[Corollary~3.8]{BorelTits}} and  {\cite[Corollary 2.6]{Tits86}}]
\label{theorem:Springer}
Let~$\Gamma$ be a
connected anisotropic reductive linear algebraic group over $\KK$.
Assume,  in addition,   that either $\KK$ is perfect, or $\Gamma$ is  semi-simple,  simply connected, and $\Char \KK = p>0$ is not a torsion prime for $\Gamma$.
Then, for every element $\gamma\in\Gamma(\KK)$,
there exists a torus $T\subset\Gamma$ such that
$\gamma$ is contained in~$T(\KK)$.
\end{theorem}

\begin{corollary}\label{cor:Springer}
Under the assumptions of Theorem~\ref{theorem:Springer} the order of every finite order element of~$\Gamma(\KK)$
is coprime to the characteristic of $\KK$.
\end{corollary}
\begin{proof}
By Theorem~\ref{theorem:Springer}  it suffices to prove the assertion in the case when $\Gamma$ is a torus,
in which case it is given by Corollary~\ref{corollary:tori}.
\end{proof}

Note that over fields of positive characteristic non-reductive algebraic groups
may have unbounded finite subgroups. For instance, the $p$-torsion subgroup of $\Ga$ over an infinite
field of characteristic $p$
is an infinite-dimensional vector space over the field~$\mathbf{F}_p$ of~$p$ elements.
However, this example is in a certain sense the only source of unboundedness
for unipotent groups.

\begin{lemma}\label{lemma:unipotent-torsion}
Let $\KK$ be a  field,
and let $\Gamma$ be a unipotent group over $\KK$.
Then $\Gamma(\KK)$ does not contain elements of finite
order coprime to the characteristic of~$\KK$.
\end{lemma}
\begin{proof}
Without  loss of generality we may assume that $\KK$ is algebraically closed. In this case the assertion of the lemma follows from a similar assertion for $\Ga$ together with
Lie--Kolchin theorem (see e.g.~\cite[Corollary~10.5]{Borel}), which implies that
any unipotent group over an algebraically closed field can be obtained as a consecutive
extensions of groups isomorphic to~$\Ga$.
\end{proof}

We will need the following auxiliary fact about
orders of finite groups with given exponents.

\begin{theorem}[{see \cite[Theorem~1]{HerzogPraeger}}]
\label{theorem:Burnside}
Let $n$ and $d$ be positive integers,
and let $\KK$ be a field.
Let~\mbox{$G\subset\GL_n(\KK)$} be a finite subgroup.  If $\Char\KK>0$, denote by~$|G|'$
the largest factor of $|G|$ which is coprime to
$\Char\KK$; otherwise put~\mbox{$|G|'=|G|$}.
Suppose that
for every~\mbox{$g\in G$} such that the order of $g$ 
is coprime to the characteristic of $\KK$,  one has $g^d=1$.
Then~\mbox{$|G|'\le d^n$}.
\end{theorem}

Now we prove Theorem~\ref{theorem:LAG}.

\begin{proof}[Proof of Theorem~\ref{theorem:LAG}]
Clearly, we may assume that $\Gamma$ is connected. Note that   assertion~(i) follows from  assertion~(ii) and
Corollary~\ref{cor:Springer}.

Let $\Gamma$ be an arbitrary linear algebraic group of rank $n$, and let $\gamma\in\Gamma(\KK)$ be an element of finite order coprime to $\Char\KK$.
Then, by Corollary~\ref{cor:Springerold}, the element
$\gamma$ is contained in some  subtorus of $\Gamma$.
Thus,  it follows from Corollary~\ref{corollary:tori}
that the order of  $\gamma$  is bounded by some
constant that depends only on $n$. In other words, there exists a constant $d(n)$, such that for every connected linear algebraic group of rank $n$ and every element
$\gamma\in\Gamma(\KK)$ of order coprime to $\Char \KK$, one has
$$
\gamma^{d(n)} =1.
$$

Let $\Gamma_{\bar{\KK}}$ be the algebraic group over $\bar{\KK}$ obtained from $\Gamma$ by the base change, and let
$\Delta_{\bar{\KK}}$ be the unipotent radical of~$\Gamma_{\bar{\KK}}$. (Note that unless $\KK$ is perfect the group $\Delta_{\bar{\KK}}$  need not be defined over $\KK$.)
 Then $\Gamma_{\bar{\KK}}/\Delta_{\bar{\KK}}$ is a
reductive linear algebraic group over $\bar{\KK}$.  Moreover, the group~\mbox{$\Gamma_{\bar{\KK}}/\Delta_{\bar{\KK}}$}
has the same rank as~$\Gamma_{\bar{\KK}}$
(which is equal to the rank of ~$\Gamma$).
Indeed, the rank of a unipotent algebraic group  is zero. Hence, the rank of~\mbox{$\Gamma_{\bar{\KK}}/\Delta_{\bar{\KK}}$} is greater than or equal to  the rank of~$\Gamma_{\bar{\KK}}$. On the other hand, by~\cite[Theorem~10.6(4)]{Borel},
every extension of a torus by
a unipotent group admits a section, which means that the rank  of~$\Gamma_{\bar{\KK}}$ is greater than or equal to the rank of the quotient~\mbox{$\Gamma_{\bar{\KK}}/\Delta_{\bar{\KK}}$}.
 By a theorem of Chevalley  (see~\mbox{\cite[Theorem~9.6.2]{Springer}})
there are only finitely many isomorphism classes of connected reductive
groups of given rank over an algebraically closed field.
Every such group is linear, that is, it admits a faithful
finite-dimensional representation.
Applying this to~\mbox{$\Gamma_{\bar{\KK}}/\Delta_{\bar{\KK}}$}
we infer that there exists a constant $N(n)$,
which  depends only on~$n$,  such that,
for every connected linear  algebraic group $\Gamma$ of rank $n$, the group ~\mbox{$(\Gamma_{\bar{\KK}}/\Delta_{\bar{\KK}}) (\bar{\KK})$}
admits a faithful representation in an $N(n)$-dimensional vector space over $\bar{\KK}$:
$$
(\Gamma_{\bar{\KK}}/\Delta_{\bar{\KK}})(\bar{\KK}) \hookrightarrow \GL_{N(n)}(\bar{\KK}).
$$
Composing this embedding  with the projection  $\Gamma(\KK) \to (\Gamma_{\bar{\KK}}/\Delta_{\bar{\KK}})(\bar{\KK})$ we construct a homomorphism
$$
\phi\colon \Gamma(\KK) \to \GL_{N(n)}(\bar{\KK}),
$$
 whose kernel is contained in  $\Delta_{\bar{\KK}}(\bar{\KK})$.
By Lemma~\ref{lemma:unipotent-torsion}, every element of finite order in~\mbox{$\Delta_{\bar{\KK}}(\bar{\KK})$} has order divisible by~$\Char\KK$.
This means that the image $\phi(G)$ of a finite subgroup~\mbox{$G\subset\Gamma(\KK)$} in  $\GL_{N(n)}(\bar{\KK})$
  has order divisible by the largest factor $|G|'$
of $|G|$ coprime to~$\Char\KK$; in particular,
if $\Char\KK=0$, then $G$ projects isomorphically to~\mbox{$(\Gamma_{\bar{\KK}}/\Delta_{\bar{\KK}})(\bar{\KK})$}.
Theorem~\ref{theorem:Burnside} applied to $\phi(G)$ gives us a bound $|G|'\le L(1,n)$, where
$$
L(1,n)=d(n)^{N(n)}.
$$
This proves assertion~(ii).

For the proof of assertion~(iii), observe that since the group scheme $\pi_1(\Gamma)$ is commutative, we have that $ \pi_1(\Gamma)\cong Z \times Z'$,  where $Z$ is a group scheme of order~$p^m$
and  $Z'$ is a group scheme whose order is coprime to $p$.
The central extensions
\[
\aligned
& Z' \to \tilde \Gamma \to \tilde \Gamma/Z',\\
&Z \to \tilde \Gamma /Z' \to \Gamma
\endaligned
\]
give rise to the exact sequences of groups
\[
\aligned
& Z'(\KK)   \rar{}  \tilde \Gamma (\KK) \rar{} (\tilde \Gamma / Z') (\KK) \rar{} H^1_{fl}(\Spec \KK, Z'),\\
& Z(\KK)   \rar{}  (\tilde \Gamma /Z')(\KK) \rar{} \Gamma(\KK) \rar{N} H^1_{fl}(\Spec \KK, Z),
\endaligned
\]
where the groups on the right stand for cohomology of $Z'$ and $Z$  regarded as  sheaves for the
fppf  topology on $\Spec \KK$, and $H^1_{fl}$
denotes the first cohomology group
for the fppf  topology (see, for example, \cite[\S\,III.4]{Milne1980}).
Set $H=G \cap  \ker N$.  By Corollary \ref{cor:Springer} the group  $\tilde \Gamma (\KK)$ has no elements of order $p$. Since the multiplication by $p$ is invertible in~$Z'$, the same is true for   $H^1_{fl}(\Spec \KK, Z')$. Hence $H$ has no elements of order $p$. Thus, by assertion~(ii) the order of~$H$ is less than or equal to $L(1,n)=d(n)^{N(n)}$.
On the other hand, by construction~$H$ is a normal subgroup of $G$, and $G/H$ is a subgroup of~\mbox{$H^1_{fl}(\Spec \KK, Z)$}. The latter is an abelian group annihilated by~$p^m$.  Hence, the same is true for $G/H$. Finally, since $H$ has no elements of order $p$,  a $p$-Sylow subgroup of $G$ projects  isomorphically to~$G/H$.
Thus, the group~$G$ is isomorphic to a semi-direct product of $H$ and~$G/H$.
\end{proof}

\section{Severi--Brauer varieties}
\label{section:SB}

In this section
we describe automorphism groups of Severi--Brauer varieties and prove Propositions~\ref{proposition:SB}
and~\ref{proposition:SBp}.
We refer the reader to \cite{Artin1982} for the definition and basic
facts concerning Severi--Brauer varieties.

Let $A$ be a central simple algebra of dimension $n^2$
over an arbitrary field $\KK$,
and $X$ be the corresponding
Severi--Brauer variety of dimension~\mbox{$n-1$}.
As usual, for a field extension~\mbox{$K\supset\KK$} we denote
by $A^*(K)$ the multiplicative group of invertible elements
in~\mbox{$A\otimes_{\KK} K$}.
In particular, one has $A^*(\KK^{sep})=\GL_n(\KK^{sep})$.
If $A$ is a division algebra, then
$A^*(\KK)$ consists of all non-zero elements of~$A(\KK)$.

The following fact is well known to experts (cf. Theorem~E on page~266
of~\cite{Chatelet}, or~\cite[\S1.6.1]{Artin1982}),
but for the reader's
convenience we provide a proof.

\begin{lemma}\label{lemma:SB-Aut}
One has $\Aut(X)\cong A^*(\KK)/\KK^*$.
\end{lemma}
\begin{proof}
Recall that the scheme $X$ represents the functor that takes a scheme $S$ over $\KK$ to the set of right ideals $I$  in the sheaf of algebras $A\otimes _{\KK} \mathcal{O}_S$ which are locally free of
rank~$n$ as $\mathcal{O}_S$-modules and are locally direct summands, that is,~\mbox{$I\oplus J =  A\otimes _{\KK} \mathcal{O}_S$} for some ideal~$J$.
The action of the group~\mbox{$A^*(\KK)$} on $A$ by conjugation induces an action of~\mbox{$A^*(\KK)$}  on the above functor and thus, by Yoneda Lemma, on $X$.
Obviously, the action of the central subgroup~\mbox{$\KK^*\subset A^*(\KK)$}
is trivial on $A$ and on $X$.
This gives a homomorphism of  groups
$$
\xi_{\KK}\colon  A^*(\KK)/\KK^* \to\Aut(X).
$$
Recall that
$X_{\KK^{sep}}\cong\PP_{\KK^{sep}}^{n-1}$ is the Severi--Brauer variety associated with the split central simple algebra
$A \otimes _{\KK} \KK^{sep}$ over $\KK^{sep}$.
Let~\mbox{$\Gamma=\Gal(\KK^{sep}/ \KK)$}.
We have a commutative diagram
\begin{equation*}
\def\normalbaselines{\baselineskip20pt
\lineskip3pt  \lineskiplimit3pt}
\def\mapright#1{\smash{
\mathop{\to}\limits^{#1}}}
\def\mapdown#1{\Big\downarrow\rlap
{$\vcenter{\hbox{$\scriptstyle#1$}}$}}
\begin{matrix}
 A^*(\KK)/\KK^*  & \rar{\xi_{\KK}} & \Aut(X) \cr
 \mapdown{}  && \mapdown{}  \cr
  (A^*(\KK^{sep})/(\KK^{sep})^*)^\Gamma & \rar{\xi_{\KK^{sep}}^\Gamma} &
  \big(\Aut(X_{\KK^{sep}})\big)^\Gamma
\end{matrix}
\end{equation*}
Here, for a group $C$ with an action of $\Gamma$,
we write $C^\Gamma$ for the subgroup of $\Gamma$-invariant elements.
The homomorphism
$$
\xi_{\KK^{sep}}\colon  A^*(\KK^{sep})/(\KK^{sep})^* \to\Aut(X_{\KK^{sep}})\cong\PGL_n(\KK^{sep})
$$
is an isomorphism. Thus, the lower horizontal arrow in the diagram
is also an isomorphism. In addition, the vertical arrows are injections.
Hence, to complete the proof it suffices to show
that the morphism
$$
A^*(\KK)/\KK^* \to  \big(A^*(\KK^{sep})/(\KK^{sep})^*\big)^\Gamma
$$
is surjective.
Indeed, the exact sequence of groups with $\Gamma$-action
$$
1\to (\KK^{sep})^* \to A^*(\KK^{sep}) \to A^*(\KK^{sep})/(\KK^{sep})^*\to 1
$$
gives rise to the exact sequence of Galois cohomology
groups
$$
1\to \KK^*\to A^*(\KK)\to  (A^*(\KK^{sep})/(\KK^{sep})^*)^\Gamma\to H^1(\Gamma, \KK^*).
$$
The latter cohomology group vanishes by Hilbert's Theorem~90,
and the assertion of the lemma follows.
\end{proof}

\begin{remark}
The above argument can be restated as follows: let $ \mathcal{A}^*$ be the algebraic group whose  $S$-points are invertible elements in the algebra $A \otimes _\KK  \mathcal{O}_S $.
There is a natural embedding $\Gm \to  \mathcal{A}^*$ and the quotient group scheme  is identified with the group
scheme~\mbox{$\underline{\Aut}(X)$}. Hilbert's Theorem~90 implies
that the group of $\KK$-points of the quotient~$\mathcal{A}^*/\Gm $ is  $A^*(\KK)/\KK^*$.
\end{remark}

\begin{lemma}\label{lemma:SB-exponent}
Let~$\KK$ be a field that contains all
roots of $1$.
Let $A$ be a central simple algebra  of dimension $n^2$ over a field $\KK$, and let~$x$ be an element of
$A^*(\KK)$ whose image in~\mbox{$A^*(\KK)/\KK^*$} has finite order.
Then the minimal polynomial $f(y)\in \KK[y]$ of $x$ has the form
$$
f(y)= \prod_i(y^r-a_i),
$$
for some positive integer $r$ that divides $n$ and some $a_i\in \KK$.
In particular, if $A$ is a division algebra, then the order of the image of $x$ in $A^*(\KK)/\KK^*$ divides $n$.
\end{lemma}

\begin{proof}
We know that for some
positive integer $m$, the element $a=x^m$ is contained in $\KK^*$.
We claim that there exists a positive integer $r$ dividing $m$ such that every irreducicble factor of the polynomial $y^m - a \in   \KK[y]$ has the form $y^r- b$ for some $b\in \KK$.
To show this, let $\alpha $  be a root of $y^m - a$ in $\bar{\KK}$, and let $r$ be the minimal positive integer  such that $b=\alpha ^r$ is contained in~$\KK$. Clearly, $r$ divides $m$. We need to check that the polynomial
$y^r- b$ is irreducible in  $\KK[y]$. Consider the extension $F=\KK(\alpha)\supset \KK$. Since
$\KK$ contains all roots of~$1$,
the extension $F$ is the splitting field of the polynomial $y^r- b$.

If the characteristic of $\KK$ does not divide $r$, then $F$ is separable over $\KK$ .  We have  an injective (Kummer) homomorphism
$\Gal(F/\KK) \hookrightarrow \mu_r$ to the cyclic group of $r$-th roots
of~$1$ which sends
$g\in \Gal(F/\KK)$ to $\frac{g(\alpha)}{\alpha}$. If $\mu_{r'} \subset \mu_r$ is its image, then~$\alpha^{r'}$ is invariant under the action of the Galois group and, thus, belongs to~$\KK$.
Therefore, one has~\mbox{$r=r'=[F: \KK]$}. This implies that  $y^r- b$ is irreducible in  $\KK[y]$.

If $\KK$ has positive characteristic $p$ which divides $r$,  write $r=p^k m$ with $m$ coprime to~$p$. The subfield $\KK(\alpha ^{p^k})\subset F$ is the splitting field of the separable polynomial $y^m - b$, which, by the previous argument, is irreducible over $\KK$.
On the other hand,  $\KK(\alpha ^{m})\subset F$ is the splitting field of the polynomial
$$
y^{p^k}- b= (y- \alpha ^{m})^{p^k}
$$
which is also irreducible over $\KK$. Since
the extension  $\KK \subset \KK(\alpha ^{p^k})$ is separable and the extension  $\KK \subset \KK(\alpha ^m)$ is purely inseparable, the degree of their composite is equal to the product
$$
[\KK(\alpha ^{p^k})\colon \KK] \cdot[  \KK(\alpha ^{m})\colon \KK]=r.
$$
Hence,  the degree of $F$ over $\KK$ is at least $r$, and therefore $y^r- b$ is irreducible in  $\KK[y]$.

Since the minimal polynomial $f(y)$ of $x$ divides  $y^m - a$, each of its irreducible factors also has the
form $y^r- b$ for some integer $r$ which divides $m$ and $b\in\KK$ such that ~$b^{\frac{m}{r}}=a$. We need to show that $r$ divides $n$.  Consider the subalgebra
$$
R=\KK[x] \subset A.
$$
By definition of the minimal polynomial, the algebra $R$ is isomorphic to  the quotient~\mbox{$\KK[y]/(f(y))$}.
If~\mbox{$y^r-b$} is an irreducible factor of $f(y)$ then,
using the assumption  that~$\KK$ contains all roots of~$1$ and the Chinese Remainder theorem,
we see that there exists an  integer~$e$ and a unital monomorphism
$$\KK[y]/((y^r- b)^e)\hookrightarrow R.$$

If the characteristic of $\KK$ does not divide $r$, then by
Hensel's lemma the projection
$$
\KK[y]/((y^r- b)^e) \to \KK[y]/(y^r- b)
$$
splits, and thus $A$ contains a subfield isomorphic to
$\KK[y]/(y^r- b)$. On the other hand, the degree $[F\colon \KK]$ of any subfield $F \subset A$ divides $n$. Indeed, the central simple algebra~\mbox{$A\otimes _{\KK} F$}
over $F$ acts on the $F$-vector space $A$. Hence, $[F\colon \KK]=\dim_F A$ is divisible by $n$. Therefore, we see that~$r$ divides~$n$.

If $\Char \KK = p$ divides $r$, write $r=p^k m$ with $m$ coprime to~$p$. Then $R$ contains a subfield isomorphic to $\KK[z]/(z^m- b)$. Therefore, $m$ divides $n$. Next, the algebra~\mbox{$A\otimes _{\KK} \KK^{sep}$} splits and thus
has a representaion of dimension $n$ over~$\KK^{sep}$. On the other hand, $R\otimes _{\KK} \KK^{sep}$ contains a subalgebra of the form~\mbox{$\KK^{sep}[u]/((u^{p^k}- b)^e)$}. Since the polynomial $u^{p^k}- b$ remains irreducible over $\KK^{sep}$, the dimension of any reprentation of~\mbox{$R\otimes _{\KK} \KK^{sep}$}  in a $\KK^{sep}$-vector space is divisible by~$p^k$. In particular, $n$ is divisible by $p^k$.

If $A$ is a division algebra, then the minimal polynomial of any of its elements is irreducible and the last assertion of the lemma follows.
\end{proof}

\begin{remark}
T.\,Bandman and Yu.\,Zarhin proved in~\cite[Theorem~3.4]{BandmanZarhin2017} a special case of
Lemma~\ref{lemma:SB-exponent} for $A=\mathrm{Mat}_n(\KK)$ and $\Char \KK=0$.
\end{remark}

Theorem~\ref{theorem:Burnside} implies the following result
about orders of finite groups with given exponents.

\begin{lemma}
\label{lemma:PGL-order-vs-exponent}
Let $n$ and $d$ be positive integers,
and let $\KK$ be a field such that~\mbox{$\Char\KK$} does not divide $d$.
If $\Char\KK>0$, denote by $n'$  the largest factor of $n$ which is coprime to~\mbox{$\Char\KK$}; otherwise put $n'=n$.
Let $G\subset\PGL_n(\KK)$ be a finite subgroup. Suppose that
for every~\mbox{$g\in G$} one has $g^d=1$.
Then
$$
|G|\le (n'd)^{n-1}.
$$
Moreover, if $n=d=3$ and $\Char\KK=0$, then
$$
|G|\le 27.
$$
\end{lemma}
\begin{proof}
We can assume that the field $\KK$ is algebraically (or separably) closed.
Let~\mbox{$\tilde{G}\subset\SL_n(\KK)$} be the preimage
of $G$ with respect to the natural projection~\mbox{$\phi\colon \SL_n(\KK)\to\PGL_n(\KK)$}. The kernel of $\phi$ is  a cyclic group of  order $n'$ that consists of scalar matrices.  Thus,
for every $g\in\tilde{G}$, one has $g^{n'd}=1$ and $|\tilde G| = n'|G|$.
Let~\mbox{$\mu_{n' d} \subset \GL_n(\KK)$} be the subgroup of scalar matrices whose order divides $n'd$. Consider the subgroup
$\hat{G}\subset  \GL_n(\KK)$ generated by $\tilde{G}$ and~$\mu_{n' d}$.
The order of any of its elements still divides~$n'd$. Thus, by Theorem~\ref{theorem:Burnside}, we have
$$
|\hat{G}|\le  (n'd)^n.
$$
On the other hand, we also have
$$
|\hat{G}|=n'd |G|,
$$
and the first assertion of the lemma follows.

If $n=d=3$, then $|G|=3^r$ for some $r$, so that the second assertion of the lemma follows from
the classification of finite subgroups of $\PGL_3(\KK)$ over a field of characteristic zero, see~\cite[Chapter~V]{Blichfeldt}.
\end{proof}

Now we are ready to prove Proposition~\ref{proposition:SB}.

\begin{proof}[Proof of Proposition~\ref{proposition:SB}]
Suppose that $A$ is a division algebra.
Let $g\in \Aut(X)$ be an element of finite order.
We claim that $g^{n}=1$.
Indeed, by Lemma~\ref{lemma:SB-Aut}
the element~$g$ corresponds to some invertible element $x$
of $A$ (defined up to~$\KK^*$). Since $A$ is a division algebra, the minimal polynomial of $x$ must be irreducible over $\KK$.
By Lemma \ref{lemma:SB-exponent}, the minimal polynomial of~$x$ has the form $y^r -a$
for some $a\in \KK^*$ and some positive integer~$r$ which divides~$n$.

Furthermore, one has
$$
G\subset\Aut(X)\subset\Aut(X_{\KK^{sep}})\cong\PGL_n(\KK^{sep}).
$$
Therefore, by Lemma~\ref{lemma:PGL-order-vs-exponent} we have $|G|\le {n}^{2(n-1)}$ in general, and
also $|G|\le 27$ in the case when $\Char\KK=0$ and $n=3$. Moreover,
every central simple algebra of dimension $n^2$ is a division algebra provided that
$n$ is a prime number. This proves assertions~(ii) and~(iii).

Now suppose that $A$ is not a division algebra. Then
$$
A\cong D\otimes_{\KK}\mathrm{Mat}_m(\KK)
$$
for some $2\le m\le n$, where $\mathrm{Mat}_m(\KK)$ denotes the algebra of $m\times m$-matrices.
Thus~$A$ contains $\mathrm{Mat}_m(\KK)$ as a subalgebra. Since the field~$\KK$ contains
roots of $1$ of arbitrarily large degree, we see from Lemma~\ref{lemma:SB-Aut} that the group $\Aut(X)$ contains elements of
arbitrarily large finite order. This completes the proof of assertion~(i).
\end{proof}

\begin{remark}\label{remark:char-vs-dim-perfect}
Let $\KK$ be a perfect field of positive characteristic $p$, and let $A$ be a division algebra of dimension~$n^2$ over $\KK$.
Then $p$ does not divide $n$. Indeed, the Frobenius morphism
$\mathrm{Fr}\colon \bar{\KK}^* \to \bar{\KK}^*$ is an isomorphism and, hence, the Brauer group
$$
\Br(\KK)\cong H^2(\Gal(\bar{\KK}/\KK), \bar{\KK}^*)
$$
of $\KK$ has no $p$-torsion elements and it is $p$-divisible. Therefore, our claim  follows from the fact that, over any field, the dimension of a central division algebra and the order of its class in the Brauer group have the same prime factors
(see for instance~\mbox{\cite[Lemma~2.1.1.3]{Lieblich}}).
\end{remark}

\begin{remark}\label{remark:SB-via-LAG}
Let $A$ be a central simple algebra  over a field $\KK$.  Denote by $ \mathcal{A}^*$ the algebraic group whose  $S$-points are invertible elements in the algebra $A \otimes _\KK  \mathcal{O}_S$. We have a natural embedding
$\Gm \hookrightarrow \mathcal{A}^*$ induced by the homomorphism $\mathcal{O}_S^* \hookrightarrow (A \otimes _\KK  \mathcal{O}_S)^*$.  The quotient group scheme
$\mathcal{A}^*/\Gm$ is anisotropic if and only if $A$ is a division algebra (see, for instance,~\mbox{\cite[\S23.1]{Borel}}).
In particular,  if $\KK$ is perfect,  then Proposition~\ref{proposition:SB}(i) follows from Theorem~\ref{theorem:LAG} applied to the reductive group $\mathcal{A}^*/\Gm$.
\end{remark}

The restriction on the characteristic of $\KK$ in  Proposition~\ref{proposition:SB} is essential for validity of the statement.

\begin{example}
\label{example:BMR}
Let $F$ be a field of characteristic $p>0$, and
let~\mbox{$\KK=F(x, y)$} be the field of rational
functions in two variables, so that $\KK$ is a non-perfect field of characteristic~$p$.
Let $A$ be an algebra over~$\KK$
with generators $u$ and $v$
and relations
$$
v^p =x,\quad u^p=y,\quad vu-uv=1.
$$
Then $A$ is a central division algebra of dimension~$p^2$ over $\KK$. This is a special case of the Azumaya property of the ring of differential operators in characteristic~$p$ (see~\cite[Theorem~2.2.3]{BMR08}), but can be also checked directly. The group $A^*(\KK)/\KK^*$  contains~\mbox{$F(v)^*/F(v^p)^*$} as a subgroup. The latter
is an infinite-dimensional vector space over the field~$\mathbf{F}_p$ of~$p$ elements.
In particular, for~\mbox{$p=2$} this construction provides an example of a conic $C$ over a non-perfect field of characteristic~$2$
such that $C$ is acted on by elementary $2$-groups of arbitrarily large order.
\end{example}

Now we prove Proposition~\ref{proposition:SBp}.

\begin{proof}[Proof of Proposition~\ref{proposition:SBp}]
By Lemma~\ref{lemma:SB-Aut} we have $\Aut(X)\cong A^*(\KK)/\KK^*$.
Recall the reduced norm homomorphism:
$$
\Norm \colon A^*(\KK) \to \KK^*.
$$
One has $\Norm(cx) = c^n \Norm(x)$ for every
$c\in \KK^*$ and $x\in  A^*(\KK)$. Hence, $\Norm$ induces a homomorphism
\begin{equation}\label{norm}
 A^*(\KK)/\KK^* \to  \KK^*/(\KK^*)^n,
\end{equation}
where $(\KK^*)^n \subset  \KK^*$ is the subgroups of $n$-th powers.  Composing homomorphism~\eqref{norm} with the projection   $\KK^*/(\KK^*)^n \to  \KK^*/(\KK^*)^{p^m}$, we get
\begin{equation}\label{norm1}
A^*(\KK)/\KK^* \to  \KK^*/(\KK^*)^{p^m}.
\end{equation}

Let $\Gamma(A)$ be the kernel of homomorphism~\eqref{norm1}. We claim that
the order of every element  of $\Gamma(A)$ of finite order divides~$n'$.
Indeed, if $n=n'$, this follows from  Proposition~\ref{proposition:SB}(ii). Hence, we may assume that~\mbox{$m>0$}.
By Lemma \ref{lemma:SB-exponent},  the order of every element  of $\Gamma(A)$ of finite order divides~$n$. Thus, it suffices to check that
$\Gamma(A)$ has no elements of order $p$.  Assuming the contrary,  let $g\in  \Gamma(A)$ be an element of order $p$, and let $x$
be its preimage in $A^*(\KK)$.
Then $x^p =a$ for some $a \in \KK^*$.
Since $a$  does not belong to $(\KK^*)^{p}$, the image of~$a$
in $\KK^*/(\KK^*)^{p^m}$ has order $p^m$. On the other hand, we have~\mbox{$\Norm (x)= a^ \frac{n}{p}$}.
Thus,~\mbox{$\Norm(x)$} is not equal to $1$ in $\KK^*/(\KK^*)^{p^m}$, that is,
$g$ does not belong to $\Gamma(A)$, which gives a contradiction.

It follows that for every finite subgroup $H\subset \Gamma(A)$,  one has~\mbox{$|H|\le {n'}^{2(n-1)}$}.
The proof repeats {\it verbatim} the proof of the corresponding estimate from  Proposition~\ref{proposition:SB}(ii),
using the above multiplicative bound on the orders of finite order elements of~\mbox{$\Gamma(A)$}.

Now let $G\subset  A^*(\KK)/\KK^*$  be a finite subgroup. Set $H= G\cap \Gamma(A)$. Then $H$ is normal and the quotient
$G/H$ is a subgroup of  $\KK^*/(\KK^*)^{p^m}$. Moreover, a $p$-Sylow subgroup of~$G$ must project isomorphically to $G/H$.
Thus, $G$ is isomorphic to a semi-direct product of~$H$ and~$G/H$.
\end{proof}

\begin{remark}\label{remark:SBp-via-LAG}
Under the assumptions of  Proposition~\ref{proposition:SBp},  the fact that every finite subgroup of  $  A^*(\KK)/\KK^*$ is a semi-direct product of its abelian $p$-Sylow subgroup  and a normal subgroup of bounded order is a special case of
 Theorem~\ref{theorem:LAG}(iii) applied to the reductive group $\mathcal{A}^*/\Gm$.
Indeed, the algebraic fundamental group of  $\mathcal{A}^*/\Gm$ is isomorphic to the group scheme $\mu_n=\ker(\Gm\rar{n} \Gm)$, which has order $n$. Also, since  $\mathcal{A}^*/\Gm$
has type~$\mathrm{A}_{n-1}$, the set of torsion primes for  $\mathcal{A}^*/\Gm$ is empty.
\end{remark}

Later we will need the following generalization of Theorem~\ref{theorem:Zarhin-conic}.

\begin{corollary}
\label{corollary:Zarhin-improved}
Let $\KK$ be  a field that contains all roots of~$1$, and let
$C$  be a conic over~$\KK$ with $C(\KK)=\varnothing$. Then
 every finite  subgroup  of $\Aut(C)$ is isomorphic to $(\Z/2\Z)^n$ for some non-negative integer $n$. Moreover, if $\Char \KK \ne 2$ or $\KK$ is perfect then $n\le 2$.
 \end{corollary}
\begin{proof}
The first assertion follows from Proposition~\ref{proposition:SBp}. Furthermore, if $\KK$ is perfect, then~\mbox{$\Char\KK\neq 2$} by Remark~\ref{remark:char-vs-dim-perfect}.
Thus, the second assertion follows from
Proposition~\ref{proposition:SB}(iii).
\end{proof}

We conclude this sections with a few further remarks on division algebras over a field of characteristic $p$.

\begin{lemma}\label{boundnessfornonperfect}
Let $A$ be  a division algebra of dimension $n^2$ over a field $\KK$ of finite characteristic  $p$ that contains all roots
of $1$. Then  the group $A^*(\KK)/\KK^*$ has bounded finite subgroups if and only if every element of $v\in A$ is separable over $\KK$, that is, the field extension~\mbox{$\KK(v)\supset \KK$} is separable.
\end{lemma}
\begin{proof}
If $v\in A$ is not separable, then there exists a subfield
$$
\KK \subset \LL \subset \KK(v)
$$
which is a purely inseparable extension of $\KK$ of degree $p$. Then
every non-trivial element of the group~\mbox{$\LL^*/\KK^*$} has order~$p$. Since $|\LL^*/\KK^*|=\infty$,
we conclude that the group~\mbox{$\LL^*/\KK^* \subset A^*(\KK)/\KK^*$} has unbounded finite subgroups.

Conversely, suppose that every element of $A$ is separable over $\KK$. Denote by $n'$  the largest factor of $n$ which is coprime to $\Char\KK$. Then by Lemma~\ref{lemma:SB-exponent}, for every element~\mbox{$v\in  A^*(\KK)$} whose image in  $A^*(\KK)/\KK^*$ has finite order, one has~\mbox{$v^{n'} \in \KK^*$}  (otherwise~$v^{n'}$ would be inseparable over~$\KK$).
Hence, the assertion follows from  Lemma~\ref{lemma:PGL-order-vs-exponent}.
\end{proof}

There are examples of division algebras of dimension $p^4$ over a field $\KK$ of characteristic~$p$ such that all their elements are separable over $\KK$. On the other hand,
we do not know if it is true that every  division algebra $A$ of dimension $p^2$ over a field $\KK$ of characteristic  $p$ contains an element $v\in A$ which is inseparable over $\KK$.
We refer the reader to~\mbox{\cite[\S1]{ABGV11}} and references therein for a further discussion of this problem.

The following result shows that  the division algebra in Example~\ref{example:BMR} provides a universal example of a division algebra of dimension $p^2$ over a field of characteristic $p$ which contains an inseparable element.

\begin{lemma}[{cf. ``Albert's cyclicity criterion'', \cite{Alb61}}]
\label{p-center}
Let $A$ be  a division algebra of dimension $p^2$ over a field $\KK$ of finite characteristic  $p$.
\begin{itemize}
\item[(i)]
The following conditions are equivalent.
\begin{itemize}
\item[(1)] $A$  contains an element $v\in A$ which is inseparable over $\KK$.
\item[(2)] $A$ is generated by two elements   $v, u \in A$   such that $v^p, u^p \in \KK$ and    $vu-uv=1$.
\item[(3)]$A$   contains  a subfield $\LL\subset A$ that is a Galois extension of $\KK$  of degree $p$.
\end{itemize}
\item[(ii)]
If $p=2$ or $p=3$, then every  division algebra $A$ of dimension $p^2$ over a field $\KK$ of characteristic  $p$ contains an element $v\in A$ which is inseparable over $\KK$.
In partucular, for every Severi--Brauer variety $X$ of dimension $p-1$ over an infinite field $\KK$ of characteristic $p$, the group $\Aut(X)$ has unbounded finite subgroups.
\end{itemize}
\end{lemma}
\begin{proof}
The equivalence of conditions~(1) and~(3) is due to Albert \cite[Theorem~XI.4.4]{Alb61}.

Let us show that condition~(1) implies condition~(2).  Suppose that an element  $v\in A$ is inseparable over $\KK$.
Consider the linear map $\ad v\colon A \to A$, which takes  an element $w\in A$ to the commutator~\mbox{$vw-wv$}. Since $v^p \in \KK$, we see that $(\ad v)^p=\ad v^p =0$. On the other hand, the kernel of $\ad v$ has dimension $p$ over $\KK$. Hence, the image of $\ad v$ is equal
to the kernel of $(\ad v)^{p-1}$. In particular, the element $1$ is in the image of $\ad v$. Thus, there exists an element $u\in A$ such that   $vu-uv=1$. Let
$$
y^p + a_{p-1} y^{p-1} +\ldots + a_0 \in \KK[y]
$$
be the minimal polynomial for $u$. Applying the operator $\ad v$ to the equality
$$
u^p + a_{p-1} u^{p-1} +\ldots +  a_0=0,
$$
we compute that
$$
(p-1)a_{p-1} u^{p-1} +\ldots + 2a_2 u+ a_1=0,
$$
which is impossible unless $a_{p-1}= \ldots =a_1=0$. Thus, we have $u^p \in \KK$.

To prove that condition~(2) implies condition~(3),
assume that $v, u\in A$ are as in~(2). Then one can easily check (or deduce from \cite[Lemma 1.3.1]{BMR08}) that
$$
(uv)^p - uv = u^pv^p \in \KK.
$$
Thus, the subfield $\KK(uv)\subset A$ is a Galois extension of $\KK$.

Finally, assume that a subfield $\LL \subset A$ is a Galois extension of $\KK$.  Its Galois group is a cyclic group of order $p$. Let $\sigma$ be its generator. By the Skolem--Noether theorem the action of $\sigma$  on $\LL$ extends to
an inner automorphism
of $A$ given by an element $v\in A$, that is, $vuv^{-1}=\sigma(u)$ for every $u\in \LL$.  Then, since the normalizer of $\LL$ in $A$ is $\LL$ itself,  we have that
$v^p\in \LL$. Since $\LL$ is separable over $\KK$, the element $v^p$ is, in fact, contained in~$\KK$. This completes
the proof of assertion~(i).

To prove assertion~(ii), recall that every division algebra of dimension $p^2$ contains a subfield $\LL\subset A$ of degree $p$ over $\KK$, which is separable over $\KK$. If $p=2$, then
any such subfield is a Galois extension of $\KK$. For $p=3$, existence of a subfield  $\LL\subset A$ that is a Galois extension $\KK$  of degree $p$  was proved in  \cite{Wed21}.
The last assertion follows from  Lemmas~\ref{lemma:SB-Aut} and~\ref{boundnessfornonperfect}.
\end{proof}

\section{Quadrics}
\label{section:quadrics}

In this section we study automorphism groups of quadrics and
prove assertions~\mbox{(i)--(iii)} of
Proposition~\ref{proposition:quadricnew}.

Let $V$ be a finite-dimensional vector space over a field $\KK$,  and let $q$ be a non-degenerate quadratic form on $V$.
Recall that a quadratic form $q$ is said to be non-degenerate if the associated symmetric bilinear form
$$
B_q\colon V\times V \to \KK,\quad B_q(v,w)= q(v+w) -q(v) -q(w),
$$
is non-degenerate. If $\Char \KK =2$ the form $B_q$ is also alternating. Hence, in this case the dimension of $V$ must be even.

Denote by $\mathrm{O}(V,q)$ the orthogonal (linear algebraic) group corresponding to $q$.
The following result is well know (see, for instance,~\cite[\S\S22.4, 22.6]{Borel}).

\begin{lemma}\label{lemma:O-anisotropic}
The group  $\mathrm{O}(V,q)$ is reductive. It is anisotropic if and only
if $q$ does not represent $0$.
\end{lemma}

We will also need the following structural result on quadratic forms over a perfect field of characteristic $2$ due to Arf (\cite{Arf41}).

\begin{lemma}\label{lemma:quadratic-form-in-char-2}
Let $V$ be a finite-dimensional vector space over a perfect field $\KK$ of characteristic $2$,  and let $q$ be a non-degenerate quadratic form on $V$ (so that in particular~\mbox{$\dim V=2k$} is even).  Then, for
 some coordinates $x_1, \ldots, x_{2k}$ on $V$, the quadratic form $q$ is given by
\begin{equation}\label{canforminchartwo}
q_a(x_1, \ldots , x_{2k})= x_1^2 +x_1 x_2 + a x_2^2 + x_3 x_4 +\ldots +x_{2k-1}x_{2k},
\end{equation}
where $a$ is an element of $\KK$.
Moreover, two quadratic forms $q_a(x_1, \ldots , x_{2k})$ and~\mbox{$q_{a'}(x_1, \ldots , x_{2k})$} are equivalent if and only if $a$ and $a'$ have the same image
in the cokernel of Artin--Schreier homomorphism
$$
\KK \to \KK,\quad c \mapsto c^2 -c,
$$
which is  the Arf invariant  of the quadratic form.
In particular, if $\dim V >2$ then every non-degenerate quadratic form on $V$ represents $0$.
\end{lemma}

\begin{lemma}[{cf. \cite[Lemma~2.1]{GA13}}]
\label{lemma:orderp-in-O}
Let $\KK$ be a field of characteristic $p>2$.
Let $V$ be a vector space over $\KK$, and let $q$ be a non-degenerate quadratic form on~$V$. Assume that $q$ does not represent~$0$. Then  the group of  $\KK$-points of~\mbox{$\mathrm{O}(V,q)$} has no elements of order~$p$.
 \end{lemma}
\begin{proof}
Assuming the contrary, let $g\in  \mathrm{O}(V,q)(\KK)$ be an element of order $p$. Viewing  $g$  as a linear endomorphism of $V$, we have
$$
g^p -1 =(g-1)^p=0.
$$
Applying the Jordan Normal Form theorem to $g-1$,
we can find linearly independent vectors  $v_1, v_2 \in V$
such that $g (v_1) = v_1$  and $g (v_2) = v_1 + v_2$.  Thus, we have
$$
B_q(v_1, v_2) = B_q(g (v_1), g (v_2 )) = B_q(v_1, v_1) + B_q(v_1, v_2).
$$
Hence, we obtain $2q(v_1) = B_q(v_1, v_1)= 0$, that is, $q$ represents $0$.
\end{proof}

\begin{lemma}[{cf. \cite[Corollary~4.4]{BandmanZarhin2015a}}]
\label{lemma:subgroups-in-O}
Let $\KK$ be a  field that contains all roots of $1$.
Assume that $\Char \KK \ne 2$ or $\KK$ is perfect.
Suppose that~$q$ does not represent $0$.
Then every non-trivial element of finite order in
$\mathrm{O}(V,q)(\KK)$ has order $2$,
and every finite subgroup of~\mbox{$\mathrm{O}(V,q)(\KK)$}
is abelian of order less or equal to~$2^{\dim V}$.
\end{lemma}
\begin{proof}
By Lemma \ref{lemma:quadratic-form-in-char-2}, if $\KK$ is perfect and $\Char \KK=2$, then we must have $\dim V =2$.
 In this case, by  Lemma \ref{lemma:O-anisotropic}, the group $\mathrm{O}(V,q)$ is anisotropic and, thus,
 isomorphic to the product of an anisotropic torus of rank $1$ and the finite group $\Z/2\Z$. Therefore, the assertion of the lemma follows from Lemma~\ref{lemma:P1-without-2-pts}.

Now, assume  that $\Char\KK\neq 2$.
Let $g\in  \mathrm{O}(V,q)(\KK)$ be an element of finite order. By Lemma \ref{lemma:orderp-in-O} the order of $g$ is coprime
to the characteristic of $\KK$.
Since $\KK$ contains all roots of unity, it follows that every
such element~\mbox{$g\in  \mathrm{O}(V,q)(\KK)$} viewed as a linear endomorphism of $V$
is diagonalizable in an appropriate basis for $V$.
Moreover, since $q$ does not represent~$0$, the diagonal entries of the matrix
of $g$ in this basis must be equal to~\mbox{$\pm 1$}. Hence $g^2=1$.
It follows that every finite subgroup~\mbox{$G\subset\mathrm{O}(V,q)(\KK)$} is abelian, and $G$ is conjugate
to a subgroup of
the group of diagonal matrices in $\GL(V)$.
Hence, the order of~$G$ is at most~$2^{\dim V}$.
\end{proof}

Now we are ready to prove most of
Proposition~\ref{proposition:quadricnew}.

\begin{proof}[Proof of assertions~\mbox{(i)--(iii)} of
Proposition~\ref{proposition:quadricnew}]
Let $V$ be an $n$-dimensional vector space such that $\PP^{n-1}$ is identified
with the projectivization $\PP(V)$, and let $q$ be a quadratic form
corresponding to the quadric $Q$.

First, assume that $\KK$ is a perfect field of characteristic $2$. If $n$ is even, then the quadratic form is non-degenerate;
indeed, otherwise its kernel $T$ would be at least two-dimensional, so that the singular locus of $Q$,
which is $Q\cap\PP(T)$, would be non-empty.
Thus, by Lemma~\ref{lemma:quadratic-form-in-char-2} one has~\mbox{$Q(\KK)\ne \varnothing$}.  Moreover, writing  $q$ in the form~\eqref{canforminchartwo},
we see that~\mbox{$\Aut(Q)$} contains a subgroup isomorphic to~$\KK^*$.
Hence, $\Aut(Q)$ has unbounded finite subgroups. If $n$ is odd, the symmetric bilinear form
$B_q\colon V\times V \to \KK$ associated to $q$ has
a one-dimensional kernel.
In this case $q$ can be written as
$$
q(x_1, \ldots , x_n)= x_1^2 + r(x_2, \ldots, x_n)
$$
for some coordinates $x_1, \ldots, x_{n}$ on $V$ and some
non-degenerate quadratic form $r$ in $n-1$ variables.  Applying  Lemma ~\ref{lemma:quadratic-form-in-char-2} to  $r$, we see that  ~\mbox{$Q(\KK)\ne \varnothing$}
and $\Aut(Q)$ has unbounded finite
subgroups. In fact, in this case $\Aut(Q)$ is isomorphic to  the group of linear transformations
of the quotient $\bar{V}$ of $V$ by $T$
which preserve the induced bilinear form $\bar{B}_q$
on~$\bar{V}$, i.e., to the symplectic group
$\mathrm{Sp}(\bar{V}, \bar{B}_q)(\KK)$; see~\cite[\S22.6]{Borel}.
We see that in the case when $\Char \KK = 2$, assertion~(i) holds, while the assumptions of assertions~(ii), (iii),
and~(iv) do not hold.

From now on we assume that $\Char\KK \ne 2$.
The group $\Aut(Q)$ is isomorphic to the group of $\KK$-points of
the group scheme quotient $\Gamma= \mathrm{O}(V,q)/\mu_2$, where
$\mu_2 \subset  \mathrm{O}(V,q)$ is the central subgroup of order~$2$.
(More geometrically,  $\Aut(Q)$ can be identified with the group of
automorphisms of the projective space  $\PP(V)$ that preserve $q$ up
to a scalar multiple.)
By Lemma~\ref{lemma:O-anisotropic},  the group $\Gamma $ is anisotropic if and only if~\mbox{$Q(\KK)=\varnothing$}.
The connected component of identity $\Gamma^{\circ}\subset \Gamma$  is a connected semi-simple algebraic group.
Moreover,  the algebraic fundamental
group  of $\Gamma^\circ $  has order $2$ if  $n$ is odd and order $4$  if $n$ is even.
Also, note that no prime other than $2$ is a torsion prime  for $\Gamma^{\circ}$. Thus, assertion~(i) of the proposition follows from Theorem~\ref{theorem:LAG}(iii) applied to
$\Gamma^{\circ}$. (It also follows that the group~\mbox{$\Aut(Q)$} has no elements of
order~\mbox{$\Char\KK$};  cf.  Lemma~\ref{lemma:O-anisotropic}.)
The reader will see that the argument we give below for the remaining assertions of the proposition also furnishes a direct proof of~(i).

If $n$ is odd, then the embedding
$\mu_2 \hookrightarrow \mathrm{O}(V,q)$ splits,
so that  $\Aut(Q)$ is isomorphic to the subgroup~\mbox{$\mathrm{SO}(V,q)(\KK)\subset \mathrm{O}(V,q)(\KK) $} of the orthogonal group that consists of matrices whose determinant is equal to $1$.
Thus, assertion~(ii) follows from
Lemma~\ref{lemma:subgroups-in-O}.

Suppose that $n$ is even. Then, using the exact sequence
\begin{equation}\label{exactsequence}
0\to \mu_2 \to  \mathrm{O}(V,q)(\KK)\to \Aut(Q) \to \KK^*/(\KK^*)^2.
\end{equation}
and Lemma~\ref{lemma:subgroups-in-O}, we infer that
every non-trivial element of  $\mathrm{O}(V,q)(\KK)$
of finite order has order $2$.
Hence,  every non-trivial element of $\Aut(Q)$ of finite order
has order $2$ or $4$.
Let~\mbox{$G \subset \Aut(Q)$} be a finite subgroup.
Consider the embedding
$$
\Aut(Q)\cong \Gamma (\KK) \hookrightarrow  \Gamma(\bar{\KK}) \cong  \mathrm{O}(V,q)(\bar{\KK})/\{\pm 1\},
$$
and let $\tilde G$ be the preimage of $G$ in  $\mathrm{O}(V,q)(\bar{\KK})$. The order of every element of~$\tilde G$
divides~$8$, and the same is true for the subgroup
$\hat{G}\subset \GL(V)(\bar{\KK})$ generated by $\tilde G$
and scalar matrices whose orders divide $8$.
Thus, by Theorem \ref{theorem:Burnside}, we have
$|\hat{G}| \le 8^n$.
On the other hand, we know that~\mbox{$|\hat{G}|= 8|G|$}. This proves
assertion~(iii).
\end{proof}

Proposition~\ref{proposition:quadricnew}(iii) implies that
if $\KK$ is a field  that contains all roots of $1$ such that either
$\Char \KK \ne 2$ or~$\KK$ is perfect, and $Q$ is a smooth quadric surface
over $\KK$ with~\mbox{$Q(\KK)=\varnothing$}, then
every finite subgroup of $\Aut(Q)$ has order
at most $8^3=512$.
We will complete the proof of assertion~(iv)
of Proposition~\ref{proposition:quadricnew} by showing in
Corollary~\ref{cor:DP8pt} that this bound can be improved to~$32$.

\section{Varieties with zero irregularity}
\label{section:zero-irregularity}

In this section we study automorphism groups of  varieties 
with zero irregularity and prove Theorem~\ref{theorem:family}.
For any  projective  variety $X$ over a perfect field $\KK$ 
with~\mbox{$h^1(\mathcal{O}_{X})=0$},  
the connected component of identity  $\underline{\Aut}^\circ(X) \subset \underline{\Aut}(X)$
of the group scheme of automorphisms of $X$ is a linear algebraic group, whereas the group of connected
components~\mbox{$\underline{\Aut}(X)/\underline{\Aut}^\circ(X)$} can be very
difficult to understand.
For instance, there exists a surface~$X$ that is a blow up of a~$K3$ surface
such that~\mbox{$\underline{\Aut}(X)$} is a discrete infinitely generated group, see~\cite{DinhOguiso}.
However, for any normal projective variety $X$ over a field of characteristic zero the group
$\underline{\Aut}(X)/\underline{\Aut}^\circ(X)$
has bounded finite subgroups, see~\mbox{\cite[Lemma~2.5]{MZ}}.
Therefore, under these assumptions, the group $\Aut(X)$  has bounded finite subgroups
if and only if~\mbox{$\underline{\Aut}^\circ(X)(\KK)$} enjoys the same property.
We expect that this result holds for  any projective scheme over an arbitrary field.

Given a scheme $\mathcal{S}$ and a locally free sheaf $\mathcal{E}$ on $\mathcal{S}$,
we denote by $\PGL(\mathcal{E})$ the group scheme over $\mathcal{S}$  of projective linear automorphisms of $\mathcal{E}$, that is, the group scheme~\mbox{$\underline{\Aut}(\Proj (S^\bullet  \mathcal{E}))$}
of automorphisms (over $\mathcal{S}$) of the relative projective space.
Thus, over an open subscheme $U\subset \mathcal{S}$,  over which $\mathcal{E}$ is free of rank~$n$,
the group scheme $\PGL(\mathcal{E})$ is isomorphic to $\PGL_n \times U$.
The following observation is well known.

\begin{lemma}\label{lemma:ample}
Let $\mathcal{S}$ be a locally  Noetherian   scheme, and let
$\pi\colon \mathcal{X} \to \mathcal{S}$  be a flat  projective morphism with geometrically connected fibers.
Let $\mathcal{L}$ be a line bundle over~$\mathcal{X}$ that is very ample relative to~$\mathcal{S}$.
 Assume that
$R^m\pi_* \mathcal{L}=0$, for every $ m>0$.
Then the coherent sheaf $\mathcal{E}=\pi_* \mathcal{L}$ is locally free and the functor that carries
an $\mathcal{S}$-scheme $ \mathcal{T}$ to the group of automorphisms
$$
\phi\colon \mathcal{X}\times _\mathcal{S} \cT \to  \mathcal{X}\times _\mathcal{S} \cT
$$
over $\cT$,  such that Zariski locally over $\mathcal{T}$ the line bundle $\phi^*  \pr^*_\mathcal{X} \mathcal{L}$ is
isomorphic to $\pr^*_\mathcal{X} \mathcal{L}$,  is representable by a closed group subscheme $\underline{\Aut}(\mathcal{X}, \mathcal{L})$  of  $\PGL(\mathcal{E})$.
\end{lemma}
\begin{proof}
The coherent sheaf  $\mathcal{E}$ is locally free by the  base change theorem (see e.g.~\mbox{\cite[\S5]{Mumford}}) and the vanishing of
the higher direct images $R^m\pi_* \mathcal{L}$. It also follows from the same theorem that, for every scheme $\cT$ over $\cS$
with a structure morphism~\mbox{$g\colon \cT\to \cS$},  we have
$$
g^* \mathcal{E} \iso    \pr_{\cT *} \pr^*_\mathcal{X} \mathcal{L}.
$$

Let $i\colon\mathcal{X} \hookrightarrow \Proj (S^\bullet  \mathcal{E})$ be the embedding defined by $\mathcal{L}$.
We claim that  the functor carrying an $\mathcal{S}$-scheme $\cT$ to the subgroup $\mathcal{G}(\mathcal{T}) \subset \PGL(\mathcal{E})(\cT)$  of automorphisms
of~\mbox{$\Proj (S^\bullet  \mathcal{E}) \times _\cS \cT$} preserving the closed subscheme
$$
\mathcal{X}\times _\mathcal{S} \cT
\lhook\joinrel\xrightarrow{\ i \times \mathrm{Id}\ }
\Proj(S^\bullet  \mathcal{E}) \times_\cS \cT
$$
is represented by
a  closed subgroup subscheme~$\mathcal{G}$ of~$\PGL(\mathcal{E})$.
Indeed, the image of~$i$ is given by a sheaf  $I_\bullet$ of graded ideals in  $S^\bullet  \mathcal{E}$. For sufficiently large $m$,
the $\mathcal{O}_S$-submodule~\mbox{$I^m \subset S^m\mathcal{E}$}  is locally a direct summand. Thus, $I^m$ determines an $\cS$-point of the relative Grassmannian associated to $S^m  \mathcal{E}$. The group $\mathcal{G}(\mathcal{T})$ consists  of  the elements of $\PGL(\mathcal{E})(\mathcal{T})$ which preserve $I_m\subset S^m \mathcal{E}$  for all
sufficiently large~$m$.
Thus, our subgroup~\mbox{$\mathcal{G} \subset \PGL(\mathcal{E})$} can be defined as the stabilizer of a point in the (infinite) product
of relative Grassmann schemes~\mbox{$\text{Gr}(S^m(\mathcal{E}))$} under the action of~\mbox{$\PGL(\mathcal{E})$}.
Alternatively, $\mathcal{G}$ can be described as the stabilizer in  $\PGL(\mathcal{E})$
of the  point of the Hilbert scheme  $\text{Hilb}( \Proj (S^\bullet  \mathcal{E}))$ determined by the
embedding~\mbox{$i\colon\mathcal{X} \hookrightarrow \Proj (S^\bullet  \mathcal{E})$}.

We claim that $\mathcal{G}$ represents
the functor defined in the assertion of the lemma, that is, one has   $\underline{\Aut}(\mathcal{X}, \mathcal{L})\cong \mathcal{G} $.
Let $\cT$ and $\phi$ be as in  the assertion of the lemma, and let~\mbox{$g\colon \cT \to \cS$} be the structure morphism. Choose, Zariski locally on $\cT$, an isomorphism
$$
\alpha\colon \phi^*  \pr^*_\mathcal{X} \mathcal{L}\cong \pr^*_\mathcal{X} \mathcal{L}.
$$
Then $\phi$ and $\alpha$  induce an automorphism
of the $\mathcal{O}_\cT$-module  $g^* \mathcal{E}$ and, thus, a $\cT$-valued point of $\underline{\Aut}(\mathcal{X}, \mathcal{L})$. Since the fibers of $\pi$ are  geometrically connected,
using the Theorem on  Formal Functions  (see~\cite[Theorem~4.1.5]{EGA3-1}) we conclude
that
$$
(\pi \times \mathrm{Id})_*\mathcal{O}_{\mathcal{X}\times_\mathcal{S} \cT}= \mathcal{O}_\cT.
$$
Thus, $\alpha$ is well defined locally on $\cT$ up to multiplication by an invertible function. Hence, the $\cT$-valued point of $\underline{\Aut}(\mathcal{X}, \mathcal{L})$ does not depend on the choice of $\alpha$ and is defined on the whole~$\cT$. We leave it to the reader to check that the constructed morphism of functors is an isomorphism.
\end{proof}

We will need the following boundedness result on the Picard groups.

\begin{lemma}\label{lemma:Pic-bounded}
Let $\mathcal{S}$ be a  Noetherian   scheme, and let
$\pi\colon \mathcal{X} \to \cS$  be a flat  projective morphism with  geometrically reduced and geometrically irreducible
fibers. Suppose that for a fiber $\cX_s$ over every point $s\in\cS$ 
one has~\mbox{$h^1(\mathcal{O}_{\cX_s})=0$}. Then the minimal number
of generators of $\Pic(\cX_s)$ and the order of the torsion subgroup in
$\Pic(\cX_s)$ are both bounded by a constant that depends only on~$\cX$.
\end{lemma}

\begin{proof}
Let $\bar s$  be a geometric point of $\cS$ lying over $s$ (in particular, if the residue field of
$s$ is algebraically closed, then $\bar{s}=s$).
Note that the pullback morphism ~\mbox{$\Pic(\cX_s) \to \Pic(\cX_{\bar s})$}  is injective
(see for instance~\mbox{\cite[Lemma~32.30.3]{STACK}}).
 Next, by a result of Grothendieck~\cite[Theorem~9.4.8]{Kleiman2}, the group~\mbox{$\Pic(\cX_{\bar s})$}
can be identified with the group of closed points of the Picard group
scheme~\mbox{$ \underline{\Pic}(\cX_{\bar s})$}. By~\cite[Theorem~9.5.11]{Kleiman2},
the Lie algebra of~\mbox{$\underline{\Pic}(\cX_{\bar s})$}
is isomorphic to~\mbox{$H^1(\cX_{\bar s}, \mathcal{O}_{\cX_{\bar s}})$}. The latter space is zero by our assumption.
Therefore, the homomorphism
from the Picard group of the fiber $\cX_s$ to the group of connected components of the Picard scheme of
the corresponding geometric fiber $\cX_{\bar s}$ is injective. Thus, our result
follows from the boundedness
of the order of the latter group~\cite[Theorem~5.1]{Kleiman1}.
\end{proof}

In a particular case when $\mathcal{S}$ is a point, Lemma~\ref{lemma:Pic-bounded} gives a well-known assertion that
the Picard group of any 
projective variety $X$ with 
$h^1(\mathcal{O}_{X})=0$ is finitely generated.

\begin{lemma}[{cf. the proof of~\cite[Lemma~2.5]{MZ}}]
\label{lemma:q-0}
Let $X$ be a geometrically reduced and geometrically irreducible projective variety over an arbitrary field.
Suppose that~\mbox{$h^1(\mathcal{O}_{X})=0$}.
Let $\Gamma_{\Pic}$ be the kernel of the action of $\Aut(X)$ on
the Picard group~\mbox{$\Pic(X)$}. Then there is a constant $\Xi$ that depends only on the (minimal) number
of generators of $\Pic(X)$ and the order of the torsion subgroup in
$\Pic(X)$, such that for every
finite subgroup $G\subset\Aut(X)$ the intersection $G\cap \Gamma_{\Pic}$ has index at most
$\Xi$ in~$G$.
\end{lemma}
\begin{proof}
In our case $\Pic(X)$ is a finitely generated abelian group by Lemma~\ref{lemma:Pic-bounded}.
Thus, Theorem~\ref{theorem:Minkowski}
implies that the automorphism group of~\mbox{$\Pic(X)$}
has bounded finite subgroups. This
means that the orders of all finite subgroups in the quotient
group
$$
\Aut(X)/\Gamma_{\Pic}\hookrightarrow \Aut\big(\Pic(X)\big)
$$
are also bounded by some constant~$\Xi$ that depends only on~\mbox{$\Pic(X)$}.
\end{proof}

Now we are ready to prove Theorem~\ref{theorem:family}.

\begin{proof}[Proof of Theorem~\ref{theorem:family}]
Since $\mathcal{S}$ is Noetherian, arguing by induction it suffices to prove the theorem for
the restriction of $\pi$ to a non-empty open subscheme of $\mathcal{S}$. Thus, by generic flatness
(see~\cite[Theorem~5.12]{FAG}),
we may assume
that $\pi$ is flat. Choose a line bundle~$\mathcal{L}$ on~$\mathcal{X}$ that is very ample relative to~$\mathcal{S}$. Replacing if necessary  $\mathcal{L}$  by  its power~$\mathcal{L}^N$,  we may assume that the higher direct images
$R^j\pi_* \mathcal{L}$ vanish for all $j>0$.  Then, by Lemma ~\ref{lemma:ample},  the coherent sheaf
$\mathcal{E}=\pi_* \mathcal{L}$ is free of some rank $m$
over~$\mathcal{O}_{\mathcal{S}}$.  Let
$$
\underline{\Aut}(\mathcal{X}, \mathcal{L}) \subset \PGL(\mathcal{E})\cong \PGL_m \times \mathcal{S}
$$
be the closed subgroup scheme from  Lemma~\ref{lemma:ample}. By generic flatness we may assume that $\underline{\Aut}(\mathcal{X}, \mathcal{L})$ is flat over $\mathcal{S}$.  By~\cite[Lemma~36.25.7]{STACK}
the number of connected
components of the fibers of the morphism $\underline{\Aut}(\mathcal{X}, \mathcal{L})\to \mathcal{S}$ is bounded.
In fact, it follows from~\mbox{\cite[Lemma~36.25.7]{STACK}} that after  shrinking  $\mathcal{S}$ we have
a short exact sequence
\begin{equation}\label{eq:Aut-o}
1\to  \underline{\Aut}^\circ (\mathcal{X}, \mathcal{L}) \to \underline{\Aut} (\mathcal{X}, \mathcal{L}) \to \mathcal{C} \to 1,
\end{equation}
where $\underline{\Aut}^\circ (\mathcal{X}, \mathcal{L})$ is a flat group scheme with connected fibers
and  $\mathcal{C}$ is a finite \'etale group scheme over $\mathcal{S}$.

Let $Y$ be a variety over $\KK$ such that $Y_{\bar{\KK}}\cong \mathcal{X}_s$, for some  $s\in\mathcal{S}(\bar{\KK})$. Since  $Y_{\bar{\KK}}$ is projective,
the variety~$Y$ is projective as well. Pick a very ample line bundle $L$ over $Y$. We warn the reader that the pullback of $L$ to
$ \mathcal{X}_s$ need not be isomorphic to $ \mathcal{L}_s$.
Denote by~\mbox{$\underline{\Aut}(Y, L)$} the group scheme of automorphisms of the pair $(Y, L)$ and let  $\underline{\Aut}^\circ (Y, L)$ be the connected component of the identity automorphism. Assuming that $Y$ is not birational to $\PP^1 \times Y'$,
the reduced subgroup $\underline{\Aut}^\circ (Y, L)_{\mathrm{red}}$ is a connected anisotropic reductive group  by  Corollary~\ref{corollary:unipotent-action} and  Lemma \ref{lemma:quotient-by-Gm}. On the other hand, we claim that
\begin{equation}\label{equation:irrend}
 \underline{\Aut}^\circ (Y, L)_{\bar{\KK}}\cong \underline{\Aut}^\circ (\mathcal{X}, \mathcal{L})_s = \underline{\Aut}^\circ (\mathcal{X}_s, \mathcal{L}_s).
 \end{equation}
Indeed, since the group scheme $\underline{\Aut}^\circ (Y, L)_{\bar{\KK}}$ is connected and the Picard scheme of $\mathcal{X}_s$ is discrete, we see that the action of $\underline{\Aut}^\circ (Y, L)_{\bar{\KK}}$ on $Y_{\bar{\KK}}= \mathcal{X}_s$ induces the trivial action
on the Picard scheme of $\mathcal{X}_s$. This gives a morphism $\underline{\Aut}^\circ (Y, L)_{\bar{\KK}}\to \underline{\Aut}^\circ (\mathcal{X}_s, \mathcal{L}_s)$. Switching the roles of $Y_{\bar{\KK}}$ and  $\mathcal{X}_s$, we see that this morphism is an isomorphism. In particular, the rank of the reductive group  $\underline{\Aut}^\circ (Y, L)_{\mathrm{red}}$ is bounded by the dimension of fibers  of the flat group scheme  $\underline{\Aut}^\circ (\mathcal{X}, \mathcal{L})$.
Hence,  by Theorem~\ref{theorem:LAG}(i) the orders of finite subgroups of~\mbox{$\underline{\Aut}^\circ (Y, L)(\KK)$} are bounded by some constant which depends only on that dimension.

It remains to show that, for every finite subgroup $G\subset\Aut(Y)$, the index of~\mbox{$G\cap  \underline{\Aut}^\circ (Y, L)(\KK)$} in $G$ is bounded by some universal constant. This comes as follows.
Let~$\Gamma_{\Pic}(\mathcal{X}_s)$ be the kernel of the action of $\Aut(\mathcal{X}_s)$
on $\Pic(\mathcal{X}_s)$.
Then $\Gamma_{\Pic}(\mathcal{X}_s)$ is a subgroup of~\mbox{$\underline{\Aut} (\mathcal{X}, \mathcal{L})_s(\bar{\KK})$}. In particular, by Lemmas~\ref{lemma:Pic-bounded}
and~\ref{lemma:q-0}
there exists a constant~$\Theta$ such that for every finite subgroup $G\subset\Aut(\mathcal{X}_s)$,
the index of the intersection~\mbox{$G\cap \underline{\Aut} (\mathcal{X}, \mathcal{L})_s(\bar{\KK})$} in $G$ is at most~$\Theta$.
 Thus, using the exact sequence~\eqref{eq:Aut-o}, one concludes that the index of~\mbox{$G\cap \underline{\Aut}^\circ (\mathcal{X}, \mathcal{L})_s(\bar{\KK})$} in $G$ is bounded by
  $|\mathcal{C}|\Theta$, and so we win thanks to isomorphism~\eqref{equation:irrend}.
 \end{proof}

\section{Del Pezzo surfaces}
\label{section:DP}

In this section we study groups acting on del Pezzo surfaces
and prove Proposition~\ref{proposition:dP}.

By a del Pezzo surface we always mean a smooth projective geometrically irreducible
surface with an
ample anticanonical divisor.
By the \emph{degree} of a del Pezzo surface $X$ we
mean its anticanonical degree~$K_X^2$. One has $1\le d \le 9$.

\begin{lemma}\label{lemma:sep}
Let $\KK$ be any field. Let $X$ and $Y$ be del Pezzo surfaces
over $\KK$ such that~\mbox{$X_{\bar{\KK}}\cong Y_{\bar{\KK}}$}. Then we have~\mbox{$X_{\KK ^{sep}}\cong Y_{\KK^{sep}}$}.
\end{lemma}

\begin{proof}
First, we show that, for any  del Pezzo surface $X$  over $\KK$,  the group scheme $\underline{\Aut}(X)$ is smooth. Indeed, it suffices to check that
 $\underline{\Aut}(X)_{\bar{\KK}}\cong \underline{\Aut}(X_{\bar{\KK}})$ is smooth. The Zariski tangent space to $ \underline{\Aut}(X_{\bar{\KK}})$ at the identity is
 the vector space  $H^0(X_{\bar{\KK}}, T_{X_{\bar{\KK}}})$ of global vector fields on $X_{\bar{\KK}}$. Thus, smoothness of $ \underline{\Aut}(X_{\bar{\KK}})$  is equivalent to the inequality
\begin{equation}\label{eq-enp}
\dim H^0(X_{\bar{\KK}}, T_{X_{\bar{\KK}}})\le  \dim \underline{\Aut}(X_{\bar{\KK}}).
\end{equation}
(Note that the opposite inequality
$$\dim H^0(X_{\bar{\KK}}, T_{X_{\bar{\KK}}})\ge  \dim \underline{\Aut}(X_{\bar{\KK}})$$ holds for trivial reasons.)
Recall that every del Pezzo surface  of degree $d$  over $\bar{\KK}$ is isomorphic either to $\PP^1 \times \PP^1$,
or to the blow up of $\PP^2$ in $9-d$ points in general position.
If $X_{\bar{\KK}}$ is a blow up of $\PP^2$, then
$H^0(X_{\bar{\KK}}, T_{X_{\bar{\KK}}})$ is identified with the space
of vector fields on $\PP^2$
vanishing at  the points that are being blown up.
In particular, if $d< 6$, then~\mbox{$H^0(X_{\bar{\KK}}, T_{X_{\bar{\KK}}})=0$}, because
the second Chern class of  $T_{\PP^2}$ equals $3$, and so there are no vector fields on $\PP^2$
vanishing at  $9-d>3$ points.
If $X_{\bar{\KK}}$ is a blow up of $\PP^2$ in at most~$3$ points, then
up to a change of coordinates there is a unique choice of such triple of points (respectively, pair of points,
or a point), and it is straightforward to check the inequality~\eqref{eq-enp}. Similarly,
if~\mbox{$X_{\bar{\KK}}\cong\PP^1\times\PP^1$},  the inequality~\eqref{eq-enp} is obvious.

Now, given $X$ and $Y$ subject to the assumptions of the lemma,
consider the scheme of ismorphisms~\mbox{$Z=\underline{\Iso}(X, Y)$}. Being a torsor over a smooth group scheme $\underline{\Aut}(X)$
the scheme $Z$ is smooth over $\KK$.
Hence, $Z$  has a $\KK^{sep}$-point, which means
that~\mbox{$X_{\KK ^{sep}}\cong Y_{\KK^{sep}}$}.
\end{proof}

Now we deal with del Pezzo surfaces of degree $8$.
We will say that a del Pezzo surface~$X$ of degree $8$
over a field $\KK$
is \emph{of product type},
if~\mbox{$X_{\bar{\KK}}\cong\PP^1\times\PP^1$}.

Given a finite field extension $\KK \subset \LL$ and a scheme $Y$ over $\LL$ denote by  $R_{\LL/\KK} Y$  its Weil restriction of scalars. Recall that $R_{\LL/\KK} Y$ is a scheme over $\KK$ which represents the functor sending a $\KK$-scheme $X$
to the set of morphisms $X_\LL= X \times \Spec \LL \to Y$ over $\LL$. In  other words, the functor $R_{\LL/\KK}$
from the category of schemes over $\LL$ to the category of schemes over $\KK$ is
right adjoint to the fiber product with $ \Spec \LL$
(that is, to the functor of extension of scalars to~$\LL$).

The first assertion of the following lemma is borrowed from the proof of 
\cite[Proposition~5.2]{CTKM08}.

\begin{lemma}
\label{lemma:DP8ptnew}
Let $\KK$ be a field, and let  $X$ be a del Pezzo surface  of degree
$K_X^2=8$ of product type over $\KK$.
 The following assertions hold.
\begin{itemize}
\item[(i)] The surface $X$ is  isomorphic either to the product  $C\times C'$
of two conics over~$\KK$, or to  $R_{\LL/\KK} Q$, where $\LL \supset \KK$ is a separable quadratic extension of $\KK$  and $ Q$ is a conic over $\LL$.
In the former case one has $\Pic(X)\cong\Z\oplus\Z$,
while in the latter case~\mbox{$\Pic(X)\cong\Z$}.
Furthermore, in the former case the (non-ordered) 
pair of conics~\mbox{$\{C,C'\}$} 
is uniquely determined by~$X$; in the latter case 
the extension~\mbox{$\LL \supset \KK$} and 
the conic~$Q$ are uniquely determined by~$X$ up conjugation by 
the Galois group~\mbox{$\Gal(\LL/\KK)$}.

\item[(ii)]  The surface $X$ is  isomorphic to a  quadric in $\PP^3$
if and only if $X \cong  C\times C$ or~\mbox{$X \cong  R_{\LL/\KK}  C_\LL$},
for some conic $C$ over $\KK$.

\item[(iii)]
Let $\Aut^\circ (X)\subset \Aut(X)$ be the subgroup that consists 
of all automorphisms of $X$ which act trivially  
on $\Pic(X_{\KK^{sep}})$. 
Then  $\Aut^\circ (X)$
has index at most $2$ in $ \Aut(X)$. That is, we have a left exact sequence
\begin{equation*}
1\to  \Aut^\circ (X)\to \Aut(X) \to \Z/2\Z.
\end{equation*}
Moreover, one has that
$$
\Aut(C)\times \Aut(C') \iso\Aut^\circ (C\times C')
$$
for any conics $C$ and $C'$ over $\KK$, and
$$
\Aut(Q) \iso \Aut^\circ (R_{\LL/\KK} Q)
$$
for any conic $Q$ over $\LL$.

\item[(iv)]
The homomorphism $\Aut(X) \to \Z/2\Z$  is surjective  if and only if either~\mbox{$X \cong C \times C$}, where $C$ is a conic over $\KK$, or $X\cong R_{\LL/\KK} Q$,
where $Q$ is a conic over $\LL$ such that the $\Gal(\LL/\KK)$-conjugate conic $Q'$ is isomorphic to $Q$.
The homomorphism~\mbox{$\Aut(X) \to \Z/2\Z$}  is a split surjection, that is
$$
\Aut(X) \cong \Aut^\circ (X) \rtimes \Z/2\Z,
$$
if and only if $X$ is a quadric.
\end{itemize}
\end{lemma}

\begin{proof}
By Lemma  \ref{lemma:sep}, we have \mbox{$X_{\KK^{sep}}\cong\PP^1\times\PP^1$}.
The action of the Galois group~\mbox{$\Gal(\KK^{sep}/\KK)$}
on $\Pic(X_{\KK^{sep}})\cong \Z \oplus \Z$ is either trivial or factors through an order $2$ quotient
$$
\Gal(\KK^{sep}/\KK)\twoheadrightarrow \Gal(\LL/\KK)
$$
which acts on $ \Z \oplus \Z$  by permuting the factors.
In the former case the  classes~\mbox{$(1,0)$} and~\mbox{$(0,1)$} define two morphisms~\mbox{$X_{\KK^{sep}} \to \PP^1_{\KK^{sep}}$}, which descend to morphisms  $X\to C$ and~\mbox{$X\to C'$} from $X$ to some Severi--Brauer curves $C$ and $C'$.  
This yields an isomorphism~\mbox{$X\cong C\times C'$},
so that in particular $\Pic(X)\cong\Z\oplus\Z$. 
The surface~$X$ has exactly two
extremal contractions, and these contractions are projections to the 
factors~$C$ and~$C'$. In other words, the conics~$C$ and~$C'$ are uniquely 
determined by~$X$.

If the action of~\mbox{$\Gal(\KK^{sep}/\KK)$}
on $\Pic(X_{\KK^{sep}})$ is non-trivial, then we have~\mbox{$X_\LL \cong Q\times Q'$}, 
where~$Q$ and~$Q'$ are
$\Gal(\LL/\KK)$-conjugate conics.
By adjunction, the projections
$$
Q\times Q' \to Q \quad \text{and}\quad
Q\times Q' \to Q'
$$  
yield the morphisms
$$
X\to R_{\LL/\KK} Q \quad\text{and}\quad 
X\to R_{\LL/\KK} Q',
$$
which are isomorphisms. Note that in this case the 
sublattice of $\Gal(\KK^{sep}/\KK)$-invariant elements 
in~\mbox{$\Pic(X_{\KK^{sep}})$} 
has rank~$1$, which implies that $\Pic(X)\cong\Z$. 
Finally,  assume that 
$$
R_{\LL/\KK} Q_1 \cong R_{\LL/\KK} Q_2
$$ 
for some conics $Q_1$ and $Q_2$ over $\LL$. Then we have 
$$
Q_1 \times Q'_1 \cong (R_{\LL/\KK} Q_1)_{\LL} \cong 
(R_{\LL/\KK} Q_2)_{\LL} \cong Q_2 \times Q'_2,
$$ 
where $Q_i'$  is the conic $\Gal(\LL/\KK)$-conjugate to $Q_i$. 
Hence either $Q_1 \cong Q_2$, or $Q_1 \cong Q_2'$. 
This proves assertion~(i).

A necessary and sufficient condition for $X$ to be
isomorphic to a   quadric in $\PP^3$ is that the
$\Gal(\KK^{sep}/\KK)$-invariant class
$$
(1,1)\in \Pic(X_{\KK^{sep}})\cong \Z \oplus \Z
$$
is represented
by a line bundle over $X$.  If  $X \cong C\times C$,
the line bundle corresponding to the diagonal does the job. If~\mbox{$X\cong R_{\LL/\KK} C_\LL$},  then the
morphism~\mbox{$C \to  R_{\LL/\KK} C_\LL$}
defined by adjunction specifies a divisor  of bidegree  $(1,1)$ on $X$.
Conversely, let $X\in \PP^3$ be a quadric. Let $C\hookrightarrow X$ be a
smooth hyperplane section of $X$. If $X$ is isomorphic to the product
of two conics, then both conics must be isomorphic to $C$.
If $X\cong  R_{\LL/\KK} Q$, then the morphism~\mbox{$C_\LL \to Q$} given
by adjunction is an isomorphism.
This proves assertion~(ii).

Let us prove  assertion~(iii). The subgroup  $\Aut^\circ (X)\subset \Aut(X)$ has index at most~$2$ because this is true for $X$ replaced by  
$X_{\KK^{sep}}\cong\PP^1\times\PP^1$.
The claim that  the morphism
$$
\Aut(C)\times \Aut(C') \to \Aut^\circ (C\times C')
$$
is an isomorphism  also follows by the base change to $\KK^{sep}$. 
Finally, to show that
$$
\Aut(Q) \to \Aut^\circ (R_{\LL/\KK} Q)
$$
is an isomorphism we denote by $Q'$ the $\Gal(\LL/\KK)$-conjugate conic and observe that~\mbox{$\Aut(R_{\LL/\KK} Q)$} is isomorphic to the subgroup
of
$$
\Aut\big((R_{\LL/\KK} Q)_\LL\big)\cong \Aut(Q \times Q')
$$
consisting of $\Gal(\LL/\KK)$-invariant elements.
This proves assertion~(iii).

Let us prove  the first part of assertion~(iv) concerning the surjectivity of the map~\mbox{$\Aut(X) \to \Z/2\Z$}.
If $X$ is the product of two conics,
the proof is straightforward from the case of algebraically
closed base field. If  $X \cong R_{\LL/\KK} Q$, then the claim follows from the description of $\Aut(R_{\LL/\KK} Q)$  given in the proof of assertion~(iii).

For the second part  of assertion~(iv) we have to check that  $\Aut(X)$ contains an involution which acts non-trivially on  $\Pic(X_{\KK^{sep}})$ if and only if
$X$ is a quadric. Let $\phi$ be such an involution. Using the base change to $\bar {\KK}$, one checks that the fixed
point subscheme of $\phi$ is a divisor of bidegree~\mbox{$(1,1)$}. This implies that $X$ is a quadric. Conversely, let
$X\subset \PP^3$ be a quadric given by a quadratic form $q$ on a vector space $V$ of dimension $4$.
Let $B_q$ be the symmetric bilinear form associated to $q$ (see \S\ref{section:quadrics}).
Pick a $\KK$-point of $\PP^3-X$
and consider it as a $1$-dimensional subspace $\KK v $ of $V$.
Then the orthogonal reflection
$f\colon V\to V$  given by the formula
$$
f(u)= u - \frac{B_q(v,u)}{q(v)} v
$$
defines the desired involution of $X$.
This proves assertion~(iv).
\end{proof}

\begin{remark}
Proposition~5.2(2) from \cite{Liedtke} erroneously asserts that every  del Pezzo surface of degree $8$
with Picard group isomorphic to~$\Z$ is a quadric in~$\PP^3$.
\end{remark}

One simple consequence of Lemma~\ref{lemma:DP8ptnew} is the
following  well known result.

\begin{corollary}\label{corollary:DP8rational}
Let $\KK$ be a field, and let $X$ be a del Pezzo surface
of degree $8$ of product type over $\KK$.
Suppose that $X(\KK)\neq \varnothing$. Then $X$ is rational.
\end{corollary}

\begin{proof}
By Lemma~\ref{lemma:DP8ptnew}(i), the surface $X$ is isomorphic either
to a product $C\times C'$ of two conics, or to the Weil restriction of scalars
$R_{\LL/\KK} Q$ of a conic $Q$ defined over a separable quadratic
extension $\LL$ of $\KK$. In the former case both conics $C$ and $C'$
have $\KK$-points, so that $X\cong\PP^1\times\PP^1$.
In the latter case, the conic $Q$ has an $\LL$-point by adjunction,
so that~$Q$ is rational, and thus $X$ is rational as well.
\end{proof}

Another result implied by Lemma~\ref{lemma:DP8ptnew} is as follows.

\begin{corollary}\label{cor:DP8pt}
Let $\KK$ be a field that contains  all roots of~$1$, and let $X$ be a del Pezzo surface of degree
$8$ of product type over $\KK$.
Suppose that
$X$ is not rational and not isomorphic 
to~\mbox{$\PP^1\times C$}, where~$C$ is a conic. Then every finite subgroup of $\Aut(X)$ is a $2$-group. Moreover, 
if~$\KK$ is  either a perfect field, or a field of characteristic
different from~$2$, then
every finite subgroup of $\Aut(X)$ has order at most~$32$.
\end{corollary}
\begin{proof}
We know from Lemma~\ref{lemma:DP8ptnew}(i) that $X$ is isomorphic either
to a product of two conics, or to the Weil restriction of scalars
$R_{\LL/\KK} Q$ of a conic $Q$ defined over a separable quadratic
extension $\LL$ of $\KK$. In the former case the factors of $X$ 
are not isomorphic to $\PP^1$ by assumption, and thus 
have no $\KK$-points,  
so that everything follows
from Lemma~\ref{lemma:DP8ptnew}(iii)  
and Corollary~\ref{corollary:Zarhin-improved}.
In the latter case $X$ has no $\KK$-points,
since otherwise it is rational  
by Corollary~\ref{corollary:DP8rational}.
Thus the conic $Q$ has no $\LL$-points by adjunction, and again everything is implied by
Lemma~\ref{lemma:DP8ptnew}(iii)  and Corollary  \ref{corollary:Zarhin-improved}.
\end{proof}

Note that a quadric surface $X$ over a field $\KK$ 
such that $X(\KK)=\varnothing$ is not rational.
Moreover, it cannot be 
isomorphic to $\PP^1\times C$, where $C$ is a conic,
by Lemma~\ref{lemma:DP8ptnew}(i),(ii).
In particular, this means that Corollary~\ref{cor:DP8pt} proves assertion~(iv)
of Proposition~\ref{proposition:quadricnew}.

The following result is well known (see for instance \cite[Corollary~6.2]{Liedtke}).

\begin{lemma}\label{lemma:DP-deg-7-8}
Let $\KK$ be a field, and let $X$ be a
del Pezzo surface over $\KK$. Suppose that either $X$ has degree  $8$ and
is not of product type, or $X$ has degree $7$. Then $X$ is rational.
\end{lemma}

\begin{proof}
Let $X$ be such a surface. Using Lemma~\ref{lemma:sep}, it follows that
$X_{{\KK}^{sep}}$ is a
blow up of $\PP^2$ at one or  two ${\KK}^{sep}$-points.
Thus, in both cases,  $X_{{\KK}^{sep}}$ contains a Galois-invariant $(-1)$-curve. The image of this curve under the anticanonical
 embedding is a Galois-invariant line in the projective space. Therefore,
this  Galois-invariant $(-1)$-curve contains  a Galois-invariant point.
  Hence, $X$ is a blow up of a del Pezzo surface $X'$ of degree $9$ or~$8$ at a $\KK$-point.
 In the former case, $X'$ is a Severi--Brauer surface with a $\KK$-point, which implies that $X'\cong \PP^2$. In the latter case, $X'$ is   a del Pezzo surface of product type with a
$\KK$-point which is rational by Corollary~\ref{corollary:DP8rational}.
\end{proof}

Now we proceed to del Pezzo surfaces of degree $6$ and less.

\begin{lemma}\label{lemma:DP6}
Let $\KK$ be  any field that contains all roots of~$1$, and let $X$ be a del Pezzo surface of degree
$6$ over $\KK$.
Suppose that
$X$ is not birational to~\mbox{$\PP^1\times C$}, where $C$ is a conic.
Then every finite subgroup of $\Aut(X)$ has order at most~$432$.
\end{lemma}

\begin{proof}
Consider the group scheme $\underline{\Aut}(X)$.  One has
$$
\underline{\Aut}(X)_{\bar{\KK}}\cong \underline{\Aut}(X_{\bar{\KK}}) \cong (\G_m \times \G_m) \rtimes \mathrm{D}_{12},
$$
where $\mathrm{D}_{12}$ is the dihedral group of order $12$,
cf. \cite[Theorem~8.4.2]{Dolgachev-ClassicalAlgGeom}.
Thus, the group scheme ~\mbox{$\underline{\Aut}(X)$} contains  a two-dimensional  torus $T$ such that  the quotient~\mbox{$\underline{\Aut}(X)/T$} is a finite group scheme of order $12$.
If $T$ contains a subtorus isomorphic to $\Gm$,
then~$X$ is birational to~\mbox{$\PP^1\times C$}
for some conic~$C$ by Lemma~\ref{lemma:quotient-by-Gm},
which is not the case by assumption.
Thus we see that $T$ is anisotropic,
so that every finite subgroup of $T(\KK)$ has order at most $36$
by Lemma~\ref{lemma:tori-bounded-subgroups}.
Therefore, every finite subgroup of $\Aut(X)$
has order at most~\mbox{$12\cdot 36=432$}.
\end{proof}

\begin{remark}\label{remark:dP-low-degree}
Let $X$ be a del Pezzo surface of degree $d\le 5$ over a field $\KK$.
Then the group $\Aut(X)$ is finite. Moreover, in the case when $\Char\KK=0$,
one has~\mbox{$|\Aut(X)|\le 648$}. To show this one
can assume that $\KK$ is algebraically closed, and use an explicit classification
of automorphism groups of del Pezzo surafces of low degree, see~\mbox{\cite[\S8.5.4]{Dolgachev-ClassicalAlgGeom}},
\mbox{\cite[\S8.6.4]{Dolgachev-ClassicalAlgGeom}}, \cite[\S8.7.3]{Dolgachev-ClassicalAlgGeom}, \cite[\S8.8.4]{Dolgachev-ClassicalAlgGeom}, and \cite[\S9.5.3]{Dolgachev-ClassicalAlgGeom}.

If $\KK$ is an arbitrary field,
the known upper bound is much higher.
Namely, one can check that $|\Aut(X)|$ is bounded by the order of the Weyl
group $\mathrm{W}(\mathrm{E}_8)$, that is, by the
number~$696\,729\,600$, see~\mbox{\cite[Corollary~8.2.40]{Dolgachev-ClassicalAlgGeom}}
(where this is claimed for del Pezzo surfaces over fields of characteristic zero, but the
proof does not depend on the characteristic).
This bound is not sharp. However, over appropriate fields there exist del Pezzo surfaces
of low degree with automorphism groups of rather large order.
For instance, the automorphism group of the Fermat cubic surface over an algebraically closed field of
characteristic~$2$ has order~$25\,920$, see~\cite[Table~1]{DolgachevDuncan}.
\end{remark}

Now we are ready to prove Proposition~\ref{proposition:dP}.

\begin{proof}[Proof of Proposition~\ref{proposition:dP}]
As usual, we set $d=K_X^2$.

If $d=9$, then $X$ is a Severi--Brauer surface
without $\KK$-points. If $\Char \KK \ne 3$, then by Proposition~\ref{proposition:SB}(ii) we have $|G|\le 81$.
If $\Char \KK =3$ (in this case $\KK$ must be non-perfect by
Remark~\ref{remark:char-vs-dim-perfect}), then by
Proposition~\ref{proposition:SBp} the group
$G$ is isomorphic to  $(\Z/3\Z)^n $ for some positive integer $n$, so that~\mbox{$|G|'=1$}.

If $7\le d\le 8$, then it follows from Lemma~\ref{lemma:DP-deg-7-8} that $d=8$ and $X$ is
of product type.  Thus,   by Corollary~\ref{cor:DP8pt} we have $|G|\le 32$
unless $\KK$ is an non-perfect field of characteristic~$2$. If $\KK$ is an non-perfect field of characteristic $2$,
again by  Corollary~\ref{cor:DP8pt}
the group $G$ is a $2$-group, so that~\mbox{$|G|'=1$}.

Finally, if $d=6$, then $|G|\le 432$ by Lemma~\ref{lemma:DP6}.
\end{proof}

\begin{remark}\label{remark:dP-char-2-3}
In the notation of Proposition~\ref{proposition:dP},
suppose that $\Char \KK$ equals either~$2$ or~$3$. Then one can
improve the bound for $|G|'$ provided by Proposition~\ref{proposition:dP}(ii).
Namely, if~\mbox{$\Char\KK=2$}, one can
show that $|G|' \le 108$. If $\Char \KK=3$, then
we have~\mbox{$|G|' \le 48$}.
\end{remark}

\section{Conic bundles}
\label{section:conic-bundles}

In this section we study groups acting on conic bundles. Our results  are summarized in Proposition  \ref{proposition:conicsummary} that plays a key role
in the proof of  Theorem~\ref{theorem:main}.

We begin by recalling some basic definitions concerning conic bundles;
we refer the reader to~\cite[\S3]{Iskovskikh80} for the proofs.
By a \emph{conic bundle}
we will mean a smooth geometrically irreducible surface~$X$  over a field $\KK$ together with
a proper surjective morphism~\mbox{$\phi\colon X\to C$}
to a smooth curve $C$ whose fiber $X_\eta$ over the scheme-theoretic generic point $\eta$ of $C$ is  smooth
and geometrically  irreducible, and such that the anticanonical line
bundle~$\omega^{-1}_X$ is very ample relative to $C$.
In other words, we have an embedding~\mbox{$X\hookrightarrow \PP_C(\phi_* \omega^{-1}_X)$},  and an isomorphism between~$\omega^{-1}_X$ and the restriction of
 $\mathcal{O}_{\PP_C(\phi_* \omega^{-1}_X)}(1)$.
It follows that all the fibers of~$\phi$ are geometrically reduced, the scheme-theoretic general fiber $X_\eta$ is a conic  
over the field $\KK(\eta)$ of rational functions on $C$, and every singular  geometric fiber is isomorphic to a pair of lines in $\PP^2$ meeting at one point. The \emph{singular locus} of $\phi$,
i.e., the closed subscheme of $X$ defined by the sheaf of
ideals  $\Ann(\coker(T_X \to \phi^* T_C))$, where~\mbox{$T_X \to \phi^* T_C$}
is the differential of $\phi$,  is smooth of dimension zero  over $\KK$, which means
reduced and supported on finitely many points with separable  residue fields.

By the \emph{discriminant locus}
$\Omega \subset C$ of $\phi$ we mean the image of the singular locus of $\phi$.
This is a smooth zero-dimensional subscheme of $C$.
The smooth conic bundle
$$
X - \phi^{-1} (\Omega) \to C - \Omega
$$
is a Severi--Brauer scheme over  $C - \Omega$; the latter are in one-to-one correspondence
with rank $4$ Azumaya algebras  over   $C - \Omega$. We denote by
$$
[X - \phi^{-1} (\Omega)] \in \Br(C-\Omega)
$$
the class in the Brauer group of the Azumaya algebra corresponding to  $X - \phi^{-1} (\Omega)$.

We say that $\phi \colon X\to C$ is \emph{relatively minimal}
if it does not admit a non-trivial factorization $X\to Y \to C$, where  $Y$ is a smooth surface. If $\phi$ is not relatively minimal
and~\mbox{$X\to Y \to C$} is  a  factorization as above, then
 $Y\to C$
is also a conic bundle.
A  conic bundle is relatively minimal  if and only if  the discriminant locus of $\phi$ coincides with  the discriminant locus of
the class $[X_\eta] \in \Br(\KK(\eta))$, that is,  the class~\mbox{$[X - \phi^{-1} (\Omega)] \in \Br(C-\Omega)$}  is not in the image of the (injective) restriction
map~\mbox{$\Br(U) \to   \Br(C-\Omega)$},
for any open $U\subset C$ strictly larger than $C-\Omega$.
Also, $\phi$ is relatively minimal if and only
if~\mbox{$\rk\Pic(X)-\rk\Pic(C)=1$}.

For a scheme $Y$ over $\KK$, we  set
$$
\Br ^{\mathrm{t}}(Y)= \ker\big(\Br(Y)\to \Br(Y_{\KK^{sep}} )\big),
$$
where the homomorphism of the Brauer groups is induced by the projection~\mbox{$Y_{\KK^{sep}} \to Y$}.
Given a positive integer $l$, we write $\Br(Y)_l$
for the subgroup of $l$-torsion elements in~\mbox{$\Br(Y)$}.
Similarly, we write $\Br^{\mathrm{t}}(Y)_l$
for the subgroup of $l$-torsion elements in~$\Br^{\mathrm{t}}(Y)$.
Note that if the characteristic of $\KK$ does not divide $l$
then  $\Br^{\mathrm{t}}(Y)_l= \Br(Y)_l$.

\begin{lemma}\label{lemma:tameness}
Let $C$ be a smooth proper curve over a field $\KK$, and
let~\mbox{$\phi\colon X\to C$} be a conic bundle. Then~\mbox{$[X_\eta] \in \Br^{\mathrm{t}}(\KK(\eta))_2$}.
\end{lemma}

\begin{proof}
It suffices to check that if $\KK = \KK^{sep}$, then every
conic bundle~\mbox{$\phi\colon X\to C$} has a section.
To do this, we can assume that $\phi$ is relatively minimal.

First, let us check that $\phi$  is smooth over $C$. Indeed, let $X_c$ be a singular fiber. Then
the residue field of the point $c\in C$ is $\KK$, because the discriminant locus of $\phi$
is smooth over~$\KK$. Hence, $X_c$ is a degree $2$ curve in $\PP^2_\KK$ whose base change to $\bar{\KK}$ is reducible.
Therefore, $X_c$ is itself reducible
(see, for example,~\mbox{\cite[Lemma~32.8.2]{STACK}}),
i.e.,  a pair of lines in $\PP^2_\KK$ meeting at one point.
Each of these lines is a $(-1)$-curve on $X$. It follows that  $\phi\colon X\to C$ is not  relatively minimal.

Thus, ~\mbox{$\phi\colon X\to C$} is a Severi--Brauer curve over $C$.
Since $\KK=\KK^{sep}$ is separably closed,
the Brauer group of $C$ is trivial by~\cite[Corollary 5.7]{GrBr3}. Thus,
$\phi$ has a section.
\end{proof}

\begin{remark}
Lemma~\ref{lemma:tameness}
shows that in characteristic $2$ not every conic $X_\eta$
over the scheme-theoretic generic point $\eta$ of $C$ can be extended to a conic bundle over $C$.
This seems to contradict~\mbox{\cite[Theorem~3(4)]{Iskovskikh80}}; see also \cite[pp.~33--34]{Iskovskikh80},
and especially formula~(3.3) therein.
For example, take~\mbox{$\KK =\overline{\mathbf{F}}_2(a)$}, where $a$ is a transcendental variable,
and consider the smooth conic bundle  $X_{\A^1}  \to \A^1$  given by equation
$$
ax^2 +xy +ty^2 +z^2=0
$$
in $\PP^2 \times \A^1$,
where $t$ is the coordinate on $\A^1$ and $(x:y:z)$
are homogeneous coordinates on~$\PP^2$.
The class
$$
[X_\eta]\in \Br(\A^1) \subset  \Br(\KK(t))
$$
is represented
by the division algebra over~\mbox{$\KK(t)= \overline{\mathbf{F}}_2(a,t)$}
from Example~\ref{example:BMR} and, therefore, is not equal to $0$.
However, the Leray spectral sequence
$$
E_2^{ij}=H^i\big(\Gal  (\KK^{sep}/\KK),
H^j(\A^1_{\KK^{sep}, \et}, \cO^*_{\A^1_{\KK^{sep}}})\big)
$$
converging to
$H^{\bullet}(\A^1_{\KK, \et}, \cO^*_{\A^1_{\KK }})$ shows that
$\Br^{\mathrm{t}}(\A^1_\KK)=\Br(\KK)$. On the other hand, we have~\mbox{$\Br(\KK)=0$} by Tsen's theorem.
This means that the group $\Br^{\mathrm{t}}(\A^1_\KK)$ is trivial, and thus~\mbox{$[X_\eta]\not\in\Br^{\mathrm{t}}(\A^1_\KK)$}.
Since
$$
\Br^{\mathrm{t}}(\A^1_\KK)=\Br^{\mathrm{t}}(\KK(t))\cap \Br(\A^1_\KK)\subset\Br(\KK(t)),
$$
we conclude that $[X_\eta]\not\in\Br^{\mathrm{t}}(\KK(t))$.
Now Lemma~\ref{lemma:tameness} shows that $X_{\A^1}$ and $X_\eta$  do  not extend to conic bundles over $\PP^1$.

However, one can show  that  there is a one-to-one correspondence between relatively minimal conic bundles over $C$ with discriminant $\Omega$
and those  Severi--Brauer schemes~$X_{C-\Omega}$ of dimension $1$ over  $C - \Omega$
whose class in the Brauer group belongs to~\mbox{$\Br^{\mathrm{t}}(C-\Omega)$}
and cannot be extended to a class in the Brauer group of a larger
subscheme of $C$. We will not prove this claim  here because we do not need it in the full generality, but
we will check  its special case in Lemma~\ref{lemma:extendingconicbundles}
below.
\end{remark}

\begin{lemma}\label{lemma:extendingconicbundles}
Let $C$ be a smooth proper curve over a perfect field $\KK$ of characteristic~$2$,
and let~\mbox{$U\subset C$} a dense open subscheme. Then
every conic bundle $\phi\colon X\to U$ over $U$ can be extended to a conic bundle over
$C$.
\end{lemma}
\begin{proof}
It suffices to prove the following local statement: let $Y$ be a conic over the field of Laurent series $\KK((t))$, where $\KK$ is  a perfect field of characteristic $2$. Then $Y$ can be
extended to a conic bundle over $\Spec \KK[[t]]$, that is, there exists a connected regular closed  subscheme $\bar Y\subset \PP^2_{\KK[[t]]}$ whose fiber over $\KK((t))$ is $Y$.

To check the local statement recall that by a theorem of Witt (see, for example,~\mbox{\cite[p.~188]{GrBr3}})
the Brauer group of~$\KK((t))$
fits into the short exact sequence
$$
0\to  \Br(\KK)\to   \Br\big(\KK((t))\big) \to \Hom\big(\Gal(\bar{\KK}/\KK), \Q/\Z\big) \to 0.
$$
Since the multiplication by $2$ is an isomorphism  on  $\Br(\KK)$ (see Remark \ref{remark:char-vs-dim-perfect}) we conclude that
$$
\Br\big(\KK((t))\big)_2 \cong \Hom\big(\Gal(\bar{\KK}/\KK), \Z/2\Z\big).
$$
The above isomorphism is compatible with algebraic extensions of $\KK$, that is, for every extension $\KK \subset \LL\subset \bar{\KK}$, we have a commutative diagram
\begin{equation*}
\def\normalbaselines{\baselineskip20pt
\lineskip3pt  \lineskiplimit3pt}
\def\mapright#1{\smash{
\mathop{\to}\limits^{#1}}}
\def\mapdown#1{\Big\downarrow\rlap
{$\vcenter{\hbox{$\scriptstyle#1$}}$}}
\begin{matrix}
 \Br\big(\KK((t))\big)_2  & \cong &  \Hom\big(\Gal(\bar{\KK}/\KK), \Z/2\Z\big)  \cr
 \mapdown{}  && \mapdown{}  \cr
  \Br\big(\LL((t))\big)_2  & \cong &  \Hom\big(\Gal(\bar{\KK}/\LL), \Z/2\Z\big)
\end{matrix}
\end{equation*}
where the left vertical arrow is the pullback homomorphism and the right vertical arrow is the restriction homomorphism.
It follows that  non-zero elements of $\Br(\KK((t)))_2$ are in one-to-one correspondence with quadratic extensions  $\KK \subset \LL\subset \bar{\KK}$ so that, for every such extension, there
exists a unique non-zero class in  $\Br(\KK((t)))_2$  which splits over $\LL((t))$.

Now we claim that every conic over $\KK((t))$
can be given in $\PP^2_{\KK((t))}$ by a homogeneous equation of the form
\begin{equation}\label{equation:conicext}
x^2 +xy +ay^2 + t z^2=0
\end{equation}
for some $a\in \KK$. Indeed, one can easily see that the above conic has no $\KK((t))$-points if and only if the quotient ring $\LL=\KK[h]/(h^2-h -a)$ is a field, and that in this case
the conic has a point over~$\LL$.
Hence, since every quadratic extension of $\KK$ has the form~\mbox{$\KK[h]/(h^2-h -a)$} for some $a\in \KK$,
every class in $ \Br(\KK((t)))_2$ is represented by a conic of the form~\eqref{equation:conicext},
and this proves the claim.

It remains to observe that equation~\eqref{equation:conicext} defines a regular subscheme of
 $\PP^2_{\KK[[t]]}$.
\end{proof}

\begin{lemma}\label{lemma:perfection}
Let   $\phi \colon X\to C$ be a  relatively minimal   conic bundle over a field $\KK$. Then, for every purely inseparable algebraic extension $\KK \subset \LL$, the conic bundle $\phi_\LL \colon X_\LL\to C_\LL$
obtained from $\phi$ by the base change is  relatively minimal.
\end{lemma}
 \begin{proof}
A conic bundle $\phi \colon X\to C$ is  relatively minimal if and only if,
for every point~\mbox{$c\in C$}, the fiber $X_c$ is irreducible. Now the claim follows from a general fact that,
for an irreducible scheme $Y$ over a field $\KK'$ and a  purely inseparable
algebraic extension $\KK' \subset \LL'$, the scheme~$Y_{\LL'}$ is irreducible (see, for example, ~\cite[Lemma 10.45.7]{STACK}).
\end{proof}

\begin{lemma}\label{lemma:CB-1-fiber}
Let $C$ be a smooth proper curve over a field $\KK$, and
let~\mbox{$\phi\colon X\to C$ be a conic bundle.
Suppose that} $\phi_{\bar{\KK}}\colon X_{\bar{\KK}}\to C_{\bar{\KK}}$ has a unique reducible fiber.
Then $\phi\colon X\to C$ is not relatively minimal.
\end{lemma}
\begin{proof}
Assume that $\phi$ is relatively minimal.
Replacing $\KK$ by its perfection (that is, by the maximal subfield of
$\bar{\KK}$ which is purely inseparable over $\KK$)
and using  Lemma \ref{lemma:perfection} we may assume that $\KK$ is perfect.  By our assumptions the discriminant of $\phi$ consists of a  single point  $c$ whose residue field is $\KK$.
Set $U=C-c$.
We claim that the restriction morphism~\mbox{$\Br(C) \to  \Br(U)$}
is an isomorphism. This is well known, but we sketch the argument
for the reader's convenience.
	
By Tsen's theorem the complexes $R\Gamma(C_{\bar{\KK}, \et},  \cO^*_{C_{\bar{\KK}}})$ and  $R\Gamma(U_{\bar{\KK}, \et},  \cO^*_{U_{\bar{\KK}}})$ of sheaves for the \'etale topology  on $ \Spec \KK $ have non-trivial cohomology sheaves in degrees $0$ and $1$ only (see, for example,~\cite[Theorem 1.1]{GrBr3} and \cite[Proposition 2.1]{GrBr3}). It follows that the restriction morphism
$$
R\Gamma(C_{\bar{\KK}, \et},  \cO^*_{C_{\bar{\KK}}}) \to R\Gamma(U_{\bar{\KK}, \et},  \cO^*_{U_{\bar{\KK}}})
$$
together with
the projection on the top cohomology group
$$
R\Gamma(C_{\bar{\KK}, \et},  \cO^*_{C_{\bar{\KK}}}) \to \Pic C_{\bar{\KK}} [-1]
$$
composed with the degree map
$$
\Pic C_{\bar{\KK}} [-1] \to  \Z[-1]
$$
induce a quasi-ismorphism
$$
R\Gamma(C_{\bar{\KK}, \et},  \cO^*_{C_{\bar{\KK}}}) \cong  R\Gamma(U_{\bar{\KK}, \et},  \cO^*_{U_{\bar{\KK}}})\oplus   \Z[-1].
$$
Passing to the hypercohomology we conclude that, for every integer $i\ge 0$, one has
$$ H^i(C_{\et}, \cO^*_{C}) \cong  H^i(U_{\et}, \cO^*_{U})\oplus   H^{i-1}((\Spec \KK)_{\et}, \Z).$$
Since  $H^{1}((\Spec \KK)_{\et}, \Z) =0$  (see, for example,~\cite[formula (2.5 bis)]{GrBr3}),  the claim follows.

Hence, the class $[X - \phi^{-1}(c)] \in  \Br(U)$   extends to  a class in  $\Br(C) $. This contradicts the relative minimality of $\phi$.
\end{proof}

The following result is well-known at least if $\KK$ is perfect (see~\cite[\S4.7]{Iskovskikh96}
or~\mbox{\cite[Proposition~5.2]{Prokhorov-CB}}),
but we recall its proof for completeness.

 \begin{lemma}\label{lemma:CB-2-fibers}
Let $C$ be a conic over a field $\KK$, and let $\phi\colon X\to C$ be a
conic bundle  such that $\phi_{\bar{\KK}}$ has exactly two
reducible fibers.
Suppose that $\phi$ has no sections.
Then~$X$ is a del Pezzo surface with $K_X^2=6$.
\end{lemma}

\begin{proof}
By Lemma~\ref{lemma:perfection} the conic bundle~\mbox{$\phi_{\KK^{sep}}\colon X_{\KK^{sep}} \to C_{\KK^{sep}}$} has two reducible fibers over some $\KK^{sep}$-points of
$C_{\KK^{sep}}$, and it is smooth over the complement to these points. Contracting one irreducible component of each reducible fiber on $X_{\KK^{sep}}$ we obtain
a smooth conic bundle $\phi'\colon X'\to\PP^1$ over $\KK^{sep}$.
By Lemma \ref{lemma:tameness} the conic bundle $\phi'$ has a section.
Thus, we have
$$
X'\cong \mathbb{F}_n=\PP_{\PP^1}\big(\mathcal{O}\oplus\mathcal{O}(n)\big)
$$
for some $n\ge 0$,
and $X_{\KK^{sep}}$ is obtained from $X'$ by blowing up two $\KK^{sep}$-points $P_1$ and $P_2$ not contained in one fiber of~$\phi'$.

Suppose that $n\ge 1$. Denote by $D$ the unique section of $\phi'$ with negative self-intersection.
Let $\bar{D}$ be
its proper transform on $X_{\KK^{sep}}$, so that $\bar{D}$ is a section of~$\phi_{\KK^{sep}}$.
Note that any section of $\phi'$ except $D$ has self-intersection at least~$n$.
If $n\ge 2$, or if~\mbox{$n=1$} and at least one of the points $P_1$ and $P_2$ is contained in $D$,
then $\bar{D}$ is the only section of~$\phi_{\KK^{sep}}$ with negative self-intersection;
hence it is Galois-invariant, which gives a contradiction.
Thus, we see that  $n=1$ and  none of the points $P_1$ and $P_2$ is contained in~$D$,
which implies that $X_{\KK^{sep}}$, and thus also $X$, is a del Pezzo surface.

Now suppose that $n=0$. If $P_1$ and $P_2$ are contained in one section $D$ of $\phi_{\KK^{sep}}$ with
self-intersection $0$, then the proper transform of $D$ on $X_{\KK^{sep}}$ is the only section of $\phi_{\KK^{sep}}$ with negative self-intersection, which gives a contradiction.
Thus we conclude that $P_1$ and $P_2$ are contained in different sections of $\phi_{\KK^{sep}}$ with
self-intersection $0$, which implies that $X_{\KK^{sep}}$,
and thus also $X$, is a del Pezzo surface (cf. \cite[Proposition~7.1(2)]{Liedtke}).

The equality $K_X^2=6$ follows from Noether's formula.
\end{proof}

\begin{lemma}\label{lemma:CB-2-fibersbis}
Let  $\phi\colon X\to \PP^1$  be a conic bundle over a field $\KK$ and $\Omega$ its discriminant locus. Assume that $|\Omega(\bar{\KK})| = |\Omega(\KK)| =2$, that is, $\phi$ has exactly two reducible fibers, and their images are
points with residue field~$\KK$.
Then $X$ is birational to the product $\PP^1 \times C'$, for some conic $C'$ over $\KK$.
\end{lemma}

\begin{proof}
We may  assume that $\phi$ has no sections. Then $X$ is
a del Pezzo surface of degree~$6$ by Lemma  \ref{lemma:CB-2-fibers}.
Recall that the
connected component of identity in the group scheme~\mbox{$\underline{\Aut}(X)$} of automorphisms of any
degree $6$  del Pezzo surface is a rank~$2$ torus $T$ (cf. the proof of Lemma~\ref{lemma:DP6}). The morphism $\phi$ determines a character
\begin{equation}\label{eq:character}
\chi\colon T\to \G_m
\end{equation}
as follows. Let $\{c_1, c_2\}\subset \PP^1$ be the discriminant locus
of $\phi$. Pick a rational function on $\PP^1$ whose divisor is $[c_1]-[c_2]$,
and let $f$ be its pullback to $X$. We claim
that the divisor~$(f)$ of the function $f$
is fixed under the action of $T$. Indeed, the pull back of $(f)$
to $X_{\bar{\KK}}$ is supported on the union $Z_{\bar{\KK}}$  of all six
$(-1)$-curves on
$X_{\bar{\KK}}$.  The subscheme $Z_{\bar{\KK}}$ is fixed under the action the group scheme  $\underline{\Aut}(X_{\bar{\KK}})$. Thus, its  connected subgroup $T_{\bar{\KK}}$ fixes every irreducible component of $Z_{\bar{\KK}}$ and, in particular,  the pull back of $(f)$  to $X_{\bar{\KK}}$.
It follows that, for every smooth scheme $S$ over $\KK$ and an $S$-point $v$ of $T$, we have that
$$
\frac{v^*f}{f}\in \Gamma(X\times S, \mathcal{O}^*_{X\times S})= \Gamma( S, \mathcal{O}^*_{S}).
$$

We define the character~\eqref{eq:character} by the formula
$$
v^*f= \chi(v) f, \quad v\in T.
$$
We claim that  $\chi$ is not identically equal to $1$: otherwise, the action of $T$ on $X$ would preserve every fiber of $\phi$,  which is impossible.
Since the torus $T$ has a non-trivial character, it also has a subtorus isomorphic to $\G_m$  (see~\mbox{\cite[\S8.12]{Borel}}).
Thus, we get a faithful action of  $\G_m$ on $X$. This completes the proof  by  Lemma~\ref{lemma:quotient-by-Gm}.
\end{proof}

Given a conic bundle  $\phi\colon X\to C$,
denote by  $\Bir(X, \phi)$ the group  consisting
of pairs~\mbox{$(\sigma_X, \sigma_C)$}, where $\sigma_X \in \Bir(X)$ and $\sigma_C\in \Bir(C)=\Aut(C)$
such that  $\phi \circ \sigma_X = \sigma_C \circ \phi$
as rational maps from $X$ to $C$. Note that   $\Bir(X, \phi)$
is a subgroup of $\Bir(X)$. One can consider
$\Bir(X, \phi)$ as the group of all birational automorphisms of $X$
that respect the fibration~$\phi$.

We have a left exact
sequence of groups
\begin{equation}\label{eq:exact-sequence}
1\to \Aut(X_{\eta}/\KK(\eta)) \to  \Bir(X, \phi)  \to \Aut(C),
\end{equation}
 where $ \Aut(X_{\eta}/\KK(\eta))$ is the group of automorphisms of the conic $X_{\eta}$  over  $\KK(\eta)$ which is identified with
the  subgroup  of $ \Bir(X, \phi)$ that consists of
pairs $(\sigma_X, \Id_C)$. Note that
if~\mbox{$(\sigma_X, \sigma_C) \in  \Bir(X, \phi)$}, then $\sigma_C$
 preserves the class $[X_\eta]\in \Br(\KK(\eta))$ and, in particular, it preserves the discriminant locus of this class.
For example, if  $\phi\colon X\to C$ is a relatively
minimal conic bundle with the discriminant
locus $\Omega$,
then  $\Omega$ coincides with the discriminant locus of
$[X_\eta]$, so that
$\sigma_C$ belongs to the automorphism group $\Aut(C-\Omega)$ of the affine
curve $C-\Omega$. (The latter group
can be also identified with the stabilizer of $\Omega$ in the group $\Aut(C)$.)
Thus, \eqref{eq:exact-sequence}
actually gives
a left exact sequence
\begin{equation}\label{eq:exact-sequenceminimal}
1\to  \Aut(X_{\eta}/\KK(\eta))   \to  \Bir(X, \phi)  \to \Aut(C-\Omega).
\end{equation}

Now we can prove the main result of this section.

\begin{proposition}\label{proposition:conicsummary}
For every non-negative integer $m$ there exists a number $d(m)$ with the following property.
Let $\KK$ be any field containing all roots of $1$, let
$C$ a conic over $\KK$,  and let $\phi\colon X\to C$ be a conic bundle that has not more than $m$ reducible
fibers over $\bar{\KK}$.
Assume that $X$   is not birational to~\mbox{$\PP^1\times C'$},
where $C'$ is a conic. The following assertions hold.
  \begin{itemize}
\item[(i)] If $\Char \KK \ne 2$,
then the order of every finite subgroup of  $\Bir(X, \phi)$
is at most~$d(m)$.

\item[(ii)]   If $\Char \KK = 2$, then
for every finite  subgroup  $G\subset \Bir(X, \phi)$, one has
$|G|' \leqslant d(m)$, where $|G|' $ is the largest odd factor of $|G|$.
\end{itemize}

\end{proposition}
\begin{proof}
Since  $X$  is not birational to  the product of $\PP^1$ and a conic,
we see that $\phi$ does not have sections.
If $C$ is not rational then,  using the exact sequence~\eqref{eq:exact-sequence}
and  Corollary~\ref{corollary:Zarhin-improved},  we have that, for every finite subgroup  $G \subset \Bir(X, \phi)$,
the number $|G|'$ is equal to~$1$, and if~\mbox{$\Char \KK \ne 2$}, then  $|G| \le 16$.

Now assume that $C\cong \PP^1$.
We may suppose that
$\phi$ is relatively minimal. Let $\Omega$ be
its discriminant locus.  Set $m'=|\Omega(\bar{\KK})|$.

Assume that $m'=0$. Then, since $\Br(\KK) \iso \Br(\PP^1)$, we have that $X_\eta \cong X_c \otimes _\KK \KK(\eta)$, where $c\in \PP^1(\KK)$ is any point.
Hence, $X$ is birational to  the product of $\PP^1$ and a conic $X_c$ which we assumed to be not the case.
Thus, by Lemma~\ref{lemma:CB-1-fiber},
we have that $m'>1$.

Assume that $m'=2$.
Then, by Lemma  \ref{lemma:CB-2-fibersbis}, we have~\mbox{$|\Omega(\KK)|=1$},
that is, $\Omega$ consists of a single point whose residue field is a separable quadratic extension $\LL\supset \KK$. By Corollary~\ref{corollary:P1-without-2-pts}
we have that every finite subgroup of $\Aut(\PP^1-\Omega)$
is a $2$-group of order at most $4$. Hence,  using the sequence~\eqref{eq:exact-sequenceminimal}
and  Corollary \ref{corollary:Zarhin-improved}, we have that, for every finite subgroup~\mbox{$G \subset \Bir(X, \phi)$}, the number
 $|G|'$ is equal to $1$, and if $\Char \KK \ne 2$ then  $|G| \le 16$.

Finally, if $m'>2$ then
$$
|\Aut(\PP^1-\Omega)|\le m'!
$$
Thus, we have $|G|'\le  m'!$, and  if $\Char \KK \ne 2$,
then
$$
|G| \le 4m'!
$$
Summarizing, we find that $d(m)= 16 m!$ works in all cases.
\end{proof}

We conclude this section by recalling that relatively minimal conic bundles
with many geometrically reducible fibers are known to have
few birational models.

\begin{theorem}[{see \cite[Theorem~1.6(iii)]{Iskovskikh96}}]
\label{theorem:conic-bundles-negative-degree}
Let $X$ and $X'$ be smooth projective geometrically rational
surfaces over a field~$\KK$. Let $X\to B$ and $X'\to B'$ be relatively minimal conic bundles.
Suppose that $X$ and $X'$ are birational, and~\mbox{$K_X^2\le 0$}.
Then~\mbox{$K_{X'}^2=K_X^2$}.
\end{theorem}

\begin{remark}\label{remark:isk96}
Theorem~\ref{theorem:conic-bundles-negative-degree}
was proved in \cite{Iskovskikh96} under the assumption
that $\KK$ is perfect.
However, the general case follows using Lemma~\ref{lemma:perfection}.
\end{remark}

\section{Proof of the main theorem}
\label{section:proof}

In this section we prove Theorems~\ref{theorem:main} and~\ref{theorem:main-char}
and Corollary~\ref{corollary:SB-Bir},
and provide a counterexample to Theorem~\ref{theorem:main} over a perfect field of characteristic~$2$.
We start with Theorem~\ref{theorem:main}.

\begin{proof}[Proof of Theorem~\ref{theorem:main}]
The proof is similar to that of Theorem~\ref{theorem:rational-surface-vs-BFS},
see \cite[\S2]{ProkhorovShramov-3folds}.
If~$X$ is birational to $\PP^1\times C$, then the group $\Bir(X)$ contains $\KK^*$ as a subgroup.
Since $\KK$ contains all roots of $1$ by assumption, we
see that $\Bir(X)$ contains cyclic subgroups of arbitrary finite order coprime to
$\Char\KK$, and in particular $\Bir(X)$ has unbounded finite subgroups (and also
fails to have $\Char\KK$-bounded subgroups).
Thus we will assume that $X$ is not birational to such a surface.

Let $G$ be a finite group acting by birational automorphisms of~$X$.
Replace $X$ with
a surface $\tilde{X}$ birational to $X$ such that $G\subset\Aut(\tilde{X})$, and then
further replace $\tilde{X}$ with its $G$-minimal model $Y$.
Therefore, $Y$ is either a $G$-del Pezzo surface, or a $G$-conic bundle
over some conic~$C$ (this means that $C$ can be equipped with an action of $G$, such that the morphism $\phi$ is $G$-equivariant), see~\cite[Theorem~1G]{Iskovskikh80}.
In the former case the order of $G$ is bounded by
universal constants from Proposition~\ref{proposition:dP}(i) and  Remark~\ref{remark:dP-low-degree}.
Thus, we may
assume that $Y$ has the structure $\phi\colon Y\to C$ of a
$G$-conic bundle over a conic~$C$.
In particular, it follows that $G\subset  \Bir(Y, \phi)$.
We can factor $\phi$ as~\mbox{$Y\to Y' \to C$}, where~\mbox{$\phi' \colon  Y' \to C$} is a  relatively minimal conic bundle.
The action of $G$ on $Y$ need not descend to an action on~$Y'$. However, we have that
$$
G\subset \Bir(Y', \phi' ) =  \Bir(Y, \phi).
$$
On the other hand, by
Theorem \ref{theorem:conic-bundles-negative-degree} the number of reducible fibers of $\phi'$ over $\bar{\KK}$ is bounded by some constant which depends on $X$ only. Thus, we are done by
Proposition~\ref{proposition:conicsummary}(i).
\end{proof}

The proof of Theorem~\ref{theorem:main-char} is identical to the proof of  Theorem~\ref{theorem:main}
except for the following. Let $|G|'$ be the largest factor of $|G|$ coprime to $\Char\KK$.
If $Y$ is a del Pezzo surface, to show that $|G|'$ is bounded we apply Remark~\ref{remark:dP-low-degree}
together with Proposition~\ref{proposition:dP}(ii) instead of Proposition~\ref{proposition:dP}(i).
If $Y$ is a $G$-conic bundle, we bound $|G|$ in the case~\mbox{$\Char\KK=3$} using Proposition~\ref{proposition:conicsummary}(i)
as before, while in the case $\Char\KK=2$  we bound $|G|'$
using Proposition~\ref{proposition:conicsummary}(ii).

Now we prove Corollary~\ref{corollary:SB-Bir}.

\begin{proof}[Proof of Corollary~\ref{corollary:SB-Bir}]
By Theorems~\ref{theorem:main} and~\ref{theorem:main-char},
we have to prove that $X$ is not  birational to the surface~\mbox{$\PP^1\times C$},  where $C$ is a conic.   By the theorem of Lang and Nishimura (see for instance~\mbox{\cite[Proposition~IV.6.2]{Kollar-1996-RC}}),  since  $X$ has no $\KK$-points,
every smooth proper variety birational to $X$ has no $\KK$-points either. Thus, we may assume
that  $C(\KK)= \varnothing $. We complete the proof by computing a certain birational invariant for $X$ and for $\PP^1\times C$.

For every proper geometrically integral scheme $Y$ over a field $\KK$, one has an exact sequence
\begin{equation}\label{equation:brscheme}
 0 \to \Pic(Y) \to (\Pic(Y_{\KK^{sep}}))^{\Gal(\KK^{sep}/\KK)} \to  \Br(\KK) \rar{\alpha_Y} \Br(Y).
 \end{equation}
Composing $\alpha_Y$ with the injection  $\Br(Y)\hookrightarrow \Br(\KK(Y))$, where $\KK(Y)$ is the field of rational functions on $Y$,  we obtain a homomorphism  $\Br(\KK) \to \Br(\KK(Y))$ induced by the embedding of the fields
$\KK \hookrightarrow \KK(Y)$. This shows that $\ker \alpha_Y$ is a birational invariant of $Y$.  We use exact sequence~\eqref{equation:brscheme} to compute this invariant.

If $Y= \PP^1\times C$ for some conic $C$ with $C(\KK)= \varnothing $,  then $\ker \alpha_Y$ is generated by $[C]\in  \Br(\KK)$ and, thus, has order $2$.
If $Y$ is a Severi--Brauer surface, then  $\ker \alpha_Y$  has order $3$. If $Y$ is a product of two non-isomorphic non-rational conics, then  $\ker \alpha_Y$  has order $4$. If
$Y$ is a quadric surface with  $Y(\KK)= \varnothing $ and $\Pic(Y)\cong \Z$, then   $\ker \alpha_Y$ is trivial. Indeed, in this case the group
$(\Pic(Y_{\KK^{sep}}))^{\Gal(\KK^{sep}/\KK)}$ is generated by the class of a hyperplane section in~$\PP^3$,
and thus the canonical map $\Pic(Y)\to(\Pic(Y_{\KK^{sep}}))^{\Gal(\KK^{sep}/\KK)}$
is an isomorphism. We conclude that surfaces of the latter three types cannot be birational to a product
of~$\PP^1$ and a conic.
\end{proof}

\begin{remark}
To prove Corollary~\ref{corollary:SB-Bir} in the case of Severi--Brauer surfaces
one can give a geometric argument.
Suppose
that the group of birational automorphisms of $X$ has unbounded finite subgroups.
Then by Theorem~\ref{theorem:main} the surface $X$ is birational to a surface $Y\cong\PP^1\times C$ for
some conic $C$. Since $\PP^1$ has a $\KK$-point, and $C$ has a point defined over
a quadratic extension of $\KK$, we conclude that $Y$ has a point defined over
a quadratic extension of $\KK$ as well. Therefore, the same holds for $X$ by the theorem of Lang and Nishimura.
However, a Severi--Brauer surface with a point defined over a quadratic
extension is isomorphic to $\PP^2$,
so we obtain a contradiction.
\end{remark}

Note that there is a description of birational automorphism groups of Severi--Brauer surfaces via their generators and relations, see~\cite[\S3]{IskovskikhTregub}. We do not know whether it can be used to prove Corollary~\ref{corollary:SB-Bir}.

Below, we will give an example showing that  Theorem~\ref{theorem:main}
fails over perfect fields of characteristic~$2$. We start with the following general construction.

\begin{lemma}\label{lemma:example-Azumaya}
Let $\KK$ be a perfect field of characteristic $p>0$, let $\Delta \subset \PP^1(\KK)$ be a finite subset consisting of more than one point, and let $Y=\PP^1  -  \Delta$. Assume that $\KK$ admits a Galois extension of degree $p$.
Assume also that either $p\ne 2$, or the cardinality of $\Delta$ is even. Then the following assertions hold.

\begin{itemize}
\item[(i)] There exists an element  $\alpha \in \Br (Y)$ of order $p$  in the Brauer group of $Y$ which is not contained in the image of the restriction map $\Br(\PP^1  -  \Delta') \to \Br(Y)$, for any proper subset $\Delta'\subset \Delta$.

 \item[(ii)] For any $\alpha$ as above, there exists an Azumaya algebra $A$ over $Y$ of rank $p^2$ whose class in the Brauer group $\Br(Y)$ is equal to $\alpha$.
\end{itemize}
\end{lemma}

\begin{proof}
We compute $\Br(Y)\cong H^2(Y_{\et}, \cO^*_Y)$ using the Leray spectral sequence that converges
to $H^\bullet(Y_{\et}, \cO^*_Y)$ and whose second page is given by
$H^i\big(\Gal (\bar{\KK}/\KK), H^j(Y_{\bar{\KK}, \et}, \cO^*_{Y_{\bar{\KK}}})\big)$.
By Tsen's theorem, we know that
$$
H^2(Y_{\bar{\KK}, \et}, \cO^*_{Y_{\bar{\KK}}})\cong \Br(Y_{\bar{\KK}})=0.
$$
In addition, one has $H^1(Y_{\bar{\KK}, \et}, \cO^*_{Y_{\bar{\KK}}})\cong \Pic(Y_{\bar{\KK}})=0$. Hence, we conclude that
$$
\Br(Y) \cong H^2(\Gal (\bar{\KK}/\KK),  \cO^* (Y_{\bar{\KK}})).
$$

We have an exact sequence of Galois modules
\begin{equation}\label{eq:Z-Delta}
0\to\bar{\KK}^* \to \cO^* (Y_{\bar{\KK}}) \to \Z[\Delta] \to \Z \to 0,
\end{equation}
where  $\Z[\Delta]$ is the group of divisors on $\PP^1$ supported on $\Delta$, and the  last morphism is the degree map.
The sequence~\eqref{eq:Z-Delta} splits as a sequence of Galois modules, that is, it is a direct sum of sequences of Galois modules of length $2$.
Since~\eqref{eq:Z-Delta} splits, it remains exact (and split) after passing to Galois cohomology:
$$
0\to \Br(\KK) \to \Br(Y) \rar{\Res} H^2(\Gal (\bar{\KK}/\KK),  \Z[\Delta]) \to  H^2(\Gal (\bar{\KK}/\KK),  \Z) \to 0.
$$
Since $ H^i(\Gal (\bar{\KK}/\KK),  \Q)=0$ for every $i>0$, we have
$$
H^2(\Gal (\bar{\KK}/\KK),  \Z)\cong \Hom(\Gal (\bar{\KK}/\KK), \Q/\Z).
$$
Thus, the above exact sequence takes the form
$$
0\to \Br(\KK) \to \Br(Y) \rar{\Res}  \bigoplus_{\Delta}  \Hom(\Gal (\bar{\KK}/\KK), \Q/\Z) \rar{\Sigma}   \Hom(\Gal (\bar{\KK}/\KK), \Q/\Z) \to 0.
$$

Tensoring this exact sequence by $\Z/p\Z$ 
and using the vanishing of the $p$-torsion part of~\mbox{$\Br(\KK)$}   
(see Remark \ref{remark:char-vs-dim-perfect}), we find  the  following short exact sequence
$$
0\to  \Br(Y)_p \rar{\Res} \bigoplus_{\Delta}  \Hom(\Gal (\bar{\KK}/\KK), \Z/p\Z) \rar{\Sigma}   \Hom(\Gal (\bar{\KK}/\KK), \Z/p\Z) \to 0,
$$
where  we write  $\Br(Y)_p$  for the group of $p$-torsion elements in $\Br(Y)$.  Since $\KK$ has a Galois extension of degree $p$, the group $\Hom(\Gal (\bar{\KK}/\KK), \Z/p\Z)$ is non-trivial (and, thus, it contains a cyclic subgroup of order $p$).
Under our assumptions on $p$ and the cardinality of~$\Delta$ we can find an element of
$$
\ker \Sigma \subset \bigoplus_{\Delta}  \Hom(\Gal (\bar{\KK}/\KK), \Z/p\Z)
$$
such that each of its components in~\mbox{$\Hom(\Gal (\bar{\KK}/\KK), \Z/p\Z)$} is not equal to zero.
Let $\alpha$ be the corresponding element of $ \Br(Y)_p$. The compatibility of all our constructions  with the restriction map  $\Br(\PP^1 -  \Delta') \to \Br(Y)$ shows that $\alpha$ is not contained in the image of such map. This proves assertion~(i).

To prove assertion~(ii), we recall from \cite[Proposition 4.2]{OV07} that for every smooth affine scheme $Z$ over  a perfect field $\KK$ of characteristic $p>0$ there is a surjective homomorphism
$$
\Psi\colon H^0(Z,\Omega^1_Z) \to  \Br(Z)_p,
$$
which takes a differential form $\omega = xdy$ to the class of the Azumaya algebra $A_{\omega}$  generated over $\cO(Z)$ by elements $v$ and  $u$ subject to the relations $v^p=x$, $ u^p=y$ and  $vu-uv=1$. This Azumaya algebra has rank $p^2$.    Applying this to $Z=Y$
and observing that every differential form on $Y$ can be written as $xdy$ for some regular functions
$x$ and $y$ on $Y$ (which holds since $Y$ is one-dimensional and $\Delta\neq\varnothing$)
we conclude that every $p$-torsion class in $\Br(Y)$ is represented by  an Azumaya algebra  of rank $p^2$. This completes the proof of assertion~(ii).

\end{proof}

\begin{remark}
The Artin--Schreier short exact sequence
$$
0 \to \Z/p\Z \to \bar{\KK} \rar{\theta} \bar{\KK} \to 0,
$$
where the homomorphism $\theta$ is defined by $\theta(c)= c - c^p$,  gives rise to an isomorphism
$$
\Hom(\Gal (\bar{\KK}/\KK), \Z/p\Z)\cong \coker \theta.
$$
One can verify that the composition
$$
H^0(Y,\Omega^1_Y) \rar{\Psi} \Br(Y)_p \rar{\Res} \bigoplus_{\Delta}  \Hom(\Gal (\bar{\KK}/\KK), \Z/p\Z) \cong   \bigoplus_{\Delta} \coker \theta
$$
carries a differential form $\omega$ to its residues (modulo the image of $ \theta $) at all points of $\Delta$. This can be used to construct a class $\alpha$ and an Azumaya algebra $A$ satisfying the requirements of Lemma~\ref{lemma:example-Azumaya}
explicitly. Namely, take an invertible function $f$  on $Y$ whose order at every point of $\Delta$ is not equal to $0$ modulo $p$, and an element $t\in \KK$, which is not in the image of  the Artin--Schreier homomorphism. Then the class
$$
\Psi(t\frac{df}{f}) \in  \Br(Y)_p
$$
satisfies the requirements of Lemma~\ref{lemma:example-Azumaya}(i),
and the Azumaya algebra $A_{t\frac{df}{f}}$ satisfies the requirements of Lemma~\ref{lemma:example-Azumaya}(ii).
\end{remark}

Recall the following result.

\begin{theorem}[{see~\cite[\S4]{Iskovskikh96}}]
\label{theorem:Iskovskikh}
Let $\KK$ be a field, and let $X$ be a smooth projective geometrically
rational surface over $\KK$ with
$X(\KK)\neq\varnothing$. Suppose that $X$ has a structure of a relatively minimal conic bundle and $K_X^2\le 4$.
Then $X$ is not rational.
\end{theorem}

\begin{remark}[{cf. Remark~\ref{remark:isk96}}]
Theorem~\ref{theorem:Iskovskikh}
was proved in \cite{Iskovskikh96} under the assumption
that the field~$\KK$ is perfect.
The general case follows using Lemma~\ref{lemma:perfection}.
\end{remark}

Finally, we present a counterexample to Theorem~\ref{theorem:main} over a perfect field of characteristic~$2$.

\begin{example}
\label{example:main-counterexample}
Let $\Bbbk$ be an algebraically closed field of characteristic~$2$, and let
$$
\KK=\Bbbk\big(t^{2^{-\infty}}\big),
$$
that is, $\KK$ is obtained from $\Bbbk$ by adjoining
all roots of a transcendental variable $t$ of degrees~$2^r$, for all positive
integers~$r$.
Then $\KK$ is a perfect field of characteristic~$2$, and it has a Galois extension of degree $2$. Pick a subset $\Delta \subset \PP^1(\KK)$ of $4$ points. Set~\mbox{$Y=\PP^1 - \Delta$}.
Let $\alpha\in\Br(Y)$ be an element with the property 
described in Lemma~\ref{lemma:example-Azumaya}(i),  
and let $X\to Y$ be the smooth conic
bundle corresponding to the Azumaya algebra from
Lemma~\ref{lemma:example-Azumaya}(ii). 
By Lemma~\ref{lemma:extendingconicbundles}  there exists a conic bundle~\mbox{$\phi\colon \hat{X} \to\PP^1$}
that extends~\mbox{$X\to Y$}.
Then~$\phi$ has degenerate fibers exactly over the points of $\Delta$.
This means that~\mbox{$K_{\hat{X}}^2=4$}. Moreover, 
the conic bundle $\phi$ is relatively minimal by the choice of~$\alpha$.
Also, we know that every fiber
of $\phi$ has a $\KK$-point (see Remark~\ref{remark:char-vs-dim-perfect}),
and thus, in particular, $\hat{X}$ has a $\KK$-point.
Hence, the surface $\hat{X}$ is not rational 
by Theorem~\ref{theorem:Iskovskikh}.
On the other hand, the group of birational selfmaps of $\hat{X}$ contains an automorphism
group of the scheme-theoretic general fiber~$\hat{X}_\eta$ 
of $\phi$, which is a conic over a non-perfect field~\mbox{$\KK(\PP^1)$} 
of characteristic $2$. 
The latter automorphism group has unbounded finite subgroups by
Lemma~\ref{p-center}(ii).
We conclude that the analog of Theorem~\ref{theorem:rational-surface-vs-BFS}
(and, thus,  also Theorem~\ref{theorem:main})
fails over perfect fields of characteristic~$2$.
\end{example}

\section{Jordan property}
\label{section:Jordan}

In this section we apply Theorem~\ref{theorem:main} to study groups of birational automorphisms
of higher-dimensional varieties. The group-theoretic property we will be interested in here
is described as follows.

\begin{definition}[{see \cite[Definition~2.1]{Popov-Jordan}}]
\label{definition:Jordan}
A group~$\Gamma$ is called \emph{Jordan}
if there is a constant~$J$ such that
for any finite subgroup $G\subset\Gamma$ there exists
a normal abelian subgroup~\mbox{$A\subset G$} of index at most~$J$.
\end{definition}

The first result concerning Jordan property (and motivating the modern terminology)
is an old theorem by C.\,Jordan that asserts this property for a
general linear group over a field of characteristic zero (see e.g.~\cite[Theorem~36.13]{Curtis-Reiner-1962}).

J.-P.\,Serre noticed that Jordan property sometimes holds
for groups of birational automorphisms.

\begin{theorem}[{\cite[Theorem~5.3]{Serre2009}}]
\label{theorem:Serre}
The group of birational automorphisms of $\PP^2$
over a field of characteristic zero is Jordan.
\end{theorem}

In \cite[Theorem~1.8]{ProkhorovShramov-RC}, Yu.\,Prokhorov and C.\,Shramov generalized
Theorem~\ref{theorem:Serre} to the case of rationally connected varieties of arbitrary dimension
over fields of characteristic zero
(actually, their
results were initially obtained modulo boundedness of terminal Fano
varieties, which
was recently proved by C.\,Birkar in~\cite[Theorem~1.1]{Birkar}).
Jordan property was also proved in many other cases for groups of birational automorphisms
and for automorphism groups,
see for instance~\cite{ProkhorovShramov-Bir}, \cite{MZ},
\cite{BandmanZarhin2015a}, \cite{BandmanZarhin2017}, \cite{ProkhorovShramov-compact},
\cite{Riera-Spheres}, \cite{Riera-HamSymp},
and references therein.
However, there are varieties whose groups of birational automorphisms are not Jordan.

\begin{theorem}[\cite{Zarhin14}]
\label{theorem:Zarhin}
Let $A$ be a positive-dimensional abelian variety over an algebraically closed field of characteristic zero.
Then the group of birational automorphisms
of $A\times\PP^1$ is not Jordan.
\end{theorem}

Recall that on any variety $X$ over a field of characteristic zero
one can construct the \emph{maximal rationally connected fibration},
which is a canonically defined rational map with
rationally connected fibers and non-uniruled base
(see~\cite[\S\,IV.5]{Kollar-1996-RC},
\cite[Corollary~1.4]{Graber-Harris-Starr-2003}).
The maximal rationally connected fibration is equivariant
with respect to the whole group of birational automorphisms of~$X$.
Keeping in mind Theorem~\ref{theorem:Zarhin}, it seems natural to try to understand
the groups of birational automorphisms for varieties with maximal rationally
connected fibration of small relative dimension;
note that the case when the relative dimension is~$0$ (that is, the maximal rationally
connected fibration is birational or, equivalently, the variety is not uniruled)
is settled by~\cite[Theorem~1.8(ii)]{ProkhorovShramov-Bir}.

\begin{theorem}[{\cite[Theorem~1.6]{BandmanZarhin2015a}}]
\label{theorem:BZ-Jordan}
Let $X$ be an irreducible variety over a field of characteristic zero, and let $\phi\colon X\dasharrow Y$
be its maximal rationally connected fibration.
Suppose that the relative dimension of $\phi$ equals $1$, and that $\phi$ has no rational sections.
Then the group of birational automorphisms of $X$ is Jordan.
\end{theorem}

\begin{theorem}[{\cite[Lemma~4.6]{ProkhorovShramov-3folds}}]
\label{theorem:PS-Jordan}
Let $X$ be an irreducible variety over a field of characteristic zero, and let $\phi\colon X\dasharrow Y$ be its maximal rationally connected fibration.
Suppose that the relative dimension of $\phi$ equals $2$, and that $\phi$ \emph{has} a rational section.
Suppose also that $X$ is not birational to $Y\times\PP^2$.
Then the group of birational automorphisms of $X$ is Jordan.
\end{theorem}

Using Theorem~\ref{theorem:main}, we obtain the following result that is somewhat similar
to Theorems~\ref{theorem:BZ-Jordan}
and~\ref{theorem:PS-Jordan}.

\begin{proposition}\label{proposition:Jordan}
Let $X$ be an irreducible variety over a field $\Bbbk$ of characteristic zero, and let~\mbox{$\phi\colon X\dasharrow Y$} be its maximal rationally connected fibration.
Suppose that the relative dimension of~$\phi$ equals $2$, and that $\phi$ has no rational sections.
Suppose also that $X$ is not birational
to~\mbox{$Z\times\PP^1$}, where $Z$ is a conic bundle over~$Y$.
Then the group of birational automorphisms of $X$ is Jordan.
\end{proposition}
\begin{proof}
The proof is similar to those of Theorems~\ref{theorem:BZ-Jordan}
and~\ref{theorem:PS-Jordan}, see~\cite[\S5]{BandmanZarhin2015a}
and~\mbox{\cite[\S4]{ProkhorovShramov-3folds}}. We may assume that the field
$\Bbbk$ is algebraically closed.

Let $S$ be the fiber of $\phi$ over the scheme-theoretic generic point of $Y$.
Then $S$ is a geometrically rational surface defined over the field
$\KK=\Bbbk(Y)$, and $S$ has no $\KK$-points by assumption.
Since $\phi$ is equivariant with respect to $\Bir(X)$, we have an exact sequence
of groups
$$
1\longrightarrow \Bir(X)_\phi \longrightarrow \Bir(X)\longrightarrow\Bir(Y),
$$
where the action of $\Bir(X)_\phi$ is fiberwise with respect to $\phi$.

Our assumptions imply that $S$ is not birational to $\PP^1\times C$ over~$\KK$, where $C$ is a conic.
The group $\Bir(X)_\phi$ is isomorphic to the group
$\Bir(S)$, and thus has bounded finite subgroups by Theorem~\ref{theorem:main}.
On the other hand, the variety $Y$ is not uniruled. Thus the group~\mbox{$\Bir(Y)$} is Jordan
by~\cite[Theorem~1.8(ii)]{ProkhorovShramov-Bir}.
Moreover, there is a constant $R$ such that every finite subgroup
in $\Bir(Y)$ is generated by at most $R$ elements, see e.g.~\mbox{\cite[Remark~6.9]{ProkhorovShramov-Bir}}.
Therefore, the group $\Bir(X)$ is Jordan by a simple group-theoretic argument,
see~\mbox{\cite[Lemma~2.8]{ProkhorovShramov-Bir}}.
\end{proof}

Note that if the field $\Bbbk$ is algebraically closed, the assumptions of Proposition~\ref{proposition:Jordan}
imply that the dimension of $Y$ is at least~$2$ (so that
the dimension of $X$ is at least~$4$) due to the theorem of Graber, Harris, and Starr,
see~\cite{Graber-Harris-Starr-2003}.
It would be interesting to describe more explicitly the varieties with maximal rationally connected
fibration of relative dimension~$2$ whose birational automorphism groups are not Jordan, similarly to
what was done in~\cite[Theorem~1.8]{ProkhorovShramov-3folds}.
Also, it would be interesting to try to generalize Theorem~\ref{theorem:main}
to the case of threefolds, and to derive conclusions for birational automorphisms
of varieties with maximal rationally connected
fibration of relative dimension~$3$.

\appendix
\section{Quotients}
\label{section:appendix}

In this section we recall some well known facts about the quotients of algebraic varieties by
the groups~$\Gm$ and~$\Ga$. We start with a simple observation.

\begin{lemma}\label{lemma:subgroup}
Let $\KK$ be an arbitrary field.
Let $\G$ be a linear algebraic group over $\KK$ isomorphic either to $\Gm$ or to $\Ga$, and let
$H$ be a finite subgroup scheme of $\G$.
Then the quotient $\G/H$ is isomorphic to~$\G$.
\end{lemma}

\begin{proof}
We first prove the claim for a connected  finite subgroup scheme $H$.  If $\Char \KK =0$ then every such $H$  is trivial and there is nothing to prove. Assume $\Char \KK =p>0$.
If
$$
\G = \Ga= \Spec \KK[x]
$$
then the subscheme $H\subset  \G$ is given by the ideal $(x^{p^n})$ for some integer
$n\ge 0$.
Similarly, if
$$
\G = \Gm = \Spec \KK[x,x^{-1}]
$$
then the subscheme $H\subset  \G$ is given by the ideal $((x-1)^{p^n})$ for some integer
$n\ge 0$.
It follows that $H$  fits into an exact sequence of group schemes
$$
1\longrightarrow H \xlongrightarrow{\phantom{\quad F^n\quad}} \G\xlongrightarrow{\quad F^n\quad}\G \longrightarrow  1,
$$
where $F$ is the relative Frobenius morphism. Thus $\G/H $ is  isomorphic to~$\G$.

Next, assume that $H$ is  an \'etale subgroup scheme.  Note that  \'etale subgroup schemes of $\G$ are in one-to-one correspondence with finite subgroups of $\G(\KK^{sep})$
stable under the action of the Galois group. It follows that every  \'etale subgroup scheme $H$  of $\G_m$ fits
into an exact sequence
$$
1\longrightarrow H \xlongrightarrow{\phantom{\quad x\mapsto x^n\quad}} \G_m\xlongrightarrow{\quad x\mapsto x^n\quad} \G_m \longrightarrow  1
$$
for some non-negative integer  $n$ not divisible by the characteristic of $\KK$. Hence, the quotient~\mbox{$\G_m/H$} is isomorphic to $\G_m$.

The additive group scheme $\G_a$ over a field of characteristic~$0$ has no non-trivial finite subgroup schemes.
If $\Char \KK =p>0$ and   $H\subset \G_a$ is  an \'etale   subgroup scheme,  we set
$$
f_H(x) = \prod_{\alpha \in H(\KK^{sep}) \subset  \KK^{sep}} (x - \alpha) \in \KK[x],
$$
and let $f_H\colon \G_a \to \G_a$ be the morphism of schemes given by $f_H(x)$. We claim that $f_H$ is a  group scheme homomorphism whose kernel is $H$, i.e. the identity
$$
f_H(x+y)= f_H(x) +f_H(y)
$$
holds in $\KK[x,y]$.

It suffices to verify the claim for $\KK= \bar \KK$. To do this we use the induction on the dimension of the $\mathbb{F}_p$-vector space  $H(\bar  \KK )\subset \bar  \KK$.
If $H(\bar  \KK )$ is spanned by some $0\ne c\in  \bar  \KK$, then
$$
f_H(x)= x^p- a^{p-1} x
$$
which is additive. For the induction step, pick a one-dimensional subspace
 $ H'(\bar  \KK)  \subset   H(\bar  \KK)$ and write
$f_H$ as the composition
$$
\G_a \xlongrightarrow{\quad f_{H'}\quad} \G_a \xlongrightarrow{\quad f_{H/H'}\quad} \G_a,
$$
where  $f_{H/H'}$ is the morphism corresponding to the image of $H$ under $f_{H'}$ which is an \'etale subgroup scheme of $\G_a$. The morphisms  $f_{H'}$ and $f_{H/H'}$ are group scheme
homomorphisms by the induction assumption. Hence, the morphism $f_{H}$ is also a  group scheme
homomorphism. It follows that $\G_a/H \cong \G_a$.

Finally, for  an arbitrary  finite subgroup scheme $H$ consider the exact sequence
$$
1\longrightarrow  H_0 \longrightarrow  H \longrightarrow   H_1 \longrightarrow  1,$$
where $H_0$ is the maximal connected subgroup scheme of $H$ and $H_1$ is an \'etale group scheme. The quotient of $\G/H$ is isomorphic to the quotient of $\G/H_0 \cong \G $ by
$H_1$ and we win.
\end{proof}

Let   $X$ be a scheme over a field $\KK$ acted on by a group scheme  $\Gamma$ over $\KK$:
$$
\sigma \colon   \Gamma \times X \to X.
$$
Recall  from~\cite[Definition 0.6]{MumfordFogartyKirwan} that a morphism $\pi\colon X\to Y$ to a scheme $Y$ over  $\KK$ is said to be a geometric quotient of $X$ by $\Gamma$ if the following conditions hold.
\begin{itemize}
\item The morphism  $\pi $  is  $\Gamma$-equivariant for the trivial action of $\Gamma$ on $Y$. That is
$$
\pi \circ \sigma= \pi \circ p_X,
$$
where $p_X\colon  \Gamma \times X \to X$
is the projection.

\item The morphism $\pi$ is submersive, i.e. for any subset $W\subset Y$ the set $\pi^{-1}(W) \subset X$ is open if and only if $W$ is open.

\item The group scheme $\Gamma$ acts transitively on all fibers of $\pi$,  that is the morphism
$$
\Gamma \times X \to X\times _Y X, \quad (g,x) \mapsto (gx,x),
$$
is surjective.

\item  The natural morphism of sheaves $\pi^\sharp\colon \cO_Y \to \pi_* \cO_X$ identifies the sheaf   $ \cO_Y$ with the subsheaf of  $\pi_* \cO_X$ that consists of $\Gamma$-invariant
functions. For a linear group scheme $\Gamma$ this is equivalent to the requirement that the sequence
$$
0\longrightarrow \cO_Y\longrightarrow \pi_*\cO_X\longrightarrow  \pi_*\cO_X\otimes\KK[\Gamma]
$$
is exact, where $\KK[\Gamma]$ denotes the algebra of functions on the affine variety~$\Gamma$.
\end{itemize}
By~\cite[Proposition 0.1]{MumfordFogartyKirwan} a geometric quotient  $\pi\colon X\to Y$ is also a categorical quotient. In particular, a geometric quotient  of $X$ by $\Gamma$ is unique if it exists.

\begin{lemma}\label{lemma:smoothness}
Let   $X$ be  an irreducible algebraic variety over a field $\KK$ with an action of an algebraic group $\Gamma$ over $\KK$, and let $\pi\colon X\to Y$ be a
geometric quotient of $X$ by $\Gamma$. Then~$\pi$ is generically smooth on the target, that is the fiber $X_\eta$ of $\pi$ over the scheme-theoretic generic
point $\eta \in Y$ is smooth over $\KK(\eta)$.
\end{lemma}
 \begin{proof}
Since $X_\eta$ carries a transitive action of an algebraic group it suffices to check that the smooth locus of   $X_\eta$ is not empty.
Let us check that the field of rational functions~\mbox{$\KK(X_\eta)$} on  $X_\eta$  is separable over $\KK(\eta)$.
If $\Char \KK =0$ this is automatic. Assume for the rest of the proof that  $\Char \KK =p>0$. Denote by  $ (\KK(\eta))^{\frac{1}{p}}$
the  subfield of an algebraic closure of $\KK(\eta)$ that consists of all elements whose $p$-th power is in $\KK(\eta)$.
We have to verify that~\mbox{$\KK(X_\eta) \otimes _{\KK(\eta)} (\KK(\eta))^{\frac{1}{p}}$}  is a field.
Moreover, it suffices to do this for  $(\KK(\eta))^{\frac{1}{p}}$ replaced by a finite subextension:
$$
\KK(\eta) \subset (\KK(\eta))(a_1, \ldots, a_n) \subset (\KK(\eta))^{\frac{1}{p}}.
$$
We use induction on $n$ taking $n=0$ as the base of induction.

Assume that
$$
\LL=\KK(X_\eta) \otimes _{\KK(\eta)} (\KK(\eta))(a_1, \ldots, a_{n-1})
$$
is a field.
Set $a=a_n^p \in  \KK(\eta)$. If the polynomial $x^p -a$ does not have zeros in $\LL$ then
$$
\KK(X_\eta) \otimes _{\KK(\eta)} (\KK(\eta))(a_1, \ldots, a_n) \cong \LL[x]/(x^p-a)
$$
is also a field and we are done. Thus, we suppose that the polynomial $x^p -a$ does have
a zero $f \in \LL$, $f^p=a$. The equality  $\KK(\eta)= \KK(X_\eta)^{\Gamma}$ implies that
$$
(\KK(\eta))(a_1, \ldots, a_{n-1})  = \LL^{\Gamma}.
$$
Now, since $\Gamma$ is smooth over $\KK$
and $f^p \in \LL$ is $\Gamma$-invariant, the element $f$ itself is $\Gamma$-invariant. Hence, one has
$f\in   (\KK(\eta))(a_1, \ldots, a_{n-1})$ and
$$
(\KK(\eta))(a_1, \ldots, a_n)=   (\KK(\eta))(a_1, \ldots, a_{n-1}).
$$
This proves that  $\KK(X_\eta)$  is separable over $\KK(\eta)$, and the assertion of the lemma follows.
\end{proof}

\begin{remark}
If $\Gamma$ is a  non-smooth group scheme over $\KK$, then a geometric quotient~\mbox{$\pi\colon X\to Y$} with respect to $\Gamma$  need not be generically smooth.
For example, if $\Char \KK =p>0$ the morphism
$$
\G_m\xlongrightarrow{\quad x\mapsto x^p\quad} \G_m
$$
is a geometric quotient of $\G_m$ by the finite subgroup scheme $\mu_p \subset \G_m$ and it is not generically smooth.
\end{remark}

Recall the following classical result.

\begin{theorem}[{\cite[Theorem~2]{Rosenlicht56}}]
\label{theorem:Rosenlicht}
Let
$\KK$ be a field, and let $\Gamma$ be a linear algebraic group over $\KK$.
Let $X$ be an irreducible algebraic variety over $\KK$ with an action of $\Gamma$.
Then there exists a dense open $\Gamma$-invariant subscheme $U\subset X$
such that a geometric quotient~\mbox{$\pi\colon U\to Y$} exists.
\end{theorem}

The proof of Theorem~\ref{theorem:Rosenlicht} given in \cite{Rosenlicht56} is rather complicated.
For  our main application (which is Corollary~\ref{corollary:unipotent-action} below),
 it is sufficient to use the version of Theorem~\ref{theorem:Rosenlicht} over a perfect field.
In this case, the assertion can be easily reduced to the case of an algebraically closed field $\KK$ using the Galois descent.
On the other hand, for an algebraically closed field $\KK$ a proof of Theorem~\ref{theorem:Rosenlicht} in modern language can be found in~\cite{Gyoja}.

The proof of the following result for an algebraically closed $\KK$
is contained in the proof of~\cite[Theorem~2]{BB}; cf. an earlier but less accurate
proof in \cite{Matsumura} (we refer the reader to \cite{Popov-quotient} for a detailed comment on the
latter).

\begin{lemma}\label{lemma:quotient-by-Gm}
Let $\KK$ be a field.
Let $\G$ be a linear algebraic group over $\KK$ isomorphic either to $\Gm$ or to $\Ga$.
Let $X$ be an irreducible variety
over $\KK$ with a non-trivial  action of~$\G$.
Then $X$ is birational to $\PP^1\times Y$ for some variety $Y$.
\end{lemma}

\begin{proof}
By Theorem~\ref{theorem:Rosenlicht} there exists
a dense open $\Gamma$-invariant subscheme $U\subset X$
such that a geometric quotient  $\pi\colon U\to Y$ exists.
By Lemma~\ref{lemma:smoothness}
the fiber $U_\eta$ of $\pi$ over the scheme-theoretic generic
point $\eta \in Y$ is smooth over  $\KK(\eta)$. Denote by $\LL$ a separable closure
of~\mbox{$\KK(\eta)$} and by  $U_\eta \otimes \LL$ the base change of $U_\eta$ to $\Spec \LL$. The smoothness of~$U_\eta$  implies that the set of $\LL$-points of $U_\eta$  is not empty.
Pick a point $x_0 \in U_\eta (\LL)$. The scheme~\mbox{$U_\eta \otimes\LL$} is acted on by the algebraic group $\G_\eta \otimes\LL$. Let $H_{\LL} \subset \G_\eta \otimes \LL$ be the stabilizer of the
point~$x_0$ with respect to this action. This is a finite (possibly non-reduced) subgroup scheme of~\mbox{$\G_\eta \otimes \LL$}. We claim that the orbit morphism
\begin{equation}\label{eq:orbit}
\sigma\colon (\G_\eta \otimes\LL) /H_\LL \to U_\eta \otimes \LL, \quad g H_\LL \mapsto g(x_0),
\end{equation}
is an isomorphism. Indeed, the morphism~$\sigma$  is surjective because $\G_\eta \otimes\LL$ acts transitively on $ U_\eta \otimes \LL$, and it is flat
because $ U_\eta \otimes \LL$ is smooth. Next, by definition of the   stabilizer subgroup scheme, for any scheme $S$ over $\LL$,  the morphism~$\sigma$ is injective  on $S$-points.
It follows that~$\sigma$ is a bijection on $S$-points and, hence, an isomorphism.

Now, since $\G$ is abelian and the action of $\G_\eta $ on $U_\eta$ is transitive, the stabilizer $H_\LL$ does not depend on the point $x_0$. It follows that $H_\LL  \subset \G_\eta \otimes \LL$
is a $\Gal(\LL/\KK(\eta))$-invariant closed subscheme, i.e. $H_\LL$ is obtained from a closed subgroup scheme $H\subset \G_\eta$ by the base change to $\Spec \LL$.

The group scheme
$\G_{\KK(\eta)}/H$  acts on   $U_\eta$. We claim that $U_\eta$  is a  $\G_{\KK(\eta)}/H$-torsor, that is the morphism
$$
\G_{\KK(\eta)}/H \times  U_\eta \to   U_\eta \times   U_\eta, \quad (g,x) \mapsto (gx,x),
$$
is an isomorphism. Indeed, it suffices to check this after the base change to $\Spec \LL$, where the assertion is clear since the morphism~$\sigma$
in~\eqref{eq:orbit} is an isomorphism.

Finally, by Lemma~\ref{lemma:subgroup}
the quotient $ \G_{\KK(\eta)}/H$   is isomorphic to  $\G_{\KK(\eta)}$.
On the other hand, by~\mbox{\cite[Theorem~10]{Rosenlicht56}}  every $\G_{\KK(\eta)}$-torsor is trivial,
so that $U_\eta \cong \G_{\KK(\eta)}$ (see also~\cite{Rosenlicht67}  for a different proof of this statement).  Now the assertion easily follows.
\end{proof}

An alternative proof of Lemma~\ref{lemma:quotient-by-Gm} over algebraically closed fields
can be found in~\cite{Popov-quotient}. An interesting feature of this proof is that
it avoids working with the torsor~$U_\eta$.

\begin{corollary}\label{corollary:unipotent-action}
Let $\KK$ be a perfect field, and let $X$ be an irreducible variety over $\KK$ with a faithful action of a  connected linear
algebraic group $\Gamma$.
Suppose that $\Gamma$ is not reductive. Then $X$ is birational to $\PP^1\times Y$ for some variety $Y$.
\end{corollary}
\begin{proof}
Since $\Gamma$ is not reductive and~$\KK$ is perfect, the unipotent radical of~$\Gamma$ (that is, its largest normal
unipotent connected closed subgroup) is non-trivial.
Hence, again using the perfectness assumption
together with~\mbox{\cite[Theorem~15.4(iii)]{Borel}}, we see that~$\Gamma$
contains a subgroup isomorphic to~$\Ga$.
Now the assertion follows from Lemma~\ref{lemma:quotient-by-Gm}.
\end{proof}

\end{document}